\documentclass{amsart}
\usepackage{times,amsfonts,amsmath,amstext,amsbsy,amssymb,amsopn,amsthm,upref,eucal,amscd}
\usepackage[T1]{fontenc}
\usepackage{color}
\usepackage[dvips]{graphics}

\numberwithin{equation}{section}

\newtheorem{theorem}{Theorem}
\newtheorem{Theorem}[theorem]{Theorem}
\newtheorem*{NNTheorem}{Theorem}

\newtheorem*{NNConjecture}{Conjecture}

\newtheorem{lemma}{Lemma}[section]
\newtheorem{Lemma}[lemma]{Lemma}

\newtheorem{corollary}{Corollary}[section]
\newtheorem{Corollary}[corollary]{Corollary}

\newtheorem{proposition}{Proposition}[section]
\newtheorem{Proposition}[proposition]{Proposition}

\newtheorem{Question}{Question}

\theoremstyle{definition}

\newtheorem{Convention}{Convention}

\theoremstyle{remark}

\newtheorem{remark}{Remark}[section]
\newtheorem{Remark}[remark]{Remark}
\newtheorem*{NNRemark}{Remark}

\renewcommand{\Im}{\operatorname{Im}}
\renewcommand{\Re}{\operatorname{Re}}
\renewcommand{\epsilon}{\varepsilon}
\renewcommand{\phi}{\varphi}

\newcommand{\field}[1]{\mathbb{#1}}

\newcommand{\R}{\field{R}}

\newcommand{\C}{\field{C}}
\newcommand{\Z}{\field{Z}}

\newcommand{\Id}{\operatorname{Id}}
\newcommand{\id}{\operatorname{id}}

\newcommand{\cE}{{\mathcal E}}

\newcommand{\cH}{{\mathcal H}}
\newcommand{\cK}{{\mathcal K}}
\newcommand{\cL}{{\mathcal L}}
\newcommand{\cM}{{\mathcal M}}

\newcommand{\cQ}{{\mathcal Q}}

\newcommand{\cT}{{\mathcal T}}
\newcommand{\cZ}{{\mathcal Z}}

\newcommand{\noz}{n}

\newcommand{\PcL}{\operatorname{P}\hspace*{-2pt}\cL}

\newcommand{\SLZ}{\operatorname{SL}(2,{\mathbb Z})}
\newcommand{\PSLZ}{\operatorname{PSL}(2,{\mathbb Z})}
\newcommand{\SL}{\operatorname{SL}(2,{\mathbb R})}

\newcommand{\PSO}{\operatorname{PSO}(2,{\mathbb R})}
\newcommand{\PSL}{\operatorname{PSL}(2,{\mathbb R})}
\newcommand{\GL}{\operatorname{GL}_+(2,{\mathbb R})}
\newcommand{\Upq}{\operatorname{U}(p,q)}

\newcommand{\CP}{{\mathbb C}\!\operatorname{P}^1}

\newcommand{\card}{\operatorname{card}}

\newcommand{\Area}{\operatorname{Area}}

\newcommand{\const}{\mathit{const}}

\definecolor{mygreen}{cmyk}{0.75,0,.75,0.75}

\newlength{\halfbls}

\begin{document}
\title{Zero Lyapunov exponents of the Hodge bundle}

\author{Giovanni Forni}
\address{Giovanni Forni: Department of Mathematics, University of Maryland, College Park, MD 20742-4015, USA}
\email{gforni@math.umd.edu.}
\author{Carlos Matheus}
\address{Carlos Matheus: CNRS, LAGA, Institut Galil\'ee, Universit\'e Paris 13, 99, Avenue Jean-Baptiste
Cl\'ement, 93430, Villetaneuse, France}
\email{matheus@impa.br.}
\urladdr{http://www.impa.br/$\sim$cmateus}
\author{Anton Zorich}
\address{Anton Zorich:
Institut de math\'ematiques de Jussieu
and Institut universitaire de France,
Universit\'e Paris 7, Paris, France}
\email{zorich@math.jussieu.fr}

\date{April 18, 2014}

\begin{abstract}
By the results of G.~Forni and of R.~Trevi\~no, the Lyapunov spectrum
of  the  Hodge  bundle  over  the  Teichm\"uller geodesic flow on the
strata  of  Abelian  and  of quadratic differentials does not contain
zeroes  even though for certain invariant submanifolds zero exponents
are  present  in  the  Lyapunov  spectrum.  In  all  previously known
examples,  the  zero  exponents  correspond to those $\PSL$-invariant
subbundles  of  the  real Hodge bundle for which the monodromy of the
Gauss---Manin  connection  acts by isometries of the Hodge metric. We
present  an  example  of an arithmetic Teichm\"uller curve, for which
the   real  Hodge  bundle  does  not  contain  any  $\PSL$-invariant,
subbundles,  and  nevertheless  its  spectrum of Lyapunov
exponents   contains  zeroes.  We  describe  the  mechanism  of  this
phenomenon;  it covers the previously known situation as a particular
case.  Conjecturally,  this is the only way zero exponents can appear
in the Lyapunov spectrum of the Hodge bundle for any $\PSL$-invariant
probability measure.
\end{abstract}

\maketitle

\vspace*{-0.5cm}

\setcounter{tocdepth}{2}
\tableofcontents


\section{Introduction}

A  complex  structure  on  the  Riemann  surface  $X$  of  genus  $g$
determines  a  complex  $g$-dimensional  space of holomorphic 1-forms
$\Omega(X)$ on $X$, and the Hodge decomposition
$$
H^1(X;\C{}) = H^{1,0}(X)\oplus H^{0,1}(X) \simeq
\Omega(X)\oplus\bar \Omega(X)\ .
$$
The pseudo-Hermitian intersection form
\begin{equation}
\label{eq:intersection:form}
\langle\omega_1,\omega_2\rangle:=\frac{i}{2}
\int_{X} \omega_1\wedge  \overline{\omega_2}\qquad\qquad\qquad
\end{equation}
is   positive-definite   on  $H^{1,0}(X)$  and  negative-definite  on
$H^{0,1}(X)$.

For any linear subspace $V\subset H^1(X,\C{})$ define its holomorphic
and anti-holomorphic parts respectively as
$$
V^{1,0}:=V\cap H^{1,0}(X) \qquad\text{and}\qquad
V^{0,1}:=V\cap H^{0,1}(X)\,.
$$

A subspace $V$ of the complex cohomology which decomposes as a direct
sum  of  its  holomorphic  and  anti-homolomorphic  parts,  that  is,
$V=V^{1,0}\oplus  V^{0,1}$,  will  be  called a \emph{split} subspace
(the  case  when  one  of  the  summands  is  null is not excluded: a
subspace    $V$    which    coincides   with   its   holomorphic   or
anti-homolomorphic   part  is  also  considered  as  \textit{split}).
Clearly,   the   restriction   to  any  split  subspace  $V$  of  the
pseudo-Hermitian   form  of  formula~\eqref{eq:intersection:form}  is
non-degenerate. Note that the converse is, in general, false.

The complex Hodge bundle $H^1_\C$ is the bundle over the moduli space
$\cM_g$  of  Riemann  surfaces  with  fiber  the  complex  cohomology
$H^1(X,\C)$  at any Riemann surface $X$. The complex Hodge bundle can
be   pulled  back  to  the  moduli  space  of  Abelian  or  quadratic
differentials   under  the  natural  projections  $\cH_g\to\cM_g$  or
$\cQ_g\to\cM_g$  respectively.  A  subbundle $V$ of the complex Hodge
bundle  is  called  a \emph{split} subbundle if all of its fibers are
split  subspaces or, in other terms, if it decomposes as a direct sum
of its holomorphic and anti-holomorphic parts.

Let  $\cL_1$  be  an  orbifold  in  some stratum of unit area Abelian
differentials   (respectively,   in  some  stratum  of  unit  area
meromorphic quadratic differentials with at most simple poles).
Throughout   this   paper   we   say   that   such   an  orbifold  is
$\SL$-\textit{invariant}  (respectively, $\PSL$-\textit{invariant}) if it is
the support of a Borel probability measure, invariant with respect to
the  natural  action  of  the group $\SL$ (respectively, of the group
$\PSL$)  and ergodic with respect to the Teichm\"uller geodesic flow.
The  action  of $\SL$ (respectively, of $\PSL$) on $\cL_1$ lifts to a
cocycle on the complex Hodge bundle $H^1_\C$ over $\cL_1$ by parallel
transport  of  cohomology  classes  with  respect to the Gauss--Manin
connection.  This  cocycle  is  called the complex Kontsevich--Zorich
cocycle.

It   follows   from   this   definition   that  the  pseudo-Hermitian
intersection     form     is     $\SL$-equivariant     (respectively,
$\PSL$-equivariant) under the complex Kontsevich--Zorich cocycle. The
complex  Kontsevich--Zorich cocycle has a well-defined restriction to
the  real  Hodge  bundle $H^1_\R$ (the real part of the complex Hodge
bundle), called simply the Kontsevich--Zorich cocycle.

By the results H.~Masur~\cite{Masur} and of W.~Veech~\cite{Veech}, the
Teichm\"uller geodesic flow is ergodic on all connected components of
all   strata   in   the   moduli  spaces  of  Abelian differentials
and in the moduli spaces of meromorphic quadratic
differentials with at most simple poles  with   respect   to   the   unique $\SL$-invariant
(respectively, $\PSL$-invariant),  absolutely continuous,  finite  measure. By the
further    results    of    G.~Forni~\cite{Forni:positive}   and   of
R.~Trevi\~no~\cite{Trevino}, it  is  known  that  the  action  of the
Teichm\"uller  geodesic flow on the real or complex Hodge bundle over
such  $\SL$-invariant  (respectively,  $\PSL$-invariant) orbifolds
has only non-zero Lyapunov exponents.

In this paper we continue our investigation on the occurrence of zero
Lyapunov   exponents  for  special  $\PSL$-invariant  orbifolds  (see
\cite{Forni:Matheus:Zorich_1}    ,    \cite{Forni:Matheus:Zorich_2}).
Previous examples of $\SL$-invariant (respectively, $\PSL$-invariant)
measures  with  zero exponents in the Lyapunov spectrum were found in
the class of cyclic covers over $\CP$ branched exactly at four points
(see~\cite{Bouw:Moeller},      \cite{Eskin:Kontsevich:Zorich:cyclic},
\cite{ForniSurvey},         \cite{Forni:Matheus:Zorich_1}         and
\cite{Forni:Matheus:Zorich_2}).  In all of those examples the neutral
Oseledets  subbundle  (that  is,  the  subbundle of the zero Lyapunov
exponent  in the Oseledets decomposition) is a smooth $\SL$-invariant
(respectively, $\PSL$-invariant) split subbundle.

Our  main  contribution  in  this  paper  is the analysis of a cocycle   acting   on   the
complex  Hodge  bundle  over  a  certain $\PSL$-invariant   orbifold   (which   projects  onto
an  arithmetic Teichm\"uller  curve)  in  the  moduli space of holomorphic quadratic
differentials in genus four. This particular example was inspired by   the   work   of  C.~McMullen
on  the  Hodge  theory  of general cyclic  covers~\cite{McMullen}.
It is the first explicit example of a cocycle with the Lyapunov spectrum containing zero
exponents such that the neutral Oseledets subbundle, which is by definition flow-invariant,
is nevertheless \textit{not} $\PSL$-invariant. In other words, the neutral subbundle
in this example \textit{is not} a pullback of a flat  subbundle  of  the  Hodge
bundle  over the corresponding Teichm\"uller curve.

In  fact, the zero exponents in this new example, as well as those in
all  previously  known  ones,  can  be  explained  by a simple common
mechanism.  Conjecturally  such a mechanism is completely general and
accounts  for  all zero exponents with respect to any $\SL$-invariant
(respectively,  $\PSL$-invariant)  probability  measure on the moduli
spaces  of Abelian (respectively, quadratic) differentials. It can be
outlined  as  follows.  We conjecture that a semisimplicity property
holds  for  the  complex  Hodge  bundle  in  the  spirit  of  Deligne
Semisimplicity Theorem. Namely, we conjecture that the restriction of
the  complex  Hodge bundle to any $\SL$-invariant (respectively,
$\PSL$-invariant)  orbifold  as  above  splits  into  a direct sum of
irreducible      $\SL$-invariant      (respectively,      irreducible
$\PSL$-invariant), continuous, split subbundles.\footnote{This 
conjecture has been recently proved by S.~Filip \cite{Fil13b}. Semisimplicity of the Kontsevich--Zorich cocycle on the {\it real} Hodge bundle had been proved earlier by Avila, Eskin and M\"oller (see Theorem 1.5 in \cite{Avila:Eskin:Moeller}) after a weaker semisimplicity result, establishing semisimplicity of the algebraic hulls of the cocycle, was proved by Eskin and Mirzahani (see \cite{Eskin:Mirzakhani}, Appendix A, also quoted as Theorem 2.1 in \cite{Avila:Eskin:Moeller}).}

The  \textit{continuous} vector subbundles in the known examples are,
actually,  smooth  (even  analytic,  or holomorphic). However, in the
context of this paper it is important to distinguish subbundles which
are only \textit{measurable} and those which are \textit{continuous}.
To    stress   this   dichotomy   in the general case we  shall   always   speak
about~\textit{continuous}  subbundles,  even  when  we  know  that they are
smooth (analytic, holomorphic). In particular, a $\SL$-invariant
(respectively, $\PSL$-invariant) subbundle of the Hodge bundle is
called \emph{irreducible} if it has no non-trivial \emph{continuous}
 $\SL$-invariant  (respectively, $\PSL$-invariant) subbundle.
 In the special case of subbundles defined over suborbifolds which
 project onto Teichm\"uller curves \textit{all} $\SL$-invariant (respectively,
 $\PSL$-invariant) subbundles are continuous, in fact smooth, since
 by definition the action of the group on the suborbifold is transitive.

We describe this splitting in our example. In fact, it was
observed  by  M.~M\"oller  (see  Theorem  2.1 in~\cite{Moeller}) that
whenever  the  projection  of  the  invariant orbifold $\cL_1$ to the
moduli space $\cM_g$ is a Teichm\"uller curve (as in our example) the
Deligne    Semisimplicity   Theorem~\cite{Deligne:87}   implies   the
existence  and  uniqueness  of the above-mentioned decomposition. The
action of the group $\SL$ (respectively, $\PSL$) on each irreducible,
invariant  split  subbundle  of the complex Hodge bundle is a cocycle
with  values  in  the group $U(p,q)$ of pseudo-unitary matrices, that
is,  matrices preserving a quadratic form of signature $(p,q)$. It is
a   general   result,   very   likely  known  to  experts,  that  any
$U(p,q)$-cocycle   has   at  least  $\vert  p-q\vert$  zero  Lyapunov
exponents  (we  include  a proof of this simple fundamental result in
Appendix~\ref{a:Lyapunov:spectrum:of:pseudo-unitary:cocycles}).

In  the  very  special  case of cyclic covers branched at four points,
considered   in~\cite{Bouw:Moeller},  \cite{Eskin:Kontsevich:Zorich},
\cite{Forni:Matheus:Zorich_1},   \cite{Forni:Matheus:Zorich_2},  only
pseudo-unitary irreducible cocycles of type $(0,2)$, $(2,0)$ $(0,1)$,
$(1,0)$,  and  $(1,1)$  arise.  In  the first four cases the Lyapunov
spectrum  of  the  corresponding  invariant  irreducible component is
null,  while  in the fifth case there is a symmetric pair of non-zero
exponents.  The  examples  which  we  present  in  this  paper  are
suborbifolds of the locus of cyclic covers branched at six points.
In this case we have a decomposition  into  two   (complex conjugate)
continuous components  of  type  $(3,1)$  and  $(1,3)$. It follows that the zero
exponent  has multiplicity at least $2$ in each  component
(which  is  of  complex  dimension  $4$). We prove that, in fact, the
multiplicity   of  the  zero  exponent  is  exactly  $2$.  Our main example
is a suborbifold which projects onto a certain arithmetic Teichm\"uller
curve. In this case we prove that the above-mentioned decomposition
is in fact \textit{irreducible}. The irreducibility   of   the   components
implies   that  the  complex two-dimensional  neutral  Oseledets subbundles
of  both components  cannot  be $\PSL$-invariant.
For general suborbifolds of our locus of cyclic covers branched at six points,
it follows from  results of~\cite{Forni:Matheus:Zorich_2}  (see in particular
Theorem 8 in that paper) that whenever the complex two-dimensional
neutral  Oseledets subbundles  of  both  components  are  $\PSL$-invariant, then they
are also continuous, in fact smooth. It follows then from our irreducibility
result that the neutral  Oseledets subbundles are not $\PSL$-invariant
on the full locus of cyclic covers branched at six points, which contains our
main example. Moreover, recent work of Avila, Matheus  and  Yoccoz
\cite{Avila:Matheus:Yoccoz}  suggests that the neutral Oseledets subbundles
are also not continuous  there.

As in all known examples, our cocycle is non-degenerate, in the sense
that the multiplicity of the zero exponent is exactly equal to $\vert
p-q\vert$.  Conjecturally,  all  cocycles  arising from the action of
$\SL$   (respectively,   $\PSL$)  on  the  moduli  space  of  Abelian
(respectively,  quadratic)  differentials  are  non-degenerate in the
above  sense and are simple, in the sense that all non-zero exponents
are   simple  in  every  irreducible  $\SL$-invariant  (respectively,
$\PSL$-invariant) continuous component of the complex Hodge bundle.
(The  simplicity of the Lyapunov spectrum for the canonical invariant
measure  on  the  connected  components  of  the  strata  of  Abelian
differentials is proved in~\cite{Avila:Viana}; an analogous statement
for  the  strata  of  \textit{quadratic} differentials for the moment
remains conjectural.)


Note   that   currently   one   cannot   naively  apply  the  Deligne
Semisimplicity Theorem to construct an $\SL$-invariant (respectively,
a  $\PSL$-invariant)  splitting  of  the  Hodge bundle over a general
invariant  suborbifold  $\cL$.  Even  though  by  recent  results  of
A.~Eskin and M.~Mirzakhani~\cite{Eskin:Mirzakhani}   each   such
invariant  suborbifold  is an affine subspace in the ambient stratum, it
is not known whether it is a quasiprojective variety or not\footnote{It has been proved recently  by S.~Filip \cite{Fil13a} that all invariant suborbifolds are quasiprojective varieties.}.

Note also that the conjectural decomposition of the complex Hodge bundle into
irreducible $\SL$-invariant  (respectively, $\PSL$-invariant) components might be
finer  than  the decomposition coming from the Deligne Semisimplicity
Theorem.  The  summands in the first (hypothetical) decomposition are
irreducible  only  with  respect  to the action by parallel transport
\textit{along   the   $\GL$-orbits}   in   $\cL$,   or  equivalently,
\textit{along  the  leaves  of  the foliation by \mbox{Teichm\"uller}
discs} in the projectivization $\PcL$, while the decomposition of the
Hodge  bundle  provided  by  the  Deligne  Semisimplicity  Theorem is
invariant with respect to the action by parallel transport \textit{of
the  full  fundamental group} of $\cL$. For example, the Hodge bundle
$H^1_{\C{}}$ over the moduli space $\cH_g$ of Abelian differentials
splits  into  a  direct sum of $(1,1)$-tautological subbundle and its
$(g-1,g-1)$-orthogonal complement. This splitting is $\GL$-invariant,
but  it  is by no means invariant under the parallel transport in the
directions transversal to the orbits of $\GL$. The only case when the
two  splittings  certainly  coincide corresponds to the Teichm\"uller
curves,  when the entire orbifold $\cL$ is represented by a single
orbit of $\GL$.

We  conclude  the  introduction  by  formulating  an  outline  of the
principal  conjectures.  

\begin{NNConjecture}
Let $\cL_1$ be a suborbifold in the moduli space of unit area Abelian
differentials  or  in  the  moduli  space  of  unit  area meromorphic
quadratic  differentials  with  at  most  simple  poles. Suppose that
$\cL_1$  is  endowed with a Borel probability measure, invariant with
respect  to  the  natural action of the group $\SL$ (respectively, of
the  group  $\PSL$)  and  ergodic  with  respect to the Teichm\"uller
geodesic  flow.  The  Lyapunov  spectrum  of the complex Hogde bundle
$H^1_{\C{}}$  over the Teichm\"uller geodesic flow on $\cL_1$ has the
following properties.

(I.)  Let  $r$  be  the  total number of zero entries in the Lyapunov
spectrum.   By  passing,  if  necessary,  to  an  appropriate  finite
(possibly  ramified)  cover  $\hat\cL_1$ of $\cL_1$ one can decompose
the vector bundle induced from the Hodge bundle over $\hat\cL_1$ into
a   direct   sum   of   irreducible   $\SL$-invariant  (respectively,
irreducible  $\PSL$-invariant) continuous split subbundles.\footnote{This part of the conjecture in a more precise form has been recently established by S.~Filip \cite{Fil13b}.}
Denote by $(p_i,q_i)$  the signature of the restriction of the pseudo-Hermitian
intersection  form to the corresponding split subbundle. Then $\sum_i
|p_i-q_i|=r$.

(II.)  By passing,  if  necessary,  to  an  appropriate finite (possibly
ramified)  cover $\hat\cL_1$ of $\cL_1$ one can decompose
the vector bundle  induced  from the Hodge  bundle  over  $\hat\cL_1$ into a
direct sum of irreducible $\SL$-invariant (respectively,  irreducible $\PSL$-invariant)
continuous split  subbundles, such that the
nonzero part of the Lyapunov spectrum of each summand is simple.
\end{NNConjecture}


\subsection{Statement of the results}
\label{ss:A:concrete:example}

Let us consider  a  flat  surface  $S$  glued  from  six  unit squares as in
Figure~\ref{fig:oneline:6}.  It  is easy to see that this surface has
genus  zero,  and that the flat metric has five conical singularities
with  the  cone angle $\pi$ and one conical singularity with the cone
angle  $3\pi$. Thus, the quadratic differential representing the flat
surface  $S$ belongs to the stratum $\cQ(1,-1^5)$ in the moduli space
of meromorphic quadratic differentials.

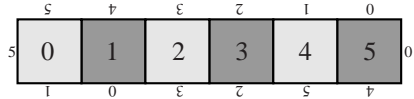
\begin{figure}[htb]
   %
   %
\includegraphics{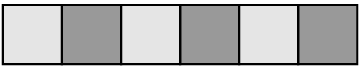}

\begin{picture}(0,0)(55,25) 
\put(-22.5,0){0}
\put(2,0){1}
\put(27,0){2}
\put(51,0){3}
\put(75,0){4}
\put(99.5,0){5}
   %
\put(-21.5,21){\rotatebox{180}{\tiny 5}}
\put(3,21){\rotatebox{180}{\tiny 4}}
\put(28.5,21){\rotatebox{180}{\tiny 3}}
\put(51,21){\rotatebox{180}{\tiny 2}}
\put(76.5,21){\rotatebox{180}{\tiny 1}}
\put(100.5,21){\rotatebox{180}{\tiny 0}}
\put(-35,1){\tiny 5}
\put(115,1){\tiny 0}
\put(-21.5,-10){\rotatebox{180}{\tiny 1}}
\put(3,-10){\rotatebox{180}{\tiny 0}}
\put(28.5,-10){\rotatebox{180}{\tiny 3}}
\put(51,-10){\rotatebox{180}{\tiny 2}}
\put(76.5,-10){\rotatebox{180}{\tiny 5}}
\put(100.5,-10){\rotatebox{180}{\tiny 4}}
\end{picture}

\vspace{40pt} 
\caption{
\label{fig:oneline:6}
Basic square-tiled surface $S$ in $\cQ(1,-1^5)$.
}
\end{figure}

The equation
\begin{equation}
\label{eq:cyclic:cover:equation}
w^3=(z-z_1)\cdot\dots\cdot(z-z_6)
\end{equation}
defines  a Riemann surface $\hat X$ of genus four, and a triple cover
$p:\hat  X\to  \CP$,  $p(w,z)=z$.  The  cover  $p$ is ramified at the
points  $z_1,\dots,z_6$  of $\CP$ and at no other points. By placing the
ramification  points  $z_1,\dots,z_6$  at  the single zero and at the
five  poles  of the flat surface $S$ as in Figure~\ref{fig:oneline:6}
we  induce  on $\hat X$ a flat structure, thus getting a square-tiled
surface  $\hat  S$. It is immediate to check that $\hat S$ belongs to
the  stratum  $\cQ(7,1^5)$  of holomorphic quadratic differentials in
genus four.

Let us consider    the    corresponding   arithmetic   Teichm\"uller   curve
$\hat\cT\subset\cM_4$  and  the  Hodge  bundle over it.  The following
theorem, announced in \cite{Forni:Matheus:Zorich_2}, Appendix B,
summarizes the statement of Proposition~\ref{prop:Lyapunov:spectrum}
and of Corollary~\ref{cor:no:R:subspaces}.

\begin{Theorem}
\label{th:R}
The  Lyapunov  spectrum  of  the  real Hodge bundle $H^1_{\R{}}$ with
respect  to  the  geodesic flow on the arithmetic Teichm\"uller curve
$\hat\cT$ is
$$
\left\{\frac{4}{9},\frac{4}{9},0,0,0,0,-\frac{4}{9},-\frac{4}{9}\right\}\,.
$$
The real  Hodge bundle $H^1_{\R{}}$ over $\hat\cT$ does not have any
nontrivial $\PSL$-invariant subbundles.
\end{Theorem}

It follows from the above theorem that  the neutral Oseledets subbundle  $E_0$ over
$\hat \cT$ is not $\PSL$-invariant.  It  seems  likely  that  it is  also not
continuous.

Note  that  the  cyclic  group $\Z/3\Z$ acts naturally on any Riemann
surface  $\hat  X$  as  in~\eqref{eq:cyclic:cover:equation}  by  deck
transformations of the triple cover $p:\hat X\to \CP$. In coordinates
this action is defined as
\begin{equation}
\label{eq:T}
T: (z,w)\mapsto (z,\zeta w)\,,
\end{equation}
where $\zeta=e^{2\pi i/3}$. Thus, the complex Hodge bundle $H^1_\C{}$
splits over  the locus of cyclic covers~\eqref{eq:cyclic:cover:equation} into a
direct sum of two flat subbundles (that is, vector subbundles invariant under
the parallel transport with respect to the Gauss--Manin connection):
\begin{equation}
\label{eq:Ezeta:oplus:Ezeta2}
H^1_\C{}=\cE(\zeta)\oplus \cE(\zeta^2)\,,
\end{equation}
where $\cE(\zeta)$, $\cE(\zeta^2)$ are the eigenspaces of the induced
action of the generator $T$ of the group of deck transformations.

The above Theorem~\ref{th:R} has an equivalent formulation in terms of the complex
Hodge bundle, which summarizes the statements of Proposition~\ref{prop:Lyapunov:spectrum}
and of Proposition~\ref{prop:only:E} below.

\begin{Theorem}
\label{th:C}
The  complex   Hodge   bundle   $H^1_{\C{}}$   over the  arithmetic
Teichm\"uller   curve   $\hat\cT$   does   not   have   any   non-trivial
$\PSL$-invariant complex subbundles other than $\cE(\zeta)$   and
$\cE(\zeta^2)$. The   Lyapunov   spectrum  of  each  of  the  subbundles  $\cE(\zeta)$,
$\cE(\zeta^2)$ with respect to the geodesic flow on $\hat\cT$ is
$$
\left\{\frac{4}{9},0,0,-\frac{4}{9}\right\}\,.
$$
\end{Theorem}

Actually,  there  is nothing special about the arithmetic Teichm\"uller
curve    $\cT\subset\cQ(1,-1^5)$   considered   above.   By taking   any
$\PSL$-invariant   suborbifold   $\cL\subseteq\cQ(1,-1^5)$   we   can
construct  a  cyclic  cover~\eqref{eq:cyclic:cover:equation} for each
flat  surface $S$ in $\cL$ placing the six ramification points at the
zero  and at the five poles of the quadratic differential. We get the
induced quadratic differential on the resulting cyclic cover. In this
way we get a $\PSL$-invariant suborbifold $\hat\cL\subseteq\cQ(7,1^5)$.
By  construction,  it  has  the  same properties  as  $\cL$, namely, it is
endowed with a Borel probability measure,  invariant  with  respect to the
natural action of the group $\PSL$  and  ergodic with respect to the Teichm\"uller
geodesic flow. (See  the end of Section~\ref{s:Hodge:bundle} for a generalization
of this construction.) Let  $\hat\cZ$ denote the suborbifold of all
cyclic covers branched at six points, namely, the suborbifold
obtained by the above construction in the case $\cL=\cQ(1,-1^5)$
(see also \cite{Forni:Matheus:Zorich_2}, Appendix B).

\begin{Theorem}
\label{th:L}
The  complex  Hodge  bundle  $H^1_{\C{}}$ over the invariant orbifold
$\hat\cL$  decomposes  into  the  direct  sum of two $\PSL$-invariant,
continuous  split subbundles   $H^1_\C{}=\cE(\zeta)\oplus   \cE(\zeta^2)$   of
signatures $(1,3)$ and $(3,1)$ respectively.
The  Lyapunov  spectrum  of  each  of  the  subbundles  $\cE(\zeta)$,
$\cE(\zeta^2)$  with  respect  to  the Teichm\"uller geodesic flow on
$\hat\cL$ is
$$
\left\{\frac{4}{9},0,0,-\frac{4}{9}\right\}\,.
$$
\end{Theorem}
The  only  difference  between the more general Theorem~\ref{th:L} and the
previous one,  treating  the particular case $\hat\cL=\hat\cT$, is that now we
do   not   claim   irreducibility  of  the  subbundles  $\cE(\zeta)$,
$\cE(\zeta^2)$  for  all  invariant  orbifolds  $\hat\cL$  as  above.
Theorem~\ref{th:L} follows from Theorem~\ref{th:spec:of:unitary:cocycle}
below and from Proposition~\ref{prop:5:2}.

Note  that  the  stratum $\cQ(1,-1^5)$ is naturally isomorphic to the
stratum      $\cH(2)$.      Thus,      the      classification     of
C.~McMullen~\cite{McMullen:genus2}   describes  all  $\PSL$-invariant
suborbifolds  in  $\cQ(1,-1^5)$:  they are represented by an explicit
infinite   series   of   suborbifolds   corresponding  to  arithmetic
Teichm\"uller  curves, by an explicit infinite series of suborbifolds
corresponding  to  non-arithmetic  Teichm\"uller  curves  and  by the
entire  stratum.  By  the way, note that the subbundles $\cE(\zeta)$,
$\cE(\zeta^2)$  of  the  Hodge  bundle over the invariant suborbifold
$\hat {\mathcal Z}\subset\cQ(7,1^5)$ (induced from the entire stratum
$\cQ(1,-1^5)$)   are   irreducible:   indeed,   this   follows   from
Theorem~\ref{th:C}    as    $\hat\cT\subset   \hat   {\mathcal   Z}$.
Theorem~\ref{th:L} confirms the Conjecture stated in the introduction
for   all   resulting   $\PSL$-invariant   suborbifolds  (up  to  the
irreducibility  of  the  decomposition  in  the  case of suborbifolds
$\hat\cL\neq \hat\cT, \hat{\mathcal Z}$).

As proved in~\cite{Forni:Matheus:Zorich_2}, Appendix B, Theorem 8, from Theorem~\ref{th:R}
and Theorem~\ref{th:L} above and from Theorem 3 of~\cite{Forni:Matheus:Zorich_2}
 we can derive the following result.

\begin{Corollary}
If the neutral  Oseledets subbundle  $E_0$ of the Kontsevich--Zorich cocycle over the invariant suborbifold $\hat\cL$ is $\PSL$-invariant, then it is continuous, in fact smooth. In particular,  since
$\hat\cT \subset \hat\cZ$, the subbundle  $E_0$ is not  almost everywhere $\PSL$-invariant over
the suborbifold $\hat  {\mathcal  Z}$ endowed with the canonical measure.
\end{Corollary}
A.~Avila, C.~Matheus and J.-C.~Yoccoz~\cite{Avila:Matheus:Yoccoz}
have recently  proved  that  indeed  $E_0$  is  also not  continuous  over  the
suborbifold   $\hat  \cZ$.

\smallskip

Note  that  a  Riemann  surface $X$, or a pair given by a Riemann surface
and an Abelian   or quadratic   differential,   might   have   a   nontrivial
automorphism  group.  This  automorphism  group is always finite. The
fiber of the Hodge bundle $H^1_{\C}$ over the corresponding point $x$
of  the moduli space is defined as the quotient of $H^1(X,\C)$ by the
corresponding  finite group $G_x$ of induced linear automorphisms. In other words,
the bundle $H^1_{\C}$ is an \textit{orbifold vector bundle}, in the sense
that it is a \textit{fibered} space $H$ over a base $M$ such that  the  fiber  $H_x$
over  any $x\in M$ is the quotient $H_x =V_x/G_x$  of  a  vector space $V_x$
over a finite subgroup $G_x$ of the group $\textrm{Aut} (V_x)$ of linear
automorphisms of $V_x$.


Since the Hodge bundle $H^1_{\C}$ is an orbifold vector bundle, the complex
Kontsevich-Zorich cocycle is an example of  an \textit{orbifold linear cocycle} on
an orbifold vector bundle $H$ over a flow $T_t$ on $M$, i.e., a flow $F_t$ on $H$
such that the restrictions $F_t : H_x \to H_{T_tx}$ are well-defined and are projections
of linear maps $\hat F_t : V_x \to V_{T_t x}$. Note that such linear maps are only
defined up to precomposition with the action of elements of $G_x$  on $V_x$ and
postcomposition with the action of elements of $G_{T_tx}$  on $V_{T_tx}$.

In  this paper we always work within the locus of cyclic covers.
For  any generic cyclic cover $x$ as in~\eqref{eq:cyclic:cover:equation} the
automorphism group  is  isomorphic to the  cyclic  group $\Z/3\Z$.
The induced  action  on the subspaces $\cE_x(\zeta)$ and $\cE_x(\zeta^2)$
is  particularly  simple:  the induced group $G_x$  of linear automorphisms
acts  by  multiplication by the complex numbers $\zeta^k$ for $k=0,1,2$
(we recall that $\zeta= e^{2 \pi i/3}$).
This implies that any complex vector subspace of $\cE_x(\zeta)$  or  $\cE_x(\zeta^2)$
is invariant. The elements of  the  monodromy  representations  of  the bundles
$\cE(\zeta)$ and $\cE(\zeta^2)$, hence in particular the  restrictions of the
Kontsevich--Zorich cocycle to those bundles,  are thus given by  linear maps
defined  only up to composition with  the maps $\zeta^k \Id$, that is, up
to multiplication by $\zeta^k$,  for $k=0,1,2$,


\subsection{Lyapunov spectrum of pseudo-unitary cocycles}

Consider  an  invertible  transformation  (or  a  flow)  ergodic with
respect  to  a  finite measure. Let $U$ be a $\log$-integrable cocycle over
this  transformation  (flow)  with  values  in  the  group  $\Upq$ of
pseudo-unitary    matrices.   The   Oseledets   Theorem   (i.e.   the
multiplicative  ergodic  theorem) can be applied to complex cocycles.
Denote  by $\lambda_1,\dots,\lambda_{p+q}$ the corresponding Lyapunov
spectrum.


\begin{Theorem}
\label{th:spec:of:unitary:cocycle}
The  Lyapunov  spectrum  of a pseudo-unitary cocycle $U$ is symmetric
with  respect  to  the  sign  change  and  has  at least $|p-q|$ zero
exponents.

In  other  words, the Lyapunov spectrum of an integrable cocycle with
the  values  in  the  group $\Upq$ of pseudo-unitary matrices has the
form
$$
\lambda_1\ge\dots\ge\lambda_r\ge 0 = \dots = 0
\ge -\lambda_r\ge \dots \ge -\lambda_1\,,
$$
where $r=\min(p,q)$. In particular, if $r=0$, the spectrum is null.
\end{Theorem}
This  theorem might  be known to experts, and, in any case, the proof
is  completely elementary. For the sake of completeness, it is given
in Appendix~\ref{a:Lyapunov:spectrum:of:pseudo-unitary:cocycles}.

\subsection{Outline of the proofs and plan of the paper}

We begin by recalling in Section~\ref{ss:Splitting:of:the:Hodge:bundle}
some basic properties of cyclic covers. In Section~\ref{ss:Construction:of:PSL:invariant:orbifolds}
we construct plenty  of  more  general  $\PSL$-invariant  orbifolds in loci of
cyclic covers.

By  applying  results  of C.~McMullen~\cite{McMullen}, we then show in
Section~\ref{ss:Splitting:of:the:Hodge:bundle} that
in the particular case of the    arithmetic    Teichm\"uller    disc   $\hat\cT$   defined   in
Section~\ref{ss:A:concrete:example},           the          splitting
$H^1_\C{}=\cE(\zeta)\oplus\cE(\zeta^2)$  of  the complex Hodge bundle
over   $\hat\cT$   decomposes  the  corresponding  cocycle  over  the
Teichm\"uller  geodesic  flow  on  $\hat\cT$  into  the direct sum of
complex         conjugate         $\operatorname{U}(3,1)$         and
$\operatorname{U}(1,3)$-cocycles.                                  By
Theorem~\ref{th:spec:of:unitary:cocycle}  the  Lyapunov  spectrum  of
each of the two cocycles has the form
$$
\{\lambda, 0, 0, -\lambda\}\,,
$$
with  nonnegative  $\lambda$.  Since  the  two  cocycles  are complex
conjugate,  their  Lyapunov  spectra  coincide.  Hence,  the Lyapunov
spectrum of real and complex Hodge bundles over $\hat\cT$ has the form
$$
\{\lambda, \lambda, 0, 0, 0, 0, -\lambda, -\lambda\}\,.
$$

To  compute $\lambda$ we construct in Section~\ref{ss:The:PSLZ:orbit}
the $\PSLZ$-orbit of the square-tiled surface $\hat S$. This orbit is
very small: it contains only two other square-tiled surfaces. Knowing
the cylinder decompositions of the resulting square-tiled surfaces in
the    $\PSLZ$-orbit    of    $\hat    S$,  we apply   a   formula
from~\cite{Eskin:Kontsevich:Zorich}  for  the  sum  of  the  positive
Lyapunov  exponents  of  the  Hodge  bundle  over  the  corresponding
arithmetic  Teichm\"uller  disc  $\hat\cT$  to get the explicit value
$\lambda=4/9$.      This      computation     is     performed     in
Section~\ref{ss:Spectrum:of:the:Lyapunov:exponents}.
(In   Section~\ref{s:non:varying}  we  present  an  alternative,  more
general, way to compute Lyapunov exponents in similar situations.)

In  Section~\ref{s:Irreducibility:of:the:Hodge:bundle}  we  check the
irreducibility  of  the  subbundles  $\cE(\zeta)$  and $\cE(\zeta^2)$
essentially    by    hands.
   Note that the monodromy representation of the subbundles $\cE(\zeta)$
and  $\cE(\zeta^2)$  factors through the action of the Veech group of
$\hat\cT$.  We encode the action of the group $\PSLZ$ on the orbit of
$\hat  S$ by a graph $\Gamma$ associating oriented edges to the basic
transformations
$$
h=
\begin{pmatrix}
1&1\\
0&1
\end{pmatrix}
\qquad
\text{ and }
\qquad
r=\left(
\begin{array}{rr}
0&-1\\
1&0
\end{array}\right)\,.
$$
The resulting graph $\Gamma$ is represented at
Figure~\ref{fig:PSL2Z:orbit}.  We choose a basis of homology on every
square-tiled  surface  in the $\PSLZ$-orbit of $\hat S$ and associate
to  every  oriented  edge  of  the  graph the corresponding monodromy
matrix.  Any  closed path on the graph defines the free homotopy type
of   the   corresponding  closed  path  on  the  Teichm\"uller  curve
$\hat\cT$.  The  monodromy  along  such  path  on  $\hat\cT$  can  be
calculated  as  the  product  of  matrices associated to edges of the
graph   in   the   order   following   the  path  on  the  graph.  In
Proposition~\ref{prop:Ezeta:does:not:have:subbundles}    we   construct
two explicit   closed   paths   and   show  that  the  induced  monodromy
transformations  cannot have common invariant subspaces. This implies
the irreducibility claims in Theorems~\ref{th:R} and~\ref{th:C}.
The evaluation of the monodromy representation is outlined in
Appendix~\ref{a:matrix:calculation}.

Following    a    suggestion   of   M.~M\"oller,   we   sketch   in
Lemma~\ref{lm:Ezeta:strong:irreducibility} in Section~\ref{ss:Zariski:closure} the computation of the Zariski closure of the monodromy
group  of  $\cE(\zeta)$. The details of this calculation are explained in Appendix~\ref{a:Zariski:closure}. Then, using Lemma~\ref{lm:Ezeta:strong:irreducibility}, we prove in Proposition~\ref{prop:strongly:irreducible} in Section~\ref{ss:strong_irreducibility} the
\textit{strong   irreducibility}\footnote{I.e.,   the
irreducibility  of  lifts  of  $\cE(\zeta)$ and $\cE(\zeta^2)$ to any
finite  (possibly ramified) cover of $\hat\cT$.} of $\cE(\zeta)$ and
of $\cE(\zeta^2)$.

   %
   %

In  Section~\ref{s:non:varying}  we  prove the non-varying phenomenon
for  certain  $\PSL$-invariant loci of cyclic covers. Namely, we show
that  the  sum  of  the  Lyapunov  exponents  is  the  same  for  any
$\PSL$-invariant  suborbifold in such loci.

Finally, in  Appendix~\ref{a:Lyapunov:spectrum:of:pseudo-unitary:cocycles}  we
discuss  some basic facts concerning linear algebra of pseudo-unitary
cocycles   and   prove  Theorem~\ref{th:spec:of:unitary:cocycle}.


\section{Hodge  bundle  over invariant suborbifolds in loci of cyclic
covers}
\label{s:Hodge:bundle}

\subsection{Splitting of the Hodge bundle over loci of cyclic covers}
\label{ss:Splitting:of:the:Hodge:bundle}
Consider  a  collection of $n$ pairwise-distinct points $z_i\in\C{}$.
The equation
\begin{equation}
\label{eq:cyclic:cover:general}
w^d=(z-z_1)\cdot\dots\cdot(z-z_n)
\end{equation}
defines  a  Riemann surface $\hat X$, and a cyclic cover $p:\hat X\to
\CP$, $p(w,z)=z$. Consider the
canonical   generator   $T$   of   the   group   $\Z/d\Z$   of   deck
transformations; let
$$
T^\ast: H^1(X;\C{})\to H^1(X;\C{})
$$
be      the      induced      action     in     cohomology.     Since
$(T^\ast)^d=\operatorname{Id}$, the eigenvalues of $T^\ast$ belong to
a       subset      of      $\{\zeta,\dots,\zeta^{d-1}\}$,      where
$\zeta=\exp\left(\dfrac{2\pi  i}{d}\right)$.  We  excluded  the  root
$\zeta^0=1$   since   any   cohomology  class  invariant  under  deck
transformations  would  be a pullback of a cohomology class on $\CP$,
and $H^1(\CP)=0$.

For $k=1,\dots,d-1$ denote
\begin{equation}
\label{eq:E:zeta:k}
\cE(\zeta^k):=\operatorname{Ker}(T^\ast-\zeta^k\operatorname{Id})
\subseteq H^1(X;\C{})\ .
\end{equation}

Denote
$$
\cE^{1,0}(\zeta^k):=\cE(\zeta^k)\cap H^{1,0}\quad
\text{ and }\quad
\cE^{0,1}(\zeta^k):=\cE(\zeta^k)\cap H^{0,1}\ .
$$
Since  a  generator $T$ of the group of deck transformations respects
the complex structure, it induces a linear map
$$
T^\ast: H^{1,0}(X)\to H^{1,0}(X)\,.
$$
This map preserves the pseudo-Hermitian form~\eqref{eq:intersection:form} on
$H^{1,0}(X)$.  This  implies  that $T^\ast$ is a unitary operator on
$H^{1,0}(X)$, and hence $H^{1,0}(X)$ admits a splitting into a direct
sum of eigenspaces of $T^\ast$,
\begin{equation}
\label{eq:H10:direct:sum:for:k}
H^{1,0}(X)=\bigoplus_{k=1}^{d-1}\cE^{1,0}(\zeta^k)\ .
\end{equation}
The  latter observation also implies that for any $k=1,\dots,d-1$ one
has   $\cE(\zeta^k)=\cE^{1,0}(\zeta^k)\oplus\cE^{0,1}(\zeta^k)$.  The
vector   bundle   $\cE^{1,0}(\zeta^k)$   over  the  locus  of  cyclic
covers~\eqref{eq:cyclic:cover:general}  is a holomorphic subbundle of
$H^1_{\C{}}$.

The decomposition
$$
H^1(X;\C{})=\oplus \cE(\zeta^k)\,,
$$
is  preserved by the Gauss---Manin connection, which implies that the
complex   Hodge   bundle   $H^1_{\C{}}$  over  the  locus  of  cyclic
covers~\eqref{eq:cyclic:cover:general}  splits  into  a direct sum of
the  subbundles  $\cE(\zeta^k)$  invariant with respect to the parallel
transport of the the Gauss---Manin connection.

\begin{NNTheorem}[C.~McMullen]
The signature of the intersection form on
$\cE(\zeta^{-k})$ is given by
\begin{equation}
\label{eq:signature}
(p,q)=\big([n(k/d)-1],[n(1-k/d)-1]\big)\,.
\end{equation}
In particular,
\begin{equation}
\label{eq:dim}
\dim\cE(\zeta^k)=
\begin{cases}
n-2&\text{ if }d\text{ divides }kn\,,\\
n-1&\text{otherwise}
\end{cases}\,.
\end{equation}
\end{NNTheorem}

By applying   these   general   results   to   the   particular   cyclic
cover~\eqref{eq:cyclic:cover:equation},       we       see       that
$H^1_\C{}=\cE(\zeta)\oplus\cE(\zeta^2)$,  where  the signature of the
intersection    form    on    $\cE(\zeta)$    is   $(3,1)$   and   on
$\cE(\zeta^2)=\overline{\cE(\zeta)}$ is $(1,3)$.

\medskip
\noindent\textbf{Bibliographical remarks.}
Cyclic covers over $\CP$ branched at four points were used by I.~Bouw
and  M.~M\"oller  in~\cite{Bouw:Moeller}  to  construct new series of
nonarithmetic \mbox{Teichm\"uller} curves. Similar cyclic covers were
independently   used   by  G.~Forni~\cite{ForniSurvey}  and  then  by
G.~Forni  and C.~Matheus~\cite{Forni:Matheus} to construct arithmetic
Teichm\"uller  curves  with  completely  degenerate  spectrum  of the
Lyapunov  exponents  of the Hodge bundle with respect to the geodesic
flow.  The  monodromy  of the Hodge bundle is explicitly described in
these  examples by C.~Matheus and J.-C.~Yoccoz~\cite{Matheus:Yoccoz}.
More  general arithmetic \mbox{Teichm\"uller} curves corresponding to
cyclic  covers  over  $\CP$  branched  at  four  points  are  studied
in~~\cite{Forni:Matheus:Zorich_1}.  The  Lyapunov spectrum of the Hodge
bundle  over  such  arithmetic  Teichm\"uller  curves  is  explicitly
computed   in~\cite{Eskin:Kontsevich:Zorich:cyclic}.   More  generally,
Abelian     covers     are     studied    in    this    context    by
A.~Wright~\cite{Wright}.   Our  consideration  of  cyclic  covers  as
in~\eqref{eq:cyclic:cover:general}   is  inspired  by  the  paper  of
C.~McMullen~\cite{McMullen},   where   he   studies   the   monodromy
representation  of the braid group in the Hodge bundle over the locus
of cyclic covers.

For  details  on geometry of cyclic covers see the original papers of
I.~Bouw~\cite{Bouw:thesis}        and~\cite{Bouw}        and       of
J.~K.~Koo~\cite{Koo},    as    well    as   the   recent   paper   of
A.~Elkin~\cite{Elkin} citing the first three references as a source.

\subsection{Construction      of      $\PSL$-invariant      orbifolds
in loci of cyclic covers}
\label{ss:Construction:of:PSL:invariant:orbifolds}
Suppose  for  simplicity  that $d$ divides $n$, where $n$ and $d$ are
the  integer  parameters in equation~\eqref{eq:cyclic:cover:general}.
The  reader  can easily  extend  the  considerations  below to the
remaining case.

Let   $\cL$   be  a  $\PSL$-invariant  suborbifold  in  some  stratum
$\cQ(m_1,\dots,m_k,-1^l)$   in   the   moduli   space   of  quadratic
differentials  with  at  most  simple  poles  on  $\CP$. For any such
invariant  orbifold  $\cL$  and for any couple of integers $(d,n)$ we
construct  a new $\PSL$-invariant suborbifold $\hat\cL$ such that the
Riemann  surfaces  underlying the flat surfaces from $\hat\cL$ belong
to  the  locus  of cyclic covers~\eqref{eq:cyclic:cover:general}. The
construction is performed as follows.

Let  $S=(\CP,q)\in\cL$.  In  the simplest case, when the total number
$k+l$  of  zeroes and poles of the meromorphic quadratic differential
$q$  on  $\CP$  coincides with the number $n$ of ramification points,
one  can  place  the points $z_1,\dots z_n$ exactly at the zeroes and
poles  of  the  corresponding  quadratic  differential  $q$. (Here we
assume   that   $d$   divides  $n$,  so  that  the  cyclic  cover  as
in~\eqref{eq:cyclic:cover:general}  is  not  ramified  at  infinity.)
Consider  the induced quadratic differential $p^\ast q$ on the cyclic
cover  $\hat  X$.  By applying  this  operation  to  every  flat surface
$S\in\cL$,  we  get  the promised orbifold $\hat\cL$.
Since by assumption $\cL$  is  $\PSL$-invariant,
the induced orbifold
$\hat\cL$  is  also $\PSL$-invariant,  and  in  the  simplest  case,  when
$k+l=n$, we get $\dim\hat\cL=\dim\cL$. In particular, starting with a
Teichm\"uller curve, we get a Teichm\"uller curve.

In  the  concrete example from Section~\ref{ss:A:concrete:example} we
start  with  an arithmetic Teichm\"uller curve $\cT$ corresponding to
the  stratum $\cQ(1,-1^5)$. Placing the points $z_1,\dots,z_6$ at the
single  zero  and at the five poles of each flat surface $S$ in $\cT$
we  get  an arithmetic Teichm\"uller curve $\hat\cT$ corresponding to
the  stratum  $\cQ(7,1^5)$. By construction, $\hat\cT$ belongs to the
locus of cyclic covers~\eqref{eq:cyclic:cover:equation}.

The latter construction can be naturally generalized to the case when
$k+l\neq n$.

When  $\PcL$ is a nonarithmetic Teichm\"uller curve, the construction
can be modified by placing the points $z_1,\dots,z_n$ at all possible
subcollections    of    $n$    distinct   \textit{periodic   points};
see~\cite{Gutkin:Hubert:Schmidt} for details.

The  construction  can  be  generalized  further.  Let  $\cL_1$  be a
$\PSL$-invariant        suborbifold       of       some       stratum
$\cQ_1(m_1,\dots,m_k,-1,\dots,-1)$   in  genus  zero.  Fix  a  subset
$\Sigma$     in     the     ordered     set    with    multiplicities
$\{m_1,\dots,m_k,-1,\dots,-1\}$;   let  $j$  be  the  cardinality  of
$\Sigma$.  For each  flat  surface $S=(\CP,q)$ in $\cL$, consider all
possible  cyclic  covers  as  in~\eqref{eq:cyclic:cover:general} such
that  the points $z_1,\dots,z_j$ run over all possible configurations
of the zeroes and poles corresponding to the subset $\Sigma$, and the
remaining   points   $z_{j+1},\dots,z_n$   run   over   all  possible
configurations  of  $n-j$ distinct regular points in $S$. Considering
for  each  configuration  a  quadratic differential $p^\ast q$ on the
resulting  cyclic  cover  $\hat  X$,  we construct a $\PSL$-invariant
suborbifold $\hat\cL$ of complex dimension $(\dim\cL+n-j)$.

Of  course,  the  proof  that  when  $\cL_1$  is endowed with a Borel
probability  measure, invariant with respect to the natural action of
the  group  $\PSL$  and  ergodic  with  respect  to the Teichm\"uller
geodesic flow, the new suborbifold $\hat\cL_1$ is also endowed with a
$\PSL$-invariant  measure satisfying the same properties, requires in
general   case  $n-j>0$  some  extra  work  (see,  for  example,  the
paper~\cite{Eskin:Marklof:Morris}   in  this  spirit).

\section{Concrete example: the calculations}
\label{s:Concrete:example:the:calculations}

In   this   section   we  treat  in  all  details  the  example  from
Section~\ref{ss:A:concrete:example}.

\subsection{The $\PSLZ$-orbit}
\label{ss:The:PSLZ:orbit}

It   is   an  exercise (left to the reader)  to  verify  that  the  $\PSLZ$-orbit
of  the square-tiled  surface  $S$  of  Figure~\ref{fig:oneline:6}  has the
structure  presented in Figure~\ref{fig:PSL2Z:orbit:below} below. By
historical  reasons,  the  initial  surface  $S$  is  denoted  as  $S_3$
there.

\begin{figure}[htb]

\includegraphics{oneline6.eps}

\includegraphics{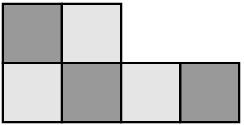}

\includegraphics{twolines6.eps}

\begin{picture}(0,0)(76,-79.5) 
\begin{picture}(0,0)(-1,98.5)
\put(0,0){5}
\put(0,-24.5){4}
\put(0,-49){3}
\put(0,-73.5){2}
\put(0,-98){1}
\put(0,-122.5){0}
\end{picture}
\begin{picture}(0,0)(1.5,146)
\put(-14,61){\tiny\textit A}
\put(-15,-12.75){\tiny\textit B}
\put(-14,-88){\tiny\textit A}
\put(15,61){\tiny\textit C}
\put(15,36.5){\tiny\textit F}
\put(15,12){\tiny\textit C}
\put(15,-12.5){\tiny\textit E}
\put(15,-37){\tiny\textit C}
\put(15,-61.5){\tiny\textit D}
\put(14.5,-88){\tiny\textit C}
\end{picture}
\begin{picture}(0,0)(-7,146)
\put(-10.5,64){\tiny $0$}
\put(-26,52.5){\rotatebox{180}{\tiny $0$}}
\put(4,52.5){\rotatebox{180}{\tiny $4$}}
\put(-26,28){\rotatebox{180}{\tiny $1$}}
\put(4,28){\rotatebox{180}{\tiny $5$}}
\put(-26,3.5){\rotatebox{180}{\tiny $2$}}
\put(4,3.5){\rotatebox{180}{\tiny $2$}}
\put(-26,-20.5){\rotatebox{180}{\tiny $3$}}
\put(4,-20.5){\rotatebox{180}{\tiny $3$}}
\put(-26,-45){\rotatebox{180}{\tiny $4$}}
\put(4,-45){\rotatebox{180}{\tiny $0$}}
\put(-26,-70){\rotatebox{180}{\tiny $5$}}
\put(4,-70){\rotatebox{180}{\tiny $1$}}
\put(-10.5,-88){\tiny $5$}
\put(-11.5,-98){\scriptsize $S_3$}
\end{picture}
\end{picture}

\begin{picture}(0,0)(0,-91.5) 
\begin{picture}(0,0)(0,97.5)
\put(-24.5,-98){5}
\put(-24.5,-122.5){4}
\put(0,-49){3}
\put(0,-73.5){2}
\put(0,-98){1}
\put(0,-122.5){0}
\end{picture}
\begin{picture}(0,0)(1.5,146)
\put(-16,12){\tiny\textit C}
\put(14.5,12){\tiny\textit F}
\put(-16,-12.75){\tiny\textit A}
\put(-40,-37){\tiny\textit E}
\put(-16,-34){\tiny\textit C}
\put(14.5,-37){\tiny\textit D}
\put(-40,-61){\tiny\textit B}
\put(-40,-85){\tiny\textit E}
\put(-14,-88){\tiny\textit C}
\put(14.5,-88){\tiny\textit F}
\end{picture}
\begin{picture}(0,0)(-7,146)
\put(-12,15){\tiny $0$}
\put(-26,4.5){\rotatebox{180}{\tiny $2$}} 
\put(4,4.5){\rotatebox{180}{\tiny $0$}}
\put(-26,-19.5){\rotatebox{180}{\tiny $3$}}
\put(4,-19.5){\rotatebox{180}{\tiny $1$}}
\put(-36,-34){\tiny $4$}
\put(-51,-45){\rotatebox{180}{\tiny $4$}}
\put(4,-45){\rotatebox{180}{\tiny $2$}}
\put(-51,-70){\rotatebox{180}{\tiny $5$}}
\put(4,-70){\rotatebox{180}{\tiny $3$}}
\put(-36,-88){\tiny $5$}
\put(-12,-88){\tiny $3$}
\put(-25,-98){\scriptsize $S_2$}
\end{picture}
\end{picture}

\begin{picture}(0,0)(-75,-103.5) 
\begin{picture}(0,0)(0,97.5)
\put(-24.5,-98){5}
\put(-24.5,-122.5){4}
\put(0,-49){3}
\put(0,-73.5){2}
\put(0,-98){1}
\put(0,-122.5){0}
\end{picture}
\begin{picture}(0,0)(1.5,146)
\put(-16,12){\tiny\textit C}
\put(14.5,-12.75){\tiny\textit F}
\put(-16,-12.75){\tiny\textit A}
\put(-40,-37){\tiny\textit B}
\put(-16,-34){\tiny\textit C}
\put(14.5,-61){\tiny\textit D}
\put(-40,-61){\tiny\textit E}
\put(-40,-85){\tiny\textit B}
\put(-14,-88){\tiny\textit C}
\end{picture}
\begin{picture}(0,0)(-7,146)
\put(-12,15){\tiny $0$}
\put(-26,4.5){\rotatebox{180}{\tiny $2$}} 
\put(4,4.5){\rotatebox{180}{\tiny $2$}}
\put(-26,-19.5){\rotatebox{180}{\tiny $3$}}
\put(4,-19.5){\rotatebox{180}{\tiny $3$}}
\put(-36,-34){\tiny $4$}
\put(-51,-45){\rotatebox{180}{\tiny $4$}}
\put(4,-45){\rotatebox{180}{\tiny $0$}}
\put(-51,-70){\rotatebox{180}{\tiny $5$}}
\put(4,-70){\rotatebox{180}{\tiny $1$}}
\put(-36,-88){\tiny $5$}
\put(-12,-88){\tiny $3$}
\put(-25,-98){\scriptsize $S_1$}
\end{picture}
\end{picture}

\begin{picture}(0,0)(-120,140) 
\put(0,0){\vector(0,1){120}}
\put(2,120){$x$}
\put(0,0){\vector(-1,0){240}}
\put(-240,-7){$y$}
\end{picture}

\includegraphics{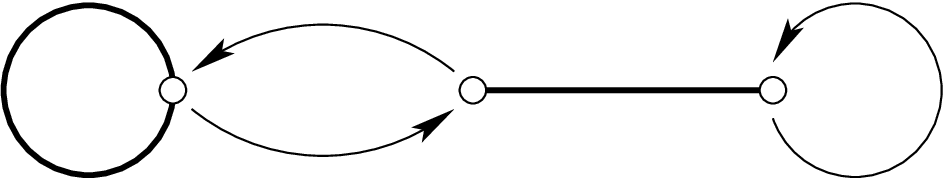}
\begin{picture}(0,0)(-80,100) 
\put(-176,-66){$h$}
\put(-118,-61){$r$}
\put(-67,-50){$h$}
\put(-67,-83){$h$}
\put(-9,-66){$r$}
\put(-140,-71){\scriptsize $S_3$}
\put(-97,-71){\scriptsize $S_2$}
\put(-36,-66){\scriptsize  $S_1$}
   %
\end{picture}

\vspace{180bp}
\caption{
\label{fig:PSL2Z:orbit:below}
$\PSLZ$-orbit of $S$.
}
\end{figure}

\begin{Convention}
\label{conv:hor:vert}
By typographical reasons, we are forced to use a peculiar orientation
as   in   Figure~\ref{fig:PSL2Z:orbit:below}  and  in  all
remaining  Figures in this paper. The notions ``horizontal''
and   ``vertical''  correspond  to  this  ``landscape  orientation'':
``horizontal''  means  ``parallel  to the $x$-axes'' and ``vertical''
means  ``parallel  to  the  $y$-axes''.  Under  this  convention, the
leftmost  surface  $S_3$ of Figure~\ref{fig:PSL2Z:orbit:below} has a
single  \textit{horizontal}  cylinder  of  height  $1$  and width $6$.
\end{Convention}

The three square-tiled surfaces $\hat S_1, \hat S_2, \hat S_3$ in the
$\PSLZ$-orbit    of    $\hat    S=\hat    S_3$   are   presented   in
Figure~\ref{fig:three:surfaces}.
This  figure also shows how the surfaces $\hat S_1, \hat S_2, \hat S_3$
are related by the basic transformations
$$
h=
\begin{pmatrix}
1&1\\
0&1
\end{pmatrix}
\qquad
\text{ and }
\qquad
r=
\begin{pmatrix}
0&-1\\
1&0
\end{pmatrix}
$$
given by the action of $\PSLZ$ on the flat surfaces
$\hat S_1, \hat S_2, \hat S_3$.

\subsection{Spectrum of Lyapunov exponents}
\label{ss:Spectrum:of:the:Lyapunov:exponents}

\begin{Lemma}
\label{lm:8:9}
The  sum  of  the  nonnegative Lyapunov exponents of the Hodge bundle
$H^1$  with  respect  to  the  geodesic flow on $\hat\cT$ is equal to
$8/9$.
\end{Lemma}
\begin{proof}
By  the  formula  for  the sum of Lyapunov exponents of the subbundle
$H^1_{\R{}}=H^1_+$ from~\cite{Eskin:Kontsevich:Zorich} one has

\begin{equation}
\label{eq:general:sum:of:plus:exponents:for:quadratic}
\lambda_1 + \dots + \lambda_g
\ = \
\cfrac{1}{24}\,\sum_{j=1}^\noz \cfrac{d_j(d_j+4)}{d_j+2}
+\frac{\pi^2}{3}\cdot c_{\mathit{area}}(\hat\cT)
\end{equation}
where  the  Siegel--Veech  constant  for the corresponding arithmetic
Teichm\"uller disc $\hat\cT$ is computed as
$$
c_{\mathit{area}}(\hat\cT)=
\cfrac{3}{\pi^2}\cdot
\cfrac{1}{\card(\PSLZ\cdot \hat S)}\
\sum_{\hat S_i\in\PSLZ\cdot \hat S}\ \
\sum_{\substack{
\mathit{horizontal}\\
\mathit{cylinders\ cyl}_{ij}\\
such\ that\\\hat S_i=\sqcup\mathit{cyl}_{ij}}}\
\cfrac{h_{ij}}{w_{ij}}\,,
$$
In   our   case  $\hat\cT\subset\cQ(7,1^5)$,  so  the  first  summand
in~\eqref{eq:general:sum:of:plus:exponents:for:quadratic} gives
$$
\frac{1}{24}\,\sum_{j=1}^\noz \cfrac{d_j(d_j+4)}{d_j+2}=
\frac{1}{24} \left(\frac{7\cdot 11}{9}  +
    5\cdot \frac{1\cdot 5}{3}\right)=
\frac{19}{27}\,.
$$
Observing  the  cylinder  decompositions of the three surfaces in the
$\PSLZ$-orbit of the initial square-tiled cyclic cover, we get:
$$
\frac{\pi^2}{3}\cdot c_{\mathit{area}}(\hat\cT)=
\frac{1}{3}\left(\frac{1}{18} + 2\left(\frac{1}{12} +
   \frac{1}{6}\right)\right)=
\frac{5}{27}\,.
$$
Thus,      taking      the      sum      of     the     two     terms
in~\eqref{eq:general:sum:of:plus:exponents:for:quadratic} we get
$$
\lambda_1 + \dots + \lambda_4=\frac{19}{27}+\frac{5}{27}=\frac{8}{9}
$$
\end{proof}

Consider    the  $\PSL$-invariant  subbundles   $\cE(\zeta)$,
$\cE(\zeta^2)$  of  the  Hodge  bundle $H^1_{\C{}}$ over $\hat\cT$ as
in~\eqref{eq:E:zeta:k}.    Note    that   in   our   case   we   have
$H^1_{\C{}}=\cE(\zeta)\oplus\cE(\zeta^2)$.

\begin{Proposition}
\label{prop:Lyapunov:spectrum}
The   Lyapunov  spectrum  of  the  real  and  complex  Hodge  bundles
$H^1_{\R{}}$  and  $H^1_{\C{}}$  with respect to the geodesic flow on
the arithmetic Teichm\"uller curve $\hat\cT$ is
$$
\left\{\frac{4}{9},\frac{4}{9},0,0,0,0,-\frac{4}{9},-\frac{4}{9}\right\}\,.
$$

The  Lyapunov  spectrum  of  each  of  the  subbundles  $\cE(\zeta)$,
$\cE(\zeta^2)$ with respect to the geodesic flow on $\hat\cT$ is
$$
\left\{\frac{4}{9},0,0,-\frac{4}{9}\right\}\,.
$$
\end{Proposition}
\begin{proof}
Note  that  the  vector  bundles  $\cE(\zeta)$ and $\cE(\zeta^2)$ are
complex conjugate. Hence, their Lyapunov spectra coincide.

The pseudo-Hermitian Hodge bilinear form on $H^1_{\C{}}$ is preserved
by  the  Gauss--Manin connection. By the theorem of C.~McMullen cited
at  the  end  of  Section~\ref{ss:Splitting:of:the:Hodge:bundle}, the
signature  of  its restriction to $\cE(\zeta)$, $\cE(\zeta^2)$ equals
to  $(3,1)$ and $(1,3)$ respectively. Thus, the restriction of the
cocycle      to      $\cE(\zeta)$,     $\cE(\zeta^2)$     lies     in
$\operatorname{U}(3,1)$  and $\operatorname{U}(1,3)$ respectively.
Hence,  by  Theorem~\ref{th:spec:of:unitary:cocycle}  the spectrum of
each of $\cE(\zeta)$ and $\cE(\zeta^2)$ has the form
$$
\{\lambda, 0, 0, -\lambda\}\,,
$$
where               $\lambda\ge               0$.               Since
$H^1_{\C{}}=\cE(\zeta)\oplus\cE(\zeta^2)$,  the  spectrum of Lyapunov
exponents of the Hodge bundle $H^1_{\C{}}$ is the union of spectra of
$\cE(\zeta)$  and  of  $\cE(\zeta^2)$. Since the Lyapunov spectrum of
$H^1_{\R{}}$  coincides with the one of $H^1_{\C{}}$ we conclude that
the spectrum of $H^1_{\R{}}$ is
$$
\{\lambda, \lambda, 0, 0, 0, 0, -\lambda, -\lambda\}\,.
$$
By Lemma~\ref{lm:8:9} we get $\lambda=4/9$.
\end{proof}

Recall  that  the  Oseledets  subspace  (subbundle)  $E_0$  (the  one
associated to the zero exponents) is called \textit{neutral Oseledets
subspace (subbundle)}.

\begin{Proposition}\label{prop:isometryU31}
The  Kontsevich-Zorich cocycle over $\hat{\cT}$ acts by isometries on
the   neutral  Oseledets
subbundle $E_0$ of each of the bundles $\cE(\zeta), \cE(\zeta^2)$. In
other  words,  the  restriction  of  the pseudo-Hermitian form to the
subbundle  $E_0$  of each of the bundles $\cE(\zeta), \cE(\zeta^2)$ is either
positive-definite or negative-definite.
\end{Proposition}

\begin{proof}
The   Kontsevich-Zorich   cocycle  over  $\hat\cT$  on  $\cE(\zeta)$,
respectively,  on  $\cE(\zeta^2)$,  is  a  $U(3,1)$,  respectively, a
$U(1,3)$,   cocycle.   Moreover,  by  Proposition~\ref{prop:Lyapunov:spectrum},  the
dimension   of  the  corresponding  neutral  Oseledets  subspaces  is
$2=3-1=|1-3|$.  By  Lemma~\ref{lemma:isometryUpq} below, this implies
that  the  Kontsevich-Zorich  cocycle  acts  by  isometries along the
neutral Oseledets subspace.
\end{proof}

\begin{remark}This Proposition was motivated by a question of Y.~Guivarch to the authors.
\end{remark}

We  prove  in Section~\ref{s:non:varying} a non-varying phenomenon
similar    to   the   one   proved   by   D.~Chen   and   M.~M\"oller
in~\cite{Chen:Moeller}  for strata in lower genera: certain invariant
loci of cyclic covers share the same sum of the Lyapunov exponents.

\section{Closed geodesics on an arithmetic Teichm\"uller curve}
\label{s:Closed:geodesics:on:an:arithmetic:Teichm:curve}

In  this  section we describe the basic facts concerning the geometry
of  a  general  arithmetic  Teichm\"uller  curve.  We  do  not  claim
originality:  these elementary facts are in part already described in
the    literature    (see~\cite{Herrlich},    \cite{Hubert:Lelievre},
\cite{Hubert:Schmidt},                 \cite{Moeller:Matheus:Yoccoz},
\cite{McMullen:square:tiled},         \cite{Schmithusen:Veech:group},
\cite{Schmithusen:examples},      \cite{Yoccoz},     \cite{Zmiaikou},
\cite{Zorich:square:tiled}  and  references  there);  in  part widely
known  in  folklore  concerning  square-tiled  surfaces (as in recent
experiments~\cite{Delecroix:Lelievre}); in part they can be extracted
from  the  broad  literature  on  coding  of geodesics on surfaces of
constant   negative   curvature   (see,   for  example,  \cite{Dalbo}
and~\cite{Series} and references there).

Consider  a  general  square-tiled  surface  $S_0$.  Throughout  this
section  we  assume  that the flat structure on $S_0$ is defined by a
quadratic  differential no matter whether it is a global square of an
Abelian  differential  or  not.  In  particular,  we deviate from the
traditional   convention   and   always   consider  the  Veech  group
$\Gamma(S_0)$  of  $S_0$  as  a  subgroup  of  $\PSL$, and never as a
subgroup of $\SL$.

We use the same notation $\cT$ for the arithmetic Teichm\"uller curve
defined  by  $S_0$  and for the corresponding hyperbolic surface with
cusps.

Note  that, when working with geodesic flows, in some situations one has
to  consider the points of the unit tangent bundle while in the other
situations the points of the base space. In our concrete example with
arithmetic Teichm\"uller curves, the orbit $\PSL\cdot S_0\subset Q_g$
of  a  square-tiled  surface  $S_0$  in the moduli space of quadratic
differentials  $\cQ_g$  plays  the role of the unit tangent bundle to
the  arithmetic  Teichm\"uller  curve $\cT\subset\cM_g$ in the moduli
space  $\cM_g$  of curves. The corresponding projection is defined by
``forgetting'' the quadratic differential:
$$
\cQ_g\ni S=(C,q)\mapsto C\in\cM_g\,.
$$

\subsection{Encoding a Veech group by a graph}
\label{ss:Encoding:a:Veech:group:by:a:graph}
Recall  that  $\PSLZ$  is isomorphic to the group with two generators
$h$  and  $r$ satisfying the relations
\begin{equation}
\label{eq:PSLZ:relations}
r^2=\id\qquad\text{and}\qquad (hr)^3=\id\,.
\end{equation}
As    generators    $r$    and    $h$    one   can   chose   matrices
$$
h=
\begin{pmatrix}
1&1\\
0&1
\end{pmatrix}
\qquad
\text{ and }
\qquad
r=
\begin{pmatrix}
0&-1\\
1&0
\end{pmatrix}\,.
$$

Having an irreducible square-tiled  surface  $S_0$  defined  by  a  quadratic differential,
construct  the  following  graph  $\mathbb{G}$. Its vertices are in a
bijection  with the elements of the orbit $\PSLZ\cdot S_0$. Its edges
are  partitioned  in two types. Edges of ``r-type'' are not oriented.
Edges of ``h-type'' are oriented. The edges are naturally constructed
as  follows.  Each vertex $S_i\in\mathbb{G}$ is joined by the edge of
the  $r$-type with the vertex represented by the square-tiled surface
$r\cdot  S_i$.  Each  vertex $S_i\in\mathbb{G}$ is also joined by the
oriented  edge  of the ``h-type'' with the vertex $h\cdot S_i$, where
the edge is oriented from $S_i$ to $h\cdot S_i$.

By  construction,  the graph $\mathbb{G}$ with marked vertex $S_0$ is
naturally   identified  with  the  coset  $\PSLZ/\Gamma(S_0)$,  where
$\Gamma(S_0)$  is  the Veech group of the square-tiled surface $S_0$.
(Irreducibility  of  $S_0$  implies  that  $\Gamma(S_0)$  is indeed a
subgroup of $\PSLZ$.)

The  structure  of  the  graph carries complete information about the
Veech group $\Gamma(S_0)$. Namely, any path on the graph $\mathbb{G}$
composed  from  a  collection  of its edges defines the corresponding
word in ``letters'' $h, h^{-1}, r$. Any \textit{closed} path starting
at   $S_0$   naturally   defines   an  element  of  the  Veech  group
$\Gamma(S_0)\subseteq\PSLZ$.    Reciprocally,    any    element    of
$\Gamma(S_0)$  represented  as  a  word  in generators $h, h^{-1}, r$
defines  a closed path starting at $S_0$. Two closed homotopic paths,
with  respect  to  the  homotopy  in $\mathbb{G}$ with the fixed base
point   $S_0$,   define   the   same   element  of  the  Veech  group
$\Gamma(S_0)$. Clearly, the resulting map
\begin{equation}
\label{eq:pi1:to:Gamma}
\pi_1(\mathbb{G},S_0)\to\Gamma(S_0)\subseteq\PSLZ\,.
\end{equation}
is  a  group homomorphism, and even epimorphism.

For  any  flat  surface $S=g\cdot S_0$ in the $\PSL$-orbit $\PSL\cdot
S_0$  of  the  initial  square-tiled  surface  $S_0$  the Veech group
$\Gamma(S)$  is  conjugated  to  the  Veech  group  of $S_0$, namely,
$\Gamma(S)=g\cdot\Gamma(S_0)\cdot   g^{-1}$.  One  can  construct  an
analogous  graph  $\mathbb{G}_S$ for $S$ which would be isomorphic to
the  initial one. The only change would concern the representation of
the  edges  in $\PSL$: the edges of the $h$-type would be represented
now by the elements $ghg^{-1}$ and the edges of the $r$-type would be
represented by the elements $grg^{-1}$.

Note  that  by  the  result~\cite{Hubert:Lelievre:nonconguence}  of
P.~Hubert  and S.~Lelievre, in general, $\Gamma(S_0)$ \textit{is not}
a congruence subgroup.

One  can  formalize  the  properties  of  the  graph  $\mathbb{G}$ as
follows:

\begin{itemize}
\item[(i)]  Each  vertex  of  $\mathbb{G}$ has valence three or four,
where  one  valence  is  represented  by  an  outgoing  edge  of  the
``h-type'',  another  one  --- by an incoming edge of the ``h-type'';
the  remaining  one or two valences are represented by an $r$-edge or
an $r$-loop respectively;
    %
\item[(ii)]  The  path  $hrhrhr$  (where we follow the orientation of
each  $h$-edge) starting from any vertex of the graph $\mathbb{G}$ is
closed.
\end{itemize}

\begin{Question}
Does any abstract graph satisfying properties (i) and (ii) represents
the $\PSLZ$-orbit of some square-tiled surface $S_0$?
\end{Question}

J.~Ellenberg  and  D.~McReynolds  gave  an  affirmative answer to the
latter  question  for  square-tiled surfaces with markings (i.e. with
``fake  zeroes''),  see~\cite{Ellenberg:McReynolds}.  We  are curious
whether the answer is still affirmative for flat surfaces without any
fake zeroes?

Note   that   certain   infinite   collections  of  pairwise-distinct
square-tiled surfaces might share the same Veech group $\Gamma$, and,
thus,  the  same  graph $\mathbb{G}$. As the simplest example one can
consider already $\Gamma=\PSLZ$: infinite collections of square-tiled
surfaces with the Veech group $\PSLZ$ are constructed
in~\cite{Schmithusen:Veech:group},~\cite{Schmoll:PSLZ},~\cite{Herrlich}, and~\cite{Forni:Matheus:Zorich_1}.

However,   if   one  considers  any  fixed  stratum,  the  number  of
square-tiled  surfaces  with  a  fixed isomorphism class of the Veech
group    is   finite   within   this   stratum   (see~\cite[Corollary
1.7]{Smillie:Weiss}). Nevertheless, these square-tiled surfaces might
be  distributed into several $\PSLZ$-orbits (see, say, Example 5.3 of
F. Nisbach's Ph.D. thesis~\cite{Nisbach}).

\subsection{Partition    of    an   arithmetic   Teichm\"uller   disc
into hyperbolic triangles}

Consider  the  modular curve (modular surface)
$$
\mathcal{MOD}=\PSO\backslash\PSL/\PSLZ\,.
$$
Consider  its  canonical  fundamental domain in the upper half plane,
namely the hyperbolic triangle
\begin{equation}
\label{eq:fundam:domain}
\{z\,|\, \Im z>0\}\ \cap\
\{z\,|\, -1/2\le \Re z\le 1/2\}\ \cap\
\{z\,|\quad |z|\ge 1\}
\end{equation}
with angles $0,\pi/3,\pi/3$. Any arithmetic Teichm\"uller curve $\cT$
has  a  natural  structure  of  a  (possibly ramified) cover over the
modular curve, and, thus, it is endowed with the natural partition by
isometric triangles as above. We accurately say ``partition'' instead
of  ``triangulation''  because of the following subtlety: the side of
the  triangle  represented  by  the circle arc in the left picture of
Figure~\ref{fig:modular:surface}  might be folded in the middle point
$B$  and glued to itself, as it happens, for example, already for the
modular  surface  $\mathcal{MOD}$. The vertices and the sides of this
partition  define  a  graph  $\check{\mathbb{G}}$  embedded  into the
compactified   surface   $\bar\cT$,  where  we  apply  the  following
convention:  each  side of the partition, which is bent in the middle
and  glued  to  itself,  is  considered  as  a loop of the graph, see
Figures~\ref{fig:modular:surface}   and~\ref{fig:graph:of:a:surface}.
In particular, the middle point of such side \textit{is not} a vertex
of the graph $\check{\mathbb{G}}$.

\begin{figure}
   %
\includegraphics{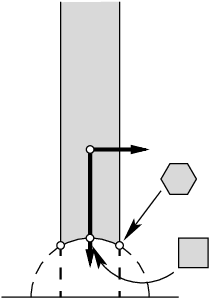}
\includegraphics{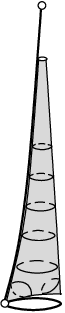}
\includegraphics{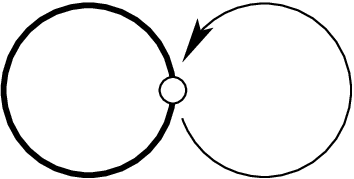}

\begin{picture}(0,0)(0,0) %
\put(-97,-97){$h$}
\put(-109,-118){$r$}
\put(-126,-147){$A$}
\put(-111,-141){$B$}
\put(-85,-147){$A$}
\put(-18,-147){$A$}
\put(2,0){$C$}
\put(3,-165){$\check{\mathbb{G}}$}
\put(100,-130){$\mathbb{G}$}
\put(75,-150){$r$}
\put(134,-150){$h$}
\end{picture}

\vspace{170pt}
\caption{
\label{fig:modular:surface}
Modular   surface,  its  fundamental  domain,  and  associated  graphs
$\mathbb{G}$ and $\check{\mathbb{G}}$.
}
\end{figure}
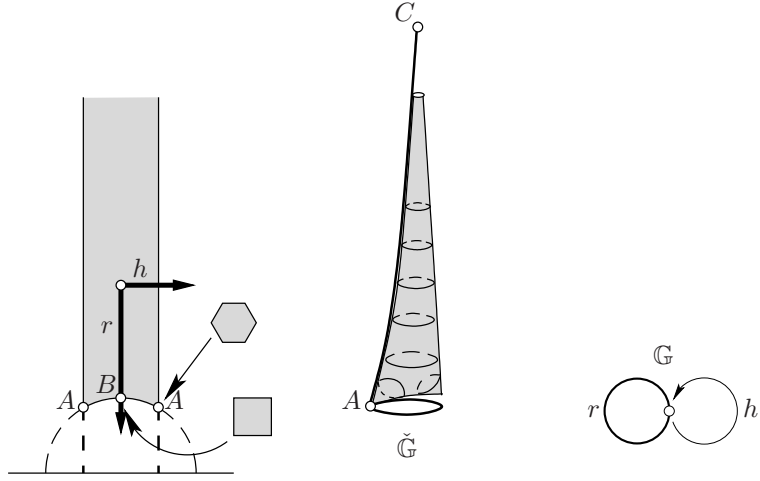

The   degree   of  the  cover  $\cT\to\mathcal{MOD}$  equals  to  the
cardinality  of the $\PSLZ$-orbit of the initial square-tiled surface
$S_0$,
\begin{equation}
\label{eq:card}
\deg(\cT\to\mathcal{MOD}) =\card\big(\PSLZ\cdot S_0\big)\,.
\end{equation}
The  cover  $\cT\to\mathcal{MOD}$  might be ramified over two special
points  of  $\mathcal{MOD}$. The first possible ramification point is
the  point  $B$  (having  coordinate $i$); it corresponds to the flat
torus       glued       from      the      unit      square,      see
Figure~\ref{fig:modular:surface}.  The  second  possible ramification
point  is  the  point  $A$  represented  by  the  identified  corners
$e^{-(\pi    i)/3}$   and   $e^{(\pi   i)/3}$   of   the   hyperbolic
triangle~\eqref{eq:fundam:domain}.  The  latter  point corresponds to
the flat torus glued from the regular hexagon.

Any  preimage of the point $B$ (see Figure~\ref{fig:modular:surface})
is  either  regular or has ramification degree two. In the first case
the preimage is a conical singularity of the hyperbolic surface $\cT$
with the cone angle $\pi$ (as for the modular surface $\mathcal{MOD}$
itself); in the latter case it is a regular point of $\cT$.

Any  preimage of the point $A$ (see Figure~\ref{fig:modular:surface})
is  either regular or has the ramification degree three. In the first
case  the preimage is a conical singularity of the hyperbolic surface
$\cT$  with  the  cone  angle  $2\pi/3$  (as  for the modular surface
$\mathcal{MOD}$  itself); in the latter case it is a regular point of
$\cT$.

For   each   of   the   two   special  points  of  the  base  surface
$\mathcal{MOD}$  some  preimages  might be regular and some preimages
might  be  ramification  points. The cover $\cT\to\mathcal{MOD}$ does
not have any other ramification points.

A  square-tiled  surface  $S=(X,q)$  in the moduli space of quadratic
differentials   defines   a  conical  point  $X$  of  the  arithmetic
Teichm\"uller  disc  if  and  only if $(X,q)$ and $(X,-q)$ define the
same  point  in  the  moduli  space.  In  other words, a square-tiled
surface   $S$   projects   to  a  conical  point  of  the  arithmetic
Teichm\"uller  disc  if  turning  it  by $\pi/2$ we get an isomorphic
square-tiled surface.

\subsection{Encoding an arithmetic Teichm\"uller curve by a graph}
\label{ss:dual:graph}
Note  that  the  set of the preimages in $\cT$ of the point $B$ (with
coordinate         $i$)         in        $\mathcal{MOD}$        (see
Figure~\ref{fig:modular:surface})        under        the       cover
$\cT\to\mathcal{MOD}$   coincides   with   the   collection   of  the
projections of the orbit $\PSLZ\cdot S_0$ in the moduli space $\cQ_g$
of  quadratic  differentials  to  the moduli space $\cM_g$ of curves.
Since  the  cover $\cT\to\mathcal{MOD}$ is, in general, ramified over
$i$,  the  cardinality  of  the  latter  set  might  be less than the
degree~\eqref{eq:card}  of the cover. In this sense, the square-tiled
surfaces  are  particularly  \textit{inconvenient}  to  enumerate the
hyperbolic triangles as above.

Consider a flat torus $T$ which does not correspond to any of the two
conical  points  of the modular surface $\mathcal{MOD}$. For example,
let $T$ correspond to the point $4i$ of the fundamental domain. Let
$$
g=\begin{pmatrix}2&0\\0&1/2\end{pmatrix}\in\PSL\,.
$$
Then  $T=g\cdot T_0$, where $T_0$ stands for the torus glued from the
standard  unit  square.  Consider  the  following  two  closed  paths
$\gamma_h,  \gamma_r$  in  the  modular surface starting at $4i$, see
Figure~\ref{fig:modular:surface}.  The  path  $\gamma_h$  follows the
horizontal  horocyclic loop, while the path $\gamma_r$ descends along
the vertical geodesic from $4i$ to $i$ and returns back following the
same  vertical geodesic. The point $T$ of $\mathcal{MOD}$ and the two
loops  $\gamma_h$,  $\gamma_r$  can be considered as a realization of
the   graph  $\mathbb{G}_T$  under  the  usual  convention  that  the
``folded'' path $\gamma_r$ is considered as the loop of the graph.

\begin{figure}
\includegraphics{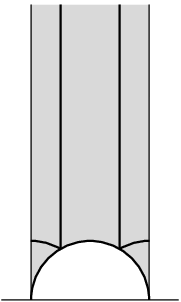}
\includegraphics{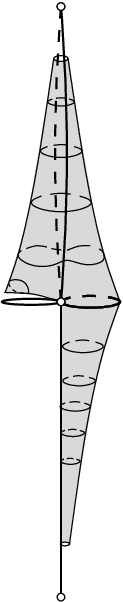}
\includegraphics{Tdisc3.eps}

\begin{picture}(0,0)(0,0) %
\put(-120,-145){$A$}
\put(-92,-145){$A$}
\put(-134,-170){$C_1$}
\put(-75,-170){$C_1$}
\put(0,-98){$A$}
\put(11,-7){$C_2$}
\put(11,-178){$C_1$}
\put(-10,-120){$\check{\mathbb{G}}$}
\put(93,-120){$\mathbb{G}$}
\put(45,-144){$r$}
\put(80,-157){$h$}
\put(80,-133){$h$}
\put(110,-139){$r$}
\put(143,-144){$h$}
\end{picture}

\vspace{170pt}
\caption{
\label{fig:graph:of:a:surface}
Partition of the Teichm\"uller curve $\hat\cT$, and associated graphs
$\mathbb{G}$ and $\check{\mathbb{G}}$.
}
\end{figure}

For  any  square-tiled  surface  $S_0$ consider the surface $S=g\cdot
S_0$  in  the  $\PSL$-orbit of $S_0$. By construction, it projects to
$T=g\cdot  T_0$  under  the cover $\cT\to\mathcal{MOD}$. Consider all
preimages  of $T$ under this cover, and consider the natural lifts of
the  loops $\gamma_h$ and $\gamma_r$. Under the usual convention that
``folded''  paths  are  considered  as loops of the graph, we get the
graph                       $\mathbb{G}_S$                       from
Section~\ref{ss:Encoding:a:Veech:group:by:a:graph}.

The geometry of the hyperbolic surface $\cT$ is completely encoded by
each     of    the    graphs    $\mathbb{G}\simeq\mathbb{G}_S$    and
$\check{\mathbb{G}}$.   For  example,  the  cusps  of  $\cT$  can  be
described as follows (see~\cite{Hubert:Lelievre}).

\begin{Lemma}
The  cusps  of  the  hyperbolic  surface  $\cT$  are  in  the natural
bijection with the orbits of the subgroup generated by the element
$$
h=
\begin{pmatrix}
1&1\\
0&1
\end{pmatrix}
$$
on  $\mathbb{G}\simeq\PSLZ/\Gamma(S_0)$. In other words, the cusps of
the  hyperbolic  surface  $\cT$ are in the natural bijection with the
maximal   positively  oriented  chains  of  $h$-edges  in  the  graph
$\mathbb{G}$.
\end{Lemma}

\begin{Remark}
As  it  was  pointed  out  by  D.~Zmiaikou, the Lemma above should be
applied  in  the  context of the action of $\PSL$ and \textit{not} of
$\SL$, see~\cite{Zmiaikou}.
\end{Remark}

It  is clear from the construction that the graphs $\mathbb{G}_S$ and
$\check{\mathbb{G}}$  are  in natural duality: under the natural
embedding  of  $\mathbb{G}_S$ into $\cT$ described above, the vertices
of   the   graph  $\mathbb{G}_S$  are  in  the  canonical  one-to-one
correspondence  with  the  hyperbolic  triangles  in the partition of
$\cT$;    the    edges    of    the    graphs    $\mathbb{G}_S$   and
$\check{\mathbb{G}}$   are  also  in  the  the  canonical  one-to-one
correspondence;  under our usual convention concerning the loops, one
can assume that the dual loops intersect transversally.
This allows us to encode the paths on $\cT$, and, more particularly, the
closed loops on $\cT$, with a fixed base point (or rather the homotopy
classes  of  such  loops in a homotopy fixing the base point), by the
closed  loops  on the graph $\mathbb{G}$. This observation is used in
the  next  section,  where we discuss the monodromy representation of
our  main  example, that is, the arithmetic Teichm\"uller curve $\hat\cT$
defined in Section~\ref{ss:A:concrete:example}.
\begin{Remark}
One  can  go  further,  and  encode  the  hyperbolic geodesics on any
arithmetic  Teichm\"uller  disc using the continued fractions and the
associated  sequences  in ``letters'' $h$, $h^{-1}$, $r$. This coding
is   the  background  of  numerous  computer  experiments  evaluating
approximate  values  of \mbox{Lyapunov} exponents of the Hodge bundle
over  general  arithmetic  Teichm\"uller  discs, as, for example, the
ones       described       in~\cite{Eskin:Kontsevich:Zorich}       or
in~\cite{Delecroix:Lelievre}.  We  refer the reader to the detailed
surveys~\cite{Series}   and~\cite{Dalbo}  (and  to  references  cited
there)  for  generalities  on  geometric  coding  of geodesics on the
modular  surface.  The  coding  adapted particularly to Teichm\"uller
discs is described in~\cite{Moeller:Matheus:Yoccoz}.
\end{Remark}
   %

\section{Irreducibility of the Hodge bundle in the example}
\label{s:Irreducibility:of:the:Hodge:bundle}



In this section we prove that the orthogonal splitting
into the subbundles $\cE(\zeta)$  and $\cE(\zeta^2)$ is the unique irreducible
decomposition of the Hodge bundle  into $\PSL$-invariant (continuous) complex subbundles. We then
use the fact that the Zariski closure of the monodromy representations on the bundles
 $\cE(\zeta)$  and $\cE(\zeta^2)$ is determined in Appendix~\ref{a:Zariski:closure} to generalize our irreducibility result to
all finite covers of the Teichm\"uller curve $\hat\cT$ (strong irreducibility).

\subsection{Irreducibilty of the decomposition}
\label{ss:irreducibility}

We  start  with  the  following  elementary Lemma from linear algebra
which we present without proof.

\begin{Lemma}
\label{lm:lin:alg}
Let  $A, B$ be two $n\!\times\! n$-matrices. If $\det(AB-BA)\neq 0$, then
the  corresponding  linear automorphisms of $\R^{n}$ (respectively
$\C^{n}$) do not have any common one-dimensional invariant subspaces.
\end{Lemma}

It would be convenient to work with the dual \textit{homology} vector
bundle   over   the   Teichm\"uller   curve  $\hat\cT$  and  with  its
decomposition  into  direct  sum  of  $\PSL$-invariant subbundles
$\mathcal{E}_*(\zeta)\oplus \mathcal{E}_*(\zeta^2)$, where
\begin{equation*}
\mathcal{E}_*(\zeta^k):=\operatorname{Ker}(T_\ast-\zeta^k\operatorname{Id})
\subseteq H_1(X;\C{})\,,
\end{equation*}
compare to~\eqref{eq:E:zeta:k}. Of course, since $H^1(X;\C)$ and $H_1(X;\C)$ are in duality, we can safely replace $\cE(\zeta^k)$ with $\cE_*(\zeta^k)$ in our subsequent discussion of the complex Kontsevich-Zorich cocycle over $\hat\cT$.

\begin{Proposition}
\label{prop:Ezeta:does:not:have:subbundles}
The subbundles  $\mathcal{E}(\zeta)$ and $\mathcal{E}(\zeta^2)$ over the Teichm\"uller curve $\hat\cT$ are irreducible, i.e., they do not have any nontrivial $\PSL$-invariant complex subbundles.
\end{Proposition}
\begin{proof}We note that $\cE_*(\zeta)$ and $\cE_*(\zeta^2)$ are complex-conjugate
and the monodromy respects the complex conjugation, hence  it suffices to prove
that one of them, for instance $\cE_*(\zeta)$ is irreducible.
In  Section~\ref{ss:calculation} of Appendix~\ref{a:matrix:calculation} we take (more or less at random) the
following  two  closed  oriented paths $\rho_1$, see~\eqref{eq:rho1},
and    $\rho_2$,    see~\eqref{eq:rho2},   starting and ending at the same
vertex $\hat S_1$,  on the graph of Figure~\ref{fig:PSL2Z:orbit} (representing the $\PSLZ$-orbit of $\hat{S}$):
\begin{align*}
\rho_1:&=h\cdot r h^{-3} r\cdot h\cdot r h^{-2} r\\
\rho_2:&=r h^{-1} r\cdot h^3\cdot r h^{-1} r \,,
\end{align*}
where  each  path  should  be  read from left to right. The paths are
chosen  to be compatible with the orientation of the graph. Using the
explicit  calculation  of  the  monodromy\footnote{Note that all monodromy matrices are
defined up to a multiplication by the complex numbers $\zeta^k$, $k=0,1,2$, induced
by the action of the automorphism group of a cyclic cover~\eqref{eq:cyclic:cover:equation}.} representation performed in Appendix~\ref{a:matrix:calculation},        we       compute       in
Section~\ref{ss:calculation}    the   monodromy   $X,Y:   \cE_*(\zeta)\to
\cE_*(\zeta)$  along the paths $\rho_1,\rho_2$ respectively, and verify
that  $\det(XY-YX)\neq 0$. By Lemma~\ref{lm:lin:alg} this proves that
$\cE_*(\zeta)$  does  not  have  any one-dimensional $\PSL$-invariant
subbundles. Since the monodromy preserves the pseudo-Hermitian
intersection  form,  which is non-degenerate, by  duality  the  complex
four-dimensional  bundle $\cE_*(\zeta)$  does  not  have  any $\PSL$-invariant
codimension one subbundles, i.e., three-dimensional ones.

Given  the  monodromy matrices $X,Y$ along the paths $\rho_1,\rho_2$
we  compute  in  Section~\ref{ss:calculation}  the  induced monodromy
matrices  $U,V$  in  the second wedge product $\Lambda^2 \cE_*(\zeta)$ of
$\cE_*(\zeta)$,  and  verify  that  $\det(UV-VU)\neq 0$. This proves that
$\cE_*(\zeta)$  does  not  have  any two-dimensional $\PSL$-invariant subbundles.
\end{proof}

\begin{NNRemark}
The  same kind of a straightforward proof of irreducibility based on
Lemma~\ref{lm:lin:alg}   was   implemented   in   a  similar  setting
in~\cite[Appendix B]{Zorich:how:do}.
\end{NNRemark}

\begin{Proposition}
\label{prop:only:E}
The  complex  Hodge  bundle  $H^1_{\C{}}$ over the Teichm\"uller curve
$\hat\cT$  has no nontrivial $\PSL$-invariant
complex subbundles other than $\cE(\zeta)$ and $\cE(\zeta^2)$.
\end{Proposition}
\begin{proof}
By Proposition~\ref{prop:Ezeta:does:not:have:subbundles} the bundles $\cE(\zeta)$ and
$\cE(\zeta^2)$
do  not  have  any non-trivial $\PSL$-invariant subbundles.
Since  the  complex  Hodge  bundle  $H^1_{\C{}}$  over  $\hat\cT$  is
decomposed into the direct sum of two orthogonal $\PSL$-invariant
subbundles
$$
H^1_\C{}=\cE(\zeta)\oplus \cE(\zeta^2)\,,
$$
this   implies  that  $H^1_\C{}$  cannot  have  $\PSL$-invariant
subbundles  of dimension $1,2,3$, otherwise the orthogonal projections
to  the direct summands would produce nontrivial $\PSL$-invariant
subbundles. Moreover, since  the  flat  connection preserves the nondegenerate
pseudo-Hermitian intersection form, this implies that the orthogonal complement to
a  $\PSL$-invariant subbundle cannot have dimension $1,2,3$, and
thus  the  Hodge  bundle  does  not  have  any  $\PSL$-invariant
subbundles of dimension $5,6,7$.

If  there  existed  a $\PSL$-invariant complex subbundle $V$ of
dimension $4$ different   from  $\cE(\zeta)$  and  $\cE(\zeta^2)$,  its  orthogonal
projections   $\pi_1,\pi_2$   to   $\cE(\zeta)$  and  $\cE(\zeta^2)$,
respectively, would be $\PSL$-invariant isomorphisms. The composition
$\pi_1^{-1}\circ  \pi_2$ would establish a $\PSL$-invariant isomorphism
between  subbundles $\cE(\zeta)$ and $\cE(\zeta^2)$. This would imply
that the vector bundles $E(\zeta)$
and  $E(\zeta^2)$  would  be isomorphic  and  would  have isomorphic
monodromy  representations. However the bundles $\cE(\zeta)$
and  $\cE(\zeta^2)$  are  complex  conjugate and the monodromy representation
respects complex conjugation, hence  the proof will be completed
by finding a monodromy matrix $C$ on $\cE(\zeta)$ which has a different
spectrum from its complex conjugate $\bar C$ even up to multiplication
times $\zeta^k$ for $k=0, 1, 2$.

In fact, let us consider closed paths  $\mu_1$  and  $\mu_2$  starting  and ending at
the same vertex $\hat S_3=\hat{S}$ on the graph of Figure~\ref{fig:PSL2Z:orbit} given
by
\begin{equation}
\label{eq:mu_1_2}
\mu_1:=h \quad \textrm{ and }\quad
\mu_2:=(r^{-1}h^{-1} r)\cdot (r^{-1}h^{-1}r)\,.
\end{equation}
A  closed  path on the graph of the $\PSLZ$-orbit of $\hat S$ defines
the  free  homotopy  type  of  a path on the corresponding arithmetic
Teichm\"uller  curve $\hat\cT$.
An explicit computation (see Sections~\ref{ss:Construction:of:homology:bases} and~\ref{ss:calculation} of  Appendix~\ref{a:matrix:calculation})
shows   that   the   monodromy matrices  $A,  B:  \cE_*(\zeta)\to \cE_*(\zeta)$ associated to
$\mu_1, \mu_2$ are
\begin{equation}
\label{eq:matrixA}
A:=A_3^{hor}=
\left(\begin{array}{cccc}
0 & 0 & 1 & \zeta^2 \\
\zeta & 0 & 0 & \zeta \\
0 & \zeta & 0 & \zeta \\
0 & 0 & 0 & -\zeta^2
\end{array}\right)\,,
\end{equation}
\begin{equation}
\label{eq:matrixB}
B:=A_1^{vert}\cdot A_3^{vert}=
\left(\begin{array}{cccc}
0&\zeta^2-1&\zeta&0\\
0&\zeta&0&0 \\
\zeta&\zeta-\zeta^2&1-\zeta^2&0 \\
1-\zeta^2&1-\zeta^2&1-\zeta&1
\end{array}\right)\,.
\end{equation}

We claim that the spectrum of the matrix
$C=B\cdot A$ is different from that of its complex conjugate $\bar C$ even up to
the action of the automorphism group of the cyclic cover, that is, up to multiplication
by the complex numbers  $\zeta^k$, $k=0,1,2$.

In fact, a computation (see Appendix~\ref{a:Zariski:closure}) shows  that
\begin{equation}
\label{eq:matrixC}
C:=B\cdot A =
\left(\begin{array}{cccc}
1-\zeta &\zeta^2& 0 &-2\zeta\\
\zeta^2&0&0& \zeta^2 \\
\zeta^2-1&\zeta-1&\zeta &-2 \\
\zeta-1&\zeta-\zeta^2&1-\zeta^2&2\zeta
\end{array}\right)\,,
\end{equation}
and that the characteristic polynomial of the matrix $C$ is
\begin{equation}
\label{eq:char:pol:C}
T^4 + (\zeta^2-\zeta)T^3 -2\zeta^2T^2 + (\zeta^2-1) T + \zeta= (T-1)( T^3-2\zeta T^2+2T-\zeta)\,.
\end{equation}
If the spectrum of $C$ and $\bar C$ up to multiplication times  $\zeta^k$, $k=0,1,2$, have a
common element not equal to $1$, then there exist $k\in \{0, 1, 2\}$ and $T \in \C$ such that
$$
T^3-2\zeta T^2+2T -\zeta =
 (\zeta^k T)^3-2 \bar \zeta (\zeta^k T)^2+2\zeta^k T -
\bar \zeta=0\,.
$$
By subtracting the two identities above, taking into account that $\zeta^3=1$, we can derive
the following identity
$$
2(\zeta^{2k-1}-\zeta) T^2 + 2(\zeta^k-1)T + \zeta -1/\zeta =0\,.
$$
The roots of the above second degree equation can be computed by hand for
$k=0, 1, 2$ and it can then be checked that none of them is a root of the characteristic polynomial
in formula~\eqref{eq:char:pol:C}. The argument is therefore completed.
\end{proof}

\begin{Corollary}
\label{cor:no:R:subspaces}
The  real  Hodge  bundle  $H^1_{\R{}}$ over the Teichm\"uller curve
$\hat\cT$ has no non-trivial
$\PSL$-invariant subbundles.
\end{Corollary}
\begin{proof}
Let  $V$ be a $\PSL$-invariant subbundle of the real Hodge bundle
$H^1_{\R{}}$ over  $\hat\cT$.  Its complexification $V_{\C{}}$ is a $\PSL$-invariant
subbundle of the complex Hodge bundle.
Moreover,   by   construction  it  is  invariant  under  the  complex
conjugation. Since $\cE(\zeta)$ and $\cE(\zeta^2)$ are complex conjugate,
Proposition~\ref{prop:only:E} implies that $V$ is trivial.
\end{proof}

\subsection{Zariski closure of the monodromy group}
\label{ss:Zariski:closure}
Following  a suggestion of M.~M\"oller we sketch in this section the computation of the
Zariski  closure  of  the monodromy group of the bundle $\cE(\zeta)$.
This     computation  (performed in details in Appendix~\ref{a:Zariski:closure} of this paper) implies     a     stronger     version    of
Proposition~\ref{prop:Ezeta:does:not:have:subbundles} stated in Proposition~\ref{prop:strongly:irreducible}.  The idea of this computation is due to A.~Eskin.

\begin{lemma}
\label{lm:Ezeta:strong:irreducibility}
The connected
component  of  the  identity  of  the Zariski  closure  of the monodromy group of the bundle
$\cE(\zeta)$  over  the  Teichm\"uller curve $\hat\cT$ is isomorphic to
$\textrm{SU}(3,1)$.
\end{lemma}
\begin{proof}
It    follows    from   the   Theorem   of   C.~McMullen   cited   in
Section~\ref{ss:Splitting:of:the:Hodge:bundle}   that  the  monodromy
group   $G$   of   the   flat   bundle   $\cE(\zeta)$  preserves  the
pseudo-Hermitian form of signature $(3,1)$. The direct computation of
the  generators  of $G$ shows that it is generated by matrices having
determinant  $\zeta^k$, with $k$ integer. Hence, the connected
component  of  the  Zariski  closure  of  $G$ containing the identity
element is isomorphic to a subgroup of $\textrm{SU}(3,1)$. In order to prove that this
subgroup  is in fact the whole $\textrm{SU}(3,1)$, it is sufficient to show
that  the  Lie  algebra  of  the  Zariski closure of $G$ has the same
dimension as the Lie algebra $\mathfrak{su}(3,1)$, that is $15$. In other
words,  it is sufficient to find $15$ linearly independent vectors
in  the  Lie  algebra  of  the  Zariski  closure of $G$, which we do,
basically, by hands.

For  any  hyperbolic  (or  parabolic)  element $C$ in $G$ the
vector $X=\log(C)$ belongs to the the Lie algebra $\mathfrak{g}_0$ of the
Zariski  closure  of $G$. Also, together with any vector $X$, the Lie
algebra         $\mathfrak{g}_0$        contains        the        vector
$\operatorname{Ad}_g(X)=gXg^{-1}$,  where  $g$ is any element in $G$.
Thus,  it  is  sufficient  to find a single vector in the Lie algebra
$\mathfrak{g}_0$,  then  conjugate it by elements of $G$;  as soon as
we get by this procedure $15$ linearly independent vectors, the proof
is completed.

The  rest  of  the computation is computer-assisted. We first find an
explicit  hyperbolic $4\times 4$ matrix $C$ in $G$ and an algebraic
expression  for  the  matrix  $P$  which conjugates $C$ to a diagonal
matrix  $D$.  This  allows  us to compute $X=\log C=P\cdot\log D\cdot
P^{-1}$ with arbitrary high precision.

As  soon  as  we have a collection of linearly independent vectors in
$\mathfrak{g}_0$  we  construct  a  new  vector  as follows: we take some
element  $g$  in  $G$ and compute the distance from $gXg^{-1}$ to the
subspace  generated  by our collection of independent vectors. If the
distance is large enough, the new vector is linearly independent from
the previous ones and  we  add  it  to  our  collection.  If  the  distance  is
suspiciously small, we try another element $g$ in $G$.

This algorithm is implemented in practice as follows.
Let $A$ and $B$ be the matrices of formulas~\eqref{eq:matrixA} and
~\eqref{eq:matrixB} respectively. Both  elements  are  elliptic; $A$ has order $18$,
$B$ has order $6$; the  monodromy  group  $G$ is generated by $A$ and $B$.
We check that the matrix $C:= B\cdot A$ of formula~\eqref{eq:matrixC}  is  hyperbolic,
then we compute $X=\log C$ as indicated above, and show that the $15$ vectors
\begin{align*}
A^{n}\,\cdot \,&\, X\cdot A^{-n} &n=0,\dots,8;\\
B\cdot A^{n}\,\cdot \,&\, X\cdot A^{-n}\cdot B^{-1} &n=0,2,3,4,5,6
\end{align*}
are   linearly   independent. See Appendix ~\ref{a:Zariski:closure} of this paper for more details.
\end{proof}

\begin{Remark}
Our  initial  plan was to use parabolic elements in the group and not
hyperbolic  ones.  Parabolic  elements have an obvious advantage that
their  logarithms  are  polynomials  and  thus, the vector in the Lie
algebra  corresponding to an integer parabolic matrix can be computed
explicitly. As  a natural candidate for a parabolic element one can consider the
map in cohomology of a square-tiled surface induced by a simultaneous
twist  of  the  horizontal  cylinders by
$$
h^n=
\begin{pmatrix}
1&n\\
0&1
\end{pmatrix}
$$
with  $n$  equal  to, say, the least common multiple of the widths of
the cylinders.

In the case of square-tiled  surfaces corresponding to Abelian differentials
we  would  certainly  get parabolic elements in this way. However, in
our case square-tiled surfaces correspond to quadratic differentials.
A  direct computation shows that the square-tiled surfaces $\hat S_1,
\hat     S_2,     \hat     S_3$    in    the    $\PSLZ$-orbit    (see
Figure~\ref{fig:three:surfaces})  have  the  following  property: the
waist  curve  of  any horizontal cylinder is homologous to zero. As a
result,  the  monodromy  along  any  path  on the Teichm\"uller curve
$\hat\cT$  represented by an element $h^n$ as above is elliptic (i.e.
has  finite  order)  and  not  parabolic. We do not know whether the
monodromy group in our example has at least one parabolic element.
\end{Remark}

\subsection{Strong irreducibilty of the decomposition}
\label{ss:strong_irreducibility}

\begin{Proposition}
\label{prop:strongly:irreducible}
The   subbundles   $\cE(\zeta)$   and   $\cE(\zeta^2)$  over the Teichm\"uller curve $\hat\cT$
are  strongly irreducible,  i.e.,  their  lifts  to  any finite (possibly ramified)
cover of $\hat\cT$ are irreducible.
\end{Proposition}
\begin{proof}
First note that $\cE(\zeta)$ and $\cE(\zeta^2)$ are complex-conjugate
and the monodromy respects the complex conjugation. Moreover, for any
finite,  possibly  ramified,  cover  the  induced vector bundles stay
complex  conjugate.  Thus, it suffices to show that one of them, say,
$\cE(\zeta)$ is strongly irreducible.

The  second  observation  is  that  the  component of identity of the
Zariski  closure  of  the  monodromy  group  of  a  vector  bundle is
invariant  under  finite covers. In order to see this it is sufficient to note
that  for  any  hyperbolic  or  parabolic element $g$ in the original
monodromy  group, the monodromy group of the vector bundle induced on
a    finite    cover    contains    some    power   of   $g$.   Thus,
Lemma~\ref{lm:Ezeta:strong:irreducibility}  implies  the statement of
the Proposition.
\end{proof}

We would like to derive from the above Proposition~\ref{prop:strongly:irreducible} a generalization
of Proposition~\ref{prop:only:E} to arbitrary finite covers of the Teichm\"uller curve $\hat \cT$.
The proof of that Proposition can be in fact generalized after we have established the following
algebraic lemma.

\begin{lemma}
\label{lemma:irrational}
The matrix $C$ in formula~\eqref{eq:matrixC} has a simple eigenvalue $\mu\in \C$ of modulus
one which is not a root of unity.
\end{lemma}
\begin{proof}  Let $P_\zeta(T) = T^3-2\zeta T^2 +2T -\zeta$ be the factor
of the characteristic polynomial of the matrix $C$, written in formula~\eqref{eq:char:pol:C}.
Since $\zeta^3=1$ the relation $\overline{ P_\zeta(T)}= P_\zeta( 1/\bar T)$ holds, hence
$P_\zeta(T)$ has exactly one root $\mu \in \C$ of modulus one (note that $P_\zeta(T)$
cannot have all the roots on the unit circle since the sum of all of its roots is equal to
$-2 \zeta$ which has modulus equal to $2$). We will compute the minimal polynomial
$M(T)$ (with integer coefficients) of $\mu$ and check that it is not a cyclotomic polynomial.
The general procedure to compute the minimal polynomial of the roots of
$P_\zeta(T)$ is to compute the resultant of $P_\zeta(T)$ and $\zeta^3-1$. In this
particular case, it can be done by hand as follows. Assume $P_\zeta(T)=0$, then
$T(T^2 +2) = \zeta (2 +T^2)$, hence
$$
 T^3(T^2+2)^3 = \zeta^3 (2 +T^2)^3
=(2 +T^2)^3.
$$
It follows then that $P_\zeta(T)$ is a divisor of the following polynomial with integer coefficients:
$$
Q(T):=T^9+ 6T^7-8T^6+12T^5 -12T^4 +8 T^3 -6T^2-1\,.
$$
The above polynomial factorizes as follows into irreducible factors:
$$
Q(T)= (T-1)(T^2-T+1)(T^6 + 2T^5 + 8T^4 + 5 T^3 + 8 T^2 +2T +1)\,.
$$
(The above factorization can be guessed by reduction modulo $2$. In fact, $Q(T) \equiv_2 T^9-1$
and $T^9-1\equiv_2 (T-1)(T^2-T+1)(T^6-T^3+1)$ and it is immediate to check that the factors
$T-1$, $T^2-T+1$ and $T^6-T^3+1$ are irreducible modulo $2$.)
Since $P_\zeta(T)$ and $(T-1)(T^2-T+1)$ have clearly no common roots, it follows that
$$
M(T)= T^6 + 2T^5 + 8T^4 + 5 T^3 + 8 T^2 +2T +1\,.
$$
The  polynomial  $M(T)$  is  not  cyclotomic.  In  fact,  it is known
(see~\cite{Migotti})  that for all positive integers $n$ with at most
two  distinct odd prime factors, the $n$-th cyclotomic polynomial has
all  the coefficients in $\{0, 1, -1\}$. It is also known that if $n$
has  $r$  distinct  odd  prime factors then $2^r$ is a divisor of the
degree  of  the  $n$-th  cyclotomic polynomial, which is equal to the
value  $\phi(n)$  of the Euler's $\phi$-function. It follows that all
cyclotomic  polynomials  of degree $6$ (which in fact appear only for
$n=7, 9, 14$ and $18$) have all the coefficients in $\{0, 1, -1\}$.
\end{proof}

\begin{Proposition}
\label{prop:only:E:bis}
The  complex  Hodge  bundle  $H^1_{\C{}}$ over any finite (possibly ramified)
cover of the Teichm\"uller curve $\hat\cT$  has no nontrivial $\PSL$-invariant
complex subbundles other than the lifts of the subbundles $\cE(\zeta)$ and $\cE(\zeta^2)$.
\end{Proposition}
\begin{proof}
By Proposition~\ref{prop:strongly:irreducible} the lifts of the bundles $\cE(\zeta)$ and $\cE(\zeta^2)$
to any finite (possibly ramified) cover of the Teichm\"uller curve $\hat\cT$  do  not  have  any
non-trivial $\PSL$-invariant subbundles. By the same argument as in the proof of Proposition~\ref{prop:only:E}, the proof can then be reduced to prove that there is no $\PSL$-invariant
isomorphism between  the lifts of the subbundles $\cE(\zeta)$ and $\cE(\zeta^2)$, which are
complex conjugate subbundles of the Hodge bundle.  By Lemma~\ref{lemma:irrational},
the monodromy matrix $C=BA$ of formula~\eqref{eq:matrixC} has a (simple)  complex
eigenvalue $\mu\in \C$ of modulus $1$ which is not a root of unity. It follows that any power of
$C$ has a non-real eigenvalue of modulus $1$, hence in particular the spectrum of any
power of $C$ is different from the spectrum of its complex conjugate. Thus for any (possibly
ramified) finite cover of the Teichm\"uller curve $\hat\cT$, the monodromy representations
on the lift of the bundles $\cE(\zeta)$ and $\cE(\zeta^2)$ are not isomorphic. In fact, for
any finite cover of the $\hat \cT$, there exists a path with monodromy representation on
the lifts of $\cE(\zeta)$ and of $\cE(\zeta^2)$ given by a power $C^k$ of $C$ and
by its complex conjugate  $\bar C^k$ respectively, which have different spectrum and thus
are not isomorphic.
\end{proof}

By the same argument as in the proof of Corollary~\ref{cor:no:R:subspaces}, this
time based on Proposition~\ref{prop:only:E:bis} (instead of Proposition~\ref{prop:only:E}) we can prove
that the real Hodge bundle is strongly irreducible.

\begin{Corollary}
\label{cor:no:R:subspaces:bis}
The  real  Hodge  bundle  $H^1_{\R{}}$ over any finite (possibly ramified)
cover of the Teichm\"uller curve $\hat\cT$ has no non-trivial
$\PSL$-invariant subbundles.
\end{Corollary}
   %

\section{Non-varying phenomenon for certain loci of cyclic covers}
\label{s:non:varying}

It  is  known  that  the  sum  of the Lyapunov exponents of the Hodge
bundle  along  the  Teichm\"uller  geodesic  flow is the same for all
$\SL$-invariant  suborbifold in any hyperelliptic locus in the moduli
space      of      Abelian      or      quadratic      differentials,
see~\cite{Eskin:Kontsevich:Zorich}.  In the paper~\cite{Chen:Moeller}
D.~Chen  and  M.~M\"oller  proved the conjecture of M.~Kontsevich and
one of the authors on non-varying of the sum of the positive Lyapunov
exponents  for  all  Teichm\"uller  curves  in  certain strata of low
genera.  We  show  that analogous non-varying phenomenon is valid for
certain loci of cyclic covers.

Let  $M$  be  a  flat surface in some stratum of Abelian or quadratic
differentials.  Together  with every closed regular geodesic $\gamma$
on $M$ we have a bunch of parallel closed regular geodesics filling a
maximal  cylinder $\mathit{cyl}$ having a conical singularity at each
of  the  two  boundary  components.  By  the  \textit{width} $w$ of a
cylinder  we  call  the  flat  length  of  each  of  the two boundary
components, and by the \textit{height} $h$ of a cylinder --- the flat
distance between the boundary components.

The  number of maximal cylinders filled with regular closed geodesics
of  bounded  length  $w(cyl)\le L$ is finite. Thus, for any $L>0$ the
following quantity is well-defined:

\begin{equation}
\label{eq:N:area}
N_{\mathit{area}}(M,L):=
\frac{1}{\Area(M)}
\sum_{\substack{
\mathit{cyl}\subset M\\
w(\mathit{cyl})<L}}
\Area(cyl)
\end{equation}
Note  that  in the above definition we do not assume that the area of
the  flat  surface  is equal to one. For a flat surface $M$ denote by
$M_{(1)}$  a  proportionally  rescaled  flat surface of area one. The
definition  of $N_{\mathit{area}}( M,L)$ immediately implies that for
any $L>0$ one has
\begin{equation}
\label{eq:area:rescaling}
N_{\mathit{area}}(M_{(1)},L)=
N_{\mathit{area}}\left(M,\sqrt{\Area(M)} L\right)\,.
\end{equation}

The following limit, when it exists,
\begin{equation}
\label{eq:SV:definition}
c_{\mathit{area}}(M):=
\lim_{L\to+\infty}\frac{N_{\mathit{area}}(M_{(1)},L)}{\pi L^2}
\end{equation}
is  called  the  \textit{Siegel---Veech  constant}. By a theorem of
H.~Masur  and  A.~Eskin~\cite{Eskin:Masur},  for  any $\PSL$-invariant
suborbifold  in  any  stratum  of meromorphic quadratic differentials
with  at  most simple poles, the limit does exist and is the same for
almost  all  points  of  the  suborbifold  (which  explains  the term
``constant'').         Moreover,         by         Theorem         3
from~\cite{Eskin:Kontsevich:Zorich},  in genus zero the Siegel--Veech
constant  $c_{\mathit{area}}(M)$  depends only on the ambient stratum
$\cQ(d_1,\dots,d_m)$ and:
\begin{equation}
\label{eq:carea:genus:0}
c_{\mathit{area}}(M)=
-\cfrac{1}{8\pi^2}\,\sum_{j=1}^m \cfrac{d_j(d_j+4)}{d_j+2}\,.
\end{equation}

Let  $S=(\CP,q)\in\cQ_1(n-5,-1^{n-1})$,  where $n\ge 4$. Suppose that
the  limit~\eqref{eq:SV:definition}  exists for $S$. Let $p:\hat C\to
\CP$ be a ramified cyclic cover
\begin{equation}
\label{eq:cyclic:cover:d:n}
w^d=(z-z_1)\cdots(z-z_n)
\end{equation}
with ramification points exactly at the singularities of $q$. Suppose
that  $d$  divides  $n$,  and  that  $d>2$. Let us consider the induced flat
surface  $\hat  S:=(\hat  C,  p^\ast  q)$. The Lemma below mimics the
analogous Lemma in the original paper~\cite{Eskin:Kontsevich:Zorich}.

\begin{Lemma}
\label{lm:SVconst:for:cover}
The  Siegel--Veech  constants of the two flat surfaces are related as
follows:
\begin{equation}
\label{eq:c:area:hat:equals:2:c:area}
c_{\mathit{area}}(\hat S)=
\begin{cases}
\frac{1}{d}\cdot c_{\mathit{area}}(S),&\text{when $d$ is odd}\\
\frac{4}{d}\cdot c_{\mathit{area}}(S),&\text{when $d$ is even.}
\end{cases}
\end{equation}
\end{Lemma}
\begin{proof}
Let us consider  any  maximal  cylinder $cyl$ on the underlying flat surface
$S$.  By  maximality of the cylinder, each of the boundary components
contains  at least one singularity of $q$. Since $S$ is a topological
sphere,  the two boundary components of $cyl$ do not intersect. Since
$q$  has  a  single  zero,  this  zero belongs to only one of the two
components of $cyl$. Since the other component contain only poles, it
contains exactly two poles.

Each of these two poles is a ramification point of $p$ of degree $d$.
Thus,  any  closed geodesic (waist curve) of the cylinder $cyl$ lifts
to  a single closed geodesic of the length $d$ when $d$ is odd and to
two distinct closed geodesics of the lengths $d/2$ when $d$ is even.

Now  note  that, since $d>2$ the quadratic differential $p^\ast q$ is
holomorphic.  The  condition that $d$ divides $n$ implies that $p$ is
non-ramified   at  infinity.  At  each  of  the  ramification  points
$z_1,\dots,z_n$  the  quadratic  differential $q$ has a zero or pole,
and  it  has no other singularities on $\CP$. Hence, the (nontrivial)
zeroes  of  $p^\ast  q$  are  exactly  the  preimages  of  the points
$z_1,\dots,z_n$,  and,  hence,  any  maximal  cylinder  on  $\hat  S$
projects to a maximal cylinder on $S$.

Note  also  that  since  $S$  has area one, $\hat S$ has area $d$. We
consider separately two different cases.
\smallskip

\textbf{Case when $\mathbf{d}$ is odd.}

Applying~\eqref{eq:area:rescaling}        followed       by       the
definition~\eqref{eq:N:area}  and  then followed by our remark on the
relation  between the corresponding maximal cylinders $\widehat{cyl}$
and $cyl$ we get the following sequence of relations:
\begin{multline*}
N_{\mathit{area}}(\hat S_{(1)},\sqrt{d}\cdot L)
  =
N_{\mathit{area}}(\hat S,d\cdot L)
  =\\=
\sum_{\substack{
\widehat{cyl}\subset \hat S\\
w(\widehat{cyl})<d\cdot L}}
\frac{\Area(\widehat{cyl})}{\Area(\hat S)}
  =
\sum_{\substack{
cyl\subset S\\
w(cyl)<L}}
\frac{\Area(cyl)}{\Area(S)}
  =
N_{\mathit{area}}(S,L)\,.
\end{multline*}

Hence,
\begin{multline*}
c_{\mathit{area}}(\hat S_{(1)})=
\lim_{R\to+\infty}
\frac{N_{\mathit{area}}(\hat S_{(1)},R)}{\pi R^2}=
\lim_{L\to+\infty}
\frac{N_{\mathit{area}}(\hat S_{(1)},\sqrt{d}L)}{\pi\cdot d\cdot L^2}
=\\=
\frac{1}{d}\lim_{L\to+\infty}
\frac{N_{\mathit{area}}(S,L)}{\pi L^2}=
\frac{1}{d}\cdot c_{\mathit{area}}(S)\,,
\end{multline*}
where we used the substitution $R:=\sqrt{d} L$.
\smallskip

\textbf{Case when $\mathbf{d}$ is even.}

These time our relations are slightly modified due to the fact that a
preimage  of  a  maximal  cylinder downstairs having a waste curve of
length  $\ell$  is a disjoint union of two maximal cylinders with the
waste curves of length $d\cdot\ell/2$.
\begin{multline*}
N_{\mathit{area}}(\hat S_{(1)},\frac{\sqrt{d}}{2} L)
  =
N_{\mathit{area}}(\hat S,\frac{d}{2} L)
  =\\=
\sum_{\substack{
\widehat{cyl}\subset \hat S\\
w(\widehat{cyl})<\frac{d}{2}\cdot L}}
\frac{\Area(\widehat{cyl})}{\Area(\hat S)}
  =
\sum_{\substack{
cyl\subset S\\
w(cyl)<L}}
\frac{\Area(cyl)}{\Area(S)}
  =
N_{\mathit{area}}(S,L)\,.
\end{multline*}

Hence,
\begin{multline*}
c_{\mathit{area}}(\hat S_{(1)})=
\lim_{R\to+\infty}
\frac{N_{\mathit{area}}(\hat S_{(1)},R)}{\pi R^2}=
\lim_{L\to+\infty}
\frac{N_{\mathit{area}}(\hat S_{(1)},\frac{\sqrt{d}}{2}L)}{\pi\cdot \frac{d}{4}\cdot L^2}
=\\=
\frac{4}{d}\lim_{L\to+\infty}
\frac{N_{\mathit{area}}(S,L)}{\pi L^2}=
\frac{4}{d}\cdot c_{\mathit{area}}(S)\,,
\end{multline*}
where we used the substitution $R:=\frac{\sqrt{d}}{2} L$.
\end{proof}

\begin{Proposition}
Under the assumptions of Lemma~\ref{lm:SVconst:for:cover} one gets
\begin{equation}
\label{eq:c:area:hat:answer}
\frac{\pi^2}{3}\cdot c_{\mathit{area}}(\hat S_{(1)})=
\frac{k}{12\cdot d}\cdot\frac{(n-1)(n-2)}{n-3}
\,,\ \text{where }
k=
\begin{cases}
1,&\text{when $d$ is odd}\\
4,&\text{when $d$ is even.}
\end{cases}
\end{equation}
\end{Proposition}
\begin{proof}
By applying  the formula~\eqref{eq:carea:genus:0} for the Siegel---Veech
constant   of   any   $\PSL$-invariant   suborbifold   in  a  stratum
$\cQ_1(d_1,\dots,d_m)$  in  genus  zero  to  a particular case of the
stratum $\cQ_1(n-5,-1^{n-1})$, we get
$$
\frac{\pi^2}{3}\cdot c_{\mathit{area}}(S)=
\frac{1}{24}\left(3(n-1)-\frac{(n-5)(n-1)}{n-3}\right)=
\frac{1}{12}\frac{(n-1)(n-2)}{n-3}\,.
$$
By applying Lemma~\ref{lm:SVconst:for:cover} we complete the proof.
\end{proof}

Let $\cM_1$ be a $\PSL$-invariant suborbifold
in the stratum $\cQ_1(n-5,-1^{n-1})$.
It is immediate to check that
the locus
$\hat \cM_1$  of flat surfaces $\hat S_{(1)}$
induced by cyclic covers~\eqref{eq:cyclic:cover:d:n}, where $d$ divides $n$,
belongs to the stratum
\begin{align*}
\cQ_1(d(n-3)-2,(d-2)^{n-1}),&\ \text{ when $d$ is odd}\\
\cH_1\left(d(n-3)/2-1,(d/2-1)^{n-1}\right),&\ \text{ when $d$ is even}
\end{align*}

Applying  Theorems 1 and 2 from~\cite{Eskin:Kontsevich:Zorich} we get
the following

\begin{Proposition}
\label{prop:5:2}
The sum of the Lyapunov exponents of the Hodge bundle $H^1$ over
$\cM_1$ is equal to
\begin{equation}
\label{eq:sum:lambda}
\lambda_1+\dots+\lambda_g=
\begin{cases}
\cfrac{\left(d^2-1\right) (n-2)}{12 d} &\text{ when $d$ is odd}\\
 \\
%
\cfrac{(n-2) \left(d^2 (n-3)+2 n\right)}{12 d (n-3)} &\text{ when $d$ is even}\,.
\end{cases}
\end{equation}
\end{Proposition}

Consider the particular case, when $d=3$ and $n=3m$.
Then
$$
\lambda_1+\dots+\lambda_g=\frac{2n-4}{9}\,,
$$
were, $g=n-2$, by Riemann---Hurwitz formula.

Note       that       $H^1=\cE(\zeta)\oplus\cE(\zeta^2)$,       where
by~\cite{McMullen}  the  restriction of the Hodge form to $\cE(\zeta)$
has  signature  $(m-1,2m-1)$ and the restriction of the Hodge form to
$\cE(\zeta^2)$ has signature $(2m-1,m-1)$. Thus, each of the
subspaces has $m$ zero exponents.

\appendix

\section{Lyapunov spectrum of pseudo-unitary cocycles: the proofs}
\label{a:Lyapunov:spectrum:of:pseudo-unitary:cocycles}
In this appendix we prove Theorem~\ref{th:spec:of:unitary:cocycle}. Its presentation below is inspired by discussions of the second author with A. Avila and J.-C. Yoccoz.

Recall  that  we  consider  an  invertible  transformation $T$ or a flow
$T_t$
preserving  a  finite  ergodic measure $\mu$ on a locally compact topological space
$M$. Let $U$ be a $\log$-integrable cocycle over this transformation (flow)
with  values  in  the  group  $\Upq$  of pseudo-unitary matrices. The
Oseledets  Theorem  (i.e.  the multiplicative ergodic theorem) can be
applied to complex cocycles. Denote by
\begin{equation}
\label{eq:Lyapunov:spectrum}
\lambda_1\ge\dots\ge\lambda_{p+q}
\end{equation}
the Lyapunov spectrum of the pseudo-unitary cocycle $U$. Let
\begin{equation}
\label{eq:Lyapunov:spectrum:without:multiplicities}
\lambda_{(1)} >  \dots > \lambda_{(s)}
\end{equation}
be  all \textit{distinct} Lyapunov exponents from the above spectrum.
By applying  the  transformation  (respectively,  the  flow)  both  in
forward  and  backward  directions, we get the corresponding Oseledets
decomposition
\begin{equation}
\label{eq:Oseledets:direct:sum}
E_{\lambda_{(1)}}\oplus\dots\oplus E_{\lambda_{(s)}}
\end{equation}
at  $\mu$-almost every point of the base space $M$. By definition all
nonzero  vectors  of each subspace $E_{\lambda_{(k)}}$ share the same
Lyapunov  exponent  $\lambda_{(k)}$ which changes sign under the time
reversing.

\begin{Lemma}
\label{l:symp-orth}
For   any  nonzero $\lambda_{(k)}$,  the  subspace
$E_{\lambda_{(k)}}$      of      the     Oseledets     direct     sum
decomposition~\eqref{eq:Oseledets:direct:sum}  is isotropic. Any two
subspaces  $E_{\lambda_{(i)}}$,  $E_{\lambda_{(j)}}$  such  that
$\lambda_{(j)}\neq-\lambda_{(i)}$  are  orthogonal with respect to the
pseudo-Hermitian  form.
\end{Lemma}
\begin{proof}
Consider   a  (measurable family of) norm(s) $\|.\|$ for  which  the  cocycle  $U$  is $\log$-integrable.
By Luzin's theorem, the absolute value of the (measurable family of) pseudo-Hermitian product(s)
$\langle.,.\rangle$ of any two  vectors  $v_1,v_2$  in  $\C^{p+q}$  is  uniformly  bounded on any
compact  set $\cK$ of positive measure in $M$ by the product of their norms,
   %
   %
$$
|\langle v_1,v_2\rangle|_x\le
\const(\cK)\cdot \|v_1\|_x\cdot\|v_2\|_x
\quad
\text{ for any }x\in\cK\,,
$$
up  to  a multiplicative constant $\const(\cK)$ depending only on the
norm  and  on the compact set $\cK$. By ergodicity of the transformation (flow),  the  trajectory  of  almost any point returns infinitely
often to the compact set $\cK$.

Suppose  that  there is a pair of Lyapunov exponents $\lambda_{(i)}$,
$\lambda_{(j)}$  satisfying  $\lambda_{(i)}\neq-\lambda_{(j)}$. We do
not exclude the case when $i=j$. Consider a pair of vectors $v_i,v_j$
such  that $v_i\in E_{\lambda_{(i)}}$, $v_j\in E_{\lambda_{(j)}}$. By
definition of $E_{\lambda_{(i)}}$, we have
$$
\|T_t(v_1)\|_x\cdot\|T_t(v_2)\|_x\sim
\exp\big((\lambda_{(i)}+\lambda_{(j)})\, t\big)\,.
$$
When  $\lambda_{(i)}+\lambda_{(j)}<0$  the latter expression tends to
zero  when  $t\to+\infty$;  when  $\lambda_{(i)}+\lambda_{(j)}>0$ the
latter expression tends to zero when $t\to-\infty$. In both cases, we
conclude  that  for a subsequence of positive or negative times $t_k$
(chosen  when  the  trajectory  visits  the  compact  set  $\cK$) the
pseudo-Hermitian  product  $\langle T_{t_k}(v_1),T_{t_k}(v_2)\rangle$
tends to zero. Since the pseudo-Hermitian product is preserved by the
flow,  this  implies that it is equal to zero, so $\langle
v_1,v_2\rangle=0$.   Thus,   we   have  proved  that  every  subspace
$E_{(\lambda_i)}$,  except  possibly $E^\mu_{(0)}$, is isotropic, and
that  any  pair of subspaces $E_{\lambda_{(i)}}$, $E_{\lambda_{(j)}}$
such   that   $\lambda_{(j)}\neq-\lambda_{(i)}$  is  orthogonal  with
respect to the pseudo-Hermitian form.
\end{proof}

We  proceed with the following elementary linear algebraic fact about
isotropic subspaces of a pseudo-Hermitian form of signature $(p,q)$.

\begin{lemma}
\label{lemma:nullcone}
The  dimension  $\dim_\C{}  V$  of  an  isotropic  subspace  $V$ of a
pseudo-Hermitian  form  of  signature  $(p,q)$  is  bounded  above by
$\min(p,q)$.
\end{lemma}
\begin{proof}
By  choosing  an  appropriate  basis,  we  can  always  suppose  that
$\langle\vec a,\vec b\rangle$ has the form
$$
\langle\vec a,\vec b\rangle=a^1\bar b^1+\dots+a^p\bar b^p -
      a^{p+1}\bar b^{p+1}-\dots-a^{p+q}\bar b^{p+q}
$$
where  $\vec  a=(a^1,\dots,a^{p+q}),  \vec b=(b^1,\dots,b^{p+q})$ and
$\vec a, \vec b\in \C^{p+q}$.

Without loss of generality we can assume that $p\leq q$. Let $\Sigma$
be  the  null  cone  of  the  pseudo-Hermitian  form, $\Sigma:=\{\vec
a\in\mathbb{C}^{p+q}\,|\,  \langle\vec a,\vec a\rangle=0\}$. We argue
by  contradiction. Suppose that $V\subset\Sigma$ is a vector subspace
of  dimension  $r$ with $r\geq p+1$. By assumption, we can find $p+1$
linearly  independent  vectors  $\vec  v_1,\dots,\vec  v_{p+1}\in V$.
By using  the  first  $p$  coordinates  of  these  vectors,  we obtain a
collection  of  $p+1$  vectors $\vec w_i=(v_i^1,\dots,v_i^p)\in\C^p$,
$1\leq i\leq p+1$. Thus, one can find a non-trivial linear relation
$$
t_1\vec w_1+\dots+t_{p+1}\vec w_{p+1}=\vec 0\in\C^p\,.
$$
Going   back  to  the  vectors  $\vec  v_i$,  we  conclude  that  the
non-trivial linear combination
$$
\vec v=t_1\vec v_1+\dots+t_{p+1}\vec v_{p+1}
   \in V-\{0\}\subset\Sigma-\{0\}
$$
has  the form $\vec v=(0,\dots,0,v^{p+1},\dots,v^{p+q})$, which leads
to  a  contradiction  since  the  inclusion  $\vec v\in\Sigma$ forces
$0=|v^{p+1}|^2+\dots+|v^{p+q}|^2$ (that is, $\vec v=0$).
\end{proof}

\begin{Lemma}
\label{lm:spectrum:is:symmetric}
The  Lyapunov spectrum~\eqref{eq:Lyapunov:spectrum} is symmetric with
respect to the sign change, that is for any $k$ satisfying $1\le k\le
p+q$ one has
$$
\lambda_k=\lambda_{p+q+1-k}
$$
\end{Lemma}
\begin{proof}
First  note  that together with any nonzero entry $\lambda_{(i)}$ the
spectrum~\eqref{eq:Lyapunov:spectrum:without:multiplicities}
necessarily   contains  the  entry  $-\lambda_{(i)}$.  Otherwise,  by
Lemma~\ref{l:symp-orth}  the  subspace  $E_{\lambda_{(i)}}$  would be
orthogonal  to  the entire vector space $\C^{p+q}$, which contradicts
the assumption that the pseudo-Hermitian form is nondegenerate.

Consider     a     nonzero     entry     $\lambda_{(i)}$    in    the
spectrum~\eqref{eq:Lyapunov:spectrum:without:multiplicities}.  Let us
decompose  the  direct  sum~\eqref{eq:Oseledets:direct:sum}  into two
terms.   As   the   first  term  we  choose  $E_{\lambda_{(i)}}\oplus
E_{-\lambda_{(i)}}$,   and   we   place   all   the   other  summands
from~\eqref{eq:Oseledets:direct:sum}   to   the   second   term.   By
Lemma~\ref{l:symp-orth} the two terms of the resulting direct sum are
orthogonal.  Hence,  the  restriction of the pseudo-Hermitian form to
the  first  term  is  non-degenerate. By Lemma~\ref{l:symp-orth} both
subspaces $E_{\lambda_{(i)}}$ and $E_{-\lambda_{(i)}}$ are isotropic.
It  follows now from Lemma~\ref{lemma:nullcone} that their dimensions
coincide.
\end{proof}

\begin{Lemma}
\label{lm:at:least:p:minus:q}
The  dimension  of  the  neutral  subspace  $E_0$  in  the  Oseledets
decomposition~\eqref{eq:Oseledets:direct:sum} is at least $|p-q|$.
\end{Lemma}
\begin{proof}
Consider  the  direct  sum  $E_u$  of  all subspaces in the Oseledets
decomposition~\eqref{eq:Oseledets:direct:sum}     corresponding    to
strictly positive Lyapunov exponents $\lambda_{(i)}> 0$,
$$
E_u:=\bigoplus_{\lambda_{(i)}>0} E_{\lambda_{(i)}}\,.
$$
Similarly,  consider  the  direct  sum  $E_s$ of all subspaces in the
Oseledets decomposition~\eqref{eq:Oseledets:direct:sum} corresponding
to strictly negative Lyapunov exponents $\lambda_{(j)}< 0$,
$$
E_s:=\bigoplus_{\lambda_{(j)}<0} E_{\lambda_{(j)}}\,.
$$
By   Lemma~\ref{l:symp-orth}  both  subspaces  $E_u$  and  $E_s$  are
isotropic. Hence, by Lemma~\ref{lemma:nullcone} the dimension of each
of  them  is  at  most  $\min(p,q)$. Since the dimension of the space
$E_u\oplus  E_0\oplus E_s$ is $p+q$, it follows that the dimension of
the neutral subspace $E_0$ (when it is present) is at least $|p-q|$.
\end{proof}

By combining  the statements of Lemma~\ref{lm:spectrum:is:symmetric} and
of   Lemma~\ref{lm:at:least:p:minus:q},  we   get  the  statement  of
Theorem~\ref{th:spec:of:unitary:cocycle}.

Concluding  this appendix, we show the following simple criterion for
the cocycle $U$ to act by isometries on $E_0$.

\begin{Lemma}\label{lemma:isometryUpq}
Suppose  that  the  neutral  subspace (subbundle) $E_0$  has  dimension  exactly
$|p-q|$. Then, the cocycle $U$ acts on $E_0$ by isometries in the sense that the restriction of the pseudo-Hermitian form to the
neutral  subspace  (subbundle)  $E_0$  is either positive definite or
negative definite.
\end{Lemma}

\begin{proof}
We  claim  that  $E_0\cap\Sigma=\{0\}$,  where  $\Sigma=\{v:  \langle
v,v\rangle=0\}$   is  the  null-cone  of  the  pseudo-Hermitian  form
$\langle.,.\rangle$  preserved  by $U$. Indeed, since $E_s$ and $E_u$
have  the  same  dimension (by Lemma~\ref{lm:spectrum:is:symmetric}),
and  $E_0$  has  dimension  $|p-q|$  (by  hypothesis),  we  have that
$\dim_\C       E_s=\dim_\C       E_u=\min\{p,q\}$.       So,       if
$E_0\cap\Sigma\neq\{0\}$,    the    arguments   of   the   proof   of
Lemma~\ref{l:symp-orth}  show  that $E_s\oplus (E_0\cap\Sigma)$ is an
isotropic subspace whose dimension is at least $\min\{p,q\}+1$, which
contradicts Lemma~\ref{lemma:nullcone}.

Since    the    pseudo-Hermitian    form    $\langle.,.\rangle$    is
non-degenerate,  the fact that $E_0\cap\Sigma=\{0\}$ implies that the
restriction of $\langle.,.\rangle$ to $E_0$ is (positive or negative)
definite.  In  other  words,  the  cocycle  $U$  restricted  to $E_0$
preserves  a  family  of  definite  forms $\langle.,.\rangle|_{E_0}$,
i.e., $U$ acts by isometries on $E_0$.
\end{proof}

\section{Evaluation of the monodromy representation}
\label{a:matrix:calculation}

\subsection{Scheme of the construction.}
Our  plan  is  as  follows. We start by constructing the square-tiled
cyclic  cover  $\hat  S=\hat S_3$ of the initial square-tiled surface
$S$   of  Figure~\ref{fig:oneline:6}.   Then   we   construct  the
$\PSLZ$-orbit  of  $\hat S=\hat S_3$. The results of this calculation
are  presented  in  Figure~\ref{fig:PSL2Z:orbit}.  In particular, the
$\PSLZ$-orbit  of  the initial square-tiled surface $\hat S=\hat S_3$
has  cardinality three, see Figure~\ref{fig:PSL2Z:orbit}.

For each of the three square-tiled surfaces $\hat S_1, \hat S_2, \hat
S_3$  in  the $\PSLZ$-orbit of $\hat S_3$ we construct an appropriate
generating  set  of  integer  cycles  and  a  basis of the eigenspace
$E_{\hat  S_i}(\zeta)\subset H_1(\hat S_i,\C{})$. Then we compute the
six  matrices  of  the  action  in  homology  induced  by  the  basic
horizontal  shear  $h$  and  by  the counterclockwise rotation $r$ by
$\pi/2$ of these flat surfaces, where
$$
h=\begin{pmatrix}1&1\\0&1\end{pmatrix}
\qquad
r=\begin{pmatrix}0&-1\\1&0\end{pmatrix}\,,
$$
thus    obtaining   an   explicit   description   of   the   holonomy
representation.  Note  that  we  work with the \textit{homology}; the
representation in the \textit{cohomology} is dual.

\begin{NNRemark}
Since we consider the representation of $\PSLZ$ we might consider all
matrices up to multiplication by $-1$.
\end{NNRemark}

\subsection{Construction of homology bases and evaluation of induced homomorphisms}
\label{ss:Construction:of:homology:bases}

\paragraph{\textbf{Step 1 (Figure~\ref{fig:oneline:shear}).}}
As a generating set of cycles of $\hat S_3$ we take the cycles
$$
a_1, b_1, c_1, d_1, \dots, a_3, b_3, c_3, d_3
$$
represented    in   the    second   picture   from   the   left   in
Figure~\ref{fig:oneline:shear}.  Each  of these cycles is represented
by  a  close loop with a base point $C$. The loops $d_i$ are composed
from   the   subpaths  $d_{i,1}$,  and  $d_{i,2}$,  as  indicated  in
Figure~\ref{fig:oneline:shear}.

Consider  the  affine  map  $\hat  S_3\to  \hat  S_3$  induced by the
horizontal shear
$$
h=\begin{pmatrix}1&1\\0&1\end{pmatrix}\,,
$$
see  Figure~\ref{fig:oneline:shear}.      (Recall     that,     by
Convention~\ref{conv:hor:vert}  established  in  the  beginning  of
Section~\ref{ss:The:PSLZ:orbit},   the   notions   ``horizontal''  and
``vertical'' correspond to the ``landscape'' orientation.)

It  is clear from Figure~\ref{fig:oneline:shear} that the induced map
$h_3$ in the integer homology acts on the chosen cycles as follows:
\begin{align}
\label{eq:h3:abc}
h_3&: a_i\mapsto a'_i = b_{i-1}\notag\\
h_3&: b_i\mapsto b'_i = c_{i-1}\\
h_3&: c_i\mapsto c'_i = a_{i}\notag\,,
\end{align}
where  we  use  the  standard  convention that indices are considered
modulo $3$.

To  compute  the  images  of  the cycles $d_i$ we introduce auxiliary
relative  cycles  $e_1,e_2,e_3$;  see  the left edge of the rightmost
picture in Figure~\ref{fig:oneline:shear}. Then,
\begin{align*}
h_3: d_{i,1}\mapsto d'_{i,1} &= b_{i-1}+c_{i-1}-d_{i+1,2}+e_{i-1}\\
h_3: d_{i,2}\mapsto d'_{i,2} &= -e_{i-1}-d_{i+1,1}+a_{i+1}\,,
\end{align*}
and taking the sum $d_i=d_{i,1}+d_{i,2}$, we get
\begin{equation}
\label{eq:h3:d}
h_3: d_i\mapsto d'_i=a_{i+1}+b_{i-1}+c_{i-1}-d_{i+1}\,.
\end{equation}

The  induced  action  of  the  generator  $T$  of  the  group of deck
transformations,  defined  in~\eqref{eq:T}, has the following form on
the generating cycles:
\begin{align*}
T_\ast: a_i&\mapsto a_{i-1}\\
T_\ast: b_i&\mapsto b_{i-1}\\
T_\ast: c_i&\mapsto c_{i-1}\\
T_\ast: d_i&\mapsto d_{i-1}
\end{align*}
This  implies  that  the following elements of $H_1(\hat S_3,\C{})$:
\begin{align*}
a_+&:=a_1+\zeta a_2 + \zeta^2 a_3\\
b_+&:=b_1+\zeta b_2 + \zeta^2 b_3\\
c_+&:=c_1+\zeta c_2 + \zeta^2 c_3\\
d_+&:=d_1+\zeta d_2 + \zeta^2 d_3
\end{align*}
are   eigenvectors   of  $T_\ast$  corresponding  to  the  eigenvalue
$\zeta=\exp(2\pi  i/3)$,  and  hence,  they  belong  to  the subspace
$\cE_*(\zeta)$.  These elements are linearly independent, and, thus, form
a  basis  of  this  four-dimensional  subspace.  To verify the latter
statement  we  compute  the  intersection  numbers  of the generating
cycles  $a_i,  b_j,  c_k,  d_l$.  Using these intersection numbers we
evaluate the quadratic form
$$
\frac{i}{2} (\alpha\cdot\bar\beta)
$$
on   the  collection  $a_+,b_+,c_+,d_+$,  and  observe  that  it  has
signature   $(3,1)$.   We   skip   the  details  of  this  elementary
calculation.

By combining  the  definition  of  the  basis $a_+,b_+,c_+,d_+$ with the
transformation  rules~\eqref{eq:h3:abc}  and~\eqref{eq:h3:d}, we  see
that the matrix $A^{\mathit{hor}}_3$ of the induced map
$$
h_3: {\cE_*}_{\hat S_3}(\zeta)\to {\cE_*}_{\hat S_3}(\zeta)
$$
has the form
\begin{equation}
\label{eq:A3:hor}
A^{\mathit{hor}}_3=\left(
\begin{array}{cccc}
 0 & 0 & 1 & \zeta^2 \\
 \zeta & 0 & 0 & \zeta \\
 0 & \zeta & 0 & \zeta \\
 0 & 0 & 0 & -\zeta^2
\end{array}
\right)
\end{equation}

\paragraph{\textbf{Step   2  (Figure~\ref{fig:oneline:rotate}).}}
In  Step  1  we  used  the  horizontal cylinder decomposition for the
initial square-tiled surface $\hat S_3$, see the left two pictures in
Figure~\ref{fig:oneline:rotate}.   Here,   as  usual,  ``horizontal''
corresponds      to      the      landscape      orientation,     see
Convention~\ref{conv:hor:vert} in Section~\ref{ss:The:PSLZ:orbit}. In
the   bottom   two  pictures  of  Figure~\ref{fig:oneline:rotate}  we
construct a pattern of the same flat surface $\hat S_3$ corresponding
to  the  vertical  cylinder  decomposition.  Finally,  we  rotate the
resulting  pattern  by $\pi/2$ clockwise; see the right two pictures.
We  renumber the squares after the rotation. The resulting surface is
the surface $\hat S_2$. It inherits a collection of generating cycles
and the basis of the subspace $\cE_*(\zeta)$ from the surface $\hat S_3$.
By construction, the matrix $R_3$ of the induced map
$$
r_3: {\cE_*}_{\hat S_3}(\zeta)\to {\cE_*}_{\hat S_2}(\zeta)
$$
is  the  identity  matrix,  $r_3=\Id$  for our choice of the basis in
${\cE_*}_{\hat S_3}(\zeta)$ and in ${\cE_*}_{\hat S_2}(\zeta)$.

Note   that   by   construction,  the  points  of  the  $\PSLZ$-orbit
corresponding  to  the  surfaces  $\hat  S_3$  and $\hat S_2$ satisfy
$[\hat     S_2]=r^{-1}    [\hat    S_3]=r    [\hat    S_3]$,    where
$r=\begin{pmatrix}0&1\\-1&0\end{pmatrix}\in\PSLZ$.

\paragraph{\textbf{Step 3 (Figure~\ref{fig:twoline:one:shear}).}}
We  consider  the  affine  map  $\hat S_2\to \hat S_1$ induced by the
horizontal      (in      the     landscape     orientation)     shear
$h=\begin{pmatrix}1&1\\0&1\end{pmatrix}$  and define a generating set
of  cycles  in the homology $H_1(\hat S_1; \Z{})$ of $\hat S_1$ and a
basis  of  cycles in the subspace ${\cE_*}_{\hat S_1}(\zeta)$ as the images
of  generating cycles previously defined in $H_1(\hat S_2; \Z{})$. By
construction, the matrix $A^{\mathit{hor}}_2$ of the induced map
$$
h_2: {\cE_*}_{\hat S_2}(\zeta)\to {\cE_*}_{\hat S_1}(\zeta)
$$
is  the  identity matrix, $A^{\mathit{hor}}_2=\Id$, for our choice of
the bases in ${\cE_*}_{\hat S_2}(\zeta)$ and in ${\cE_*}_{\hat S_1}(\zeta)$.

\paragraph{\textbf{Step 4 (Figure~\ref{fig:twoline:shear}).}}
Consider  the  affine  map  $\hat  S_1\to  \hat  S_2$  induced by the
horizontal  shear  $h=\begin{pmatrix}1&1\\0&1\end{pmatrix}$.  Now
both  homology spaces $H_1(\hat S_1; \Z{})$ and $H_1(\hat S_2; \Z{})$
are  already  endowed with the generating sets, so we can compute the
matrix $A^{\mathit{hor}}_1$ of the induced map
$$
h_1: {\cE_*}_{\hat S_1}(\zeta)\to {\cE_*}_{\hat S_2}(\zeta)\,.
$$
For  our choice of the basis in ${\cE_*}_{\hat S_1}(\zeta)$ and in ${\cE_*}_{\hat
S_2}(\zeta)$  the  matrix  $A^{\mathit{hor}}_1$  coincides  with  the
matrix   of   the   automorphism   ${\cE_*}_{\hat   S_2}(\zeta)\to  {\cE_*}_{\hat
S_2}(\zeta)$  induced  by the affine diffeomorphism $\hat S_2\to \hat
S_2$       corresponding       to      the      horizontal      shear
$h^2=\begin{pmatrix}1&2\\0&1\end{pmatrix}$.
Figure~\ref{fig:twoline:shear} describes this automorphism.

It is convenient to introduce auxiliary cycles $s_1, s_2, s_3$,
as  in  the  right  picture of Figure~\ref{fig:twoline:shear}. We use
Figure~\ref{fig:twoline:shear} to trace how the induced map $h_1\circ
h_2$  in the integer homology acts on the generating cycles. We start
by noting that
$$
h_1\circ h_2: d_i\mapsto d'_i = d_i\,.
$$
Next, we remark that,
\begin{align*}
h_1\circ h_2&: a_{i,1}\mapsto a'_{i,1} = -s_i+c_{i,1}\\
h_1\circ h_2&: a_{i,2}\mapsto a'_{i,2} = c_{i,2}+s_{i+1}\,,
\end{align*}
and, hence, taking the sum, we get
$$
h_1\circ h_2: a_i\mapsto a'_i = c_i+s_{i+1}-s_i\,.
$$
We proceed with the relations
\begin{align*}
h_1\circ h_2&: b_{i,1}\mapsto b'_{i,1} = s_{i-1}+b_{i+1,1}\notag\\
h_1\circ h_2&: b_{i,2}\mapsto b'_{i,2} = b_{i+1,2}-s_i\,,
\end{align*}
and, hence, taking the sum, we get
$$
h_1\circ h_2: b_i\mapsto b'_i = b_{i+1}+s_{i-1}-s_i\,.
$$
We conclude with the relations
\begin{align*}
h_1\circ h_2&: c_{i,1}\mapsto c'_{i,1} = -d_{i+1}+a_{i+1,1}\notag\\
h_1\circ h_2&: c_{i,2}\mapsto c'_{i,2} = a_{i+1,2}+d_{i-1}\,,
\end{align*}
and, hence, taking the sum, we get
$$
h_1\circ h_2: c_i\mapsto c'_i = a_{i+1}-d_{i+1}+d_{i-1}\,.
$$

In  order  to  express  the  auxiliary  cycles  $s_i$ in terms of the
generating  cycles  $a_i,  b_j,  c_k,  d_l$  it  is convenient
introduce   the  relative  cycle  $\vec{j}_b$  following  the  bottom
horizontal  (in  the landscape orientation) side of the square number
$j$ from left to right. In these notations
\begin{align*}
s_1&=c_{1,1}+\vec{0}_b+\vec{1}_b+a_{3,2}\\
d_3&=a_{3,1}+\vec{8}_b+\vec{9}_b+c_{1,2}\,.
\end{align*}
Adding  up  the  latter  equations and taking into consideration that
$\vec{0}_b=-\vec{9}_b$ and $\vec{1}_b=-\vec{8}_b$ we obtain
$$
s_1+d_3=c_1+a_3
$$
Analogous considerations show that
\begin{equation}
\label{eq:aux:cycles:s}
s_i=a_{i-1}+c_i-d_{i-1}\,.
\end{equation}
Summarizing  the  above relations we get
\begin{align*}
h_1\circ h_2&: a_i\mapsto a'_i = a_i-a_{i-1}+c_{i+1}-d_i+d_{i-1}\\
h_1\circ h_2&: b_i\mapsto b'_i =
             -a_{i-1}+a_{i+1}+b_{i+1}+c_{i-1}-c_i+d_{i-1}-d_{i+1}\\
h_1\circ h_2&: c_i\mapsto c'_i = a_{i+1}-d_{i+1}+d_{i-1}\\
h_1\circ h_2&: d_i\mapsto d'_i = d_i
\end{align*}
which  implies  the  following  expression  for $A^{\mathit{hor}}_1$:
$$
   %
A^{\mathit{hor}}_1=\left(
\begin{array}{cccc}
 1-\zeta & \zeta^2-\zeta & \zeta^2 & 0 \\
 0 & \zeta^2 & 0 & 0 \\
 \zeta^2 & \zeta-1 & 0 & 0 \\
 \zeta-1 & \zeta-\zeta^2 & \zeta-\zeta^2 & 1
\end{array}
\right)\,.
$$

\paragraph{\textbf{Step 5 (Figure~\ref{fig:twoline:rotate}).}}
We compute the action in homology of $\hat S_1$ induced by the
automorphism   of  $\hat  S_1$  associated  to  the  counterclockwise
rotation   $r=\begin{pmatrix}0&-1\\1&0\end{pmatrix}$   by  the  angle
$\pi/2$.  More  specifically,  we want to compute the matrix $R_1$ of
the induced map
$$
r_1: {\cE_*}_{\hat S_1}(\zeta)\to {\cE_*}_{\hat S_1}(\zeta)\,.
$$
in the chosen basis.

In a way similar to  the  calculation in Step 4, we use the auxiliary cycles
$s_i$       indicated      in      the      left      picture      of
Figure~\ref{fig:twoline:rotate}. Note that the cycles $a_i, b_j, c_k,
d_l,  s_m$  on  the  surface  $S_1$  are defined as the images of the
corresponding  cycles  on  the  surface  $S_2$. This implies that the
auxiliary       cycles       $s_i$       satisfy       the       same
relation~\eqref{eq:aux:cycles:s} as on the surface $S_2$.

We note that
\begin{align*}
r_1&: a_{i,1}\mapsto a'_{i,1} = s_i+a_{i,1}\\
r_1&: a_{i,2}\mapsto a'_{i,2} = a_{i,2}-s_{i+1}\,,
\end{align*}
and,      hence,      taking      the      sum,      and     applying
relations~~\eqref{eq:aux:cycles:s} we get
$$
r_1: a_i\mapsto a'_i= a_i+s_i-s_{i+1}=
a_{i-1}+c_i-c_{i+1}+d_i-d_{i-1}
\,.
$$
We proceed with the relations
\begin{align*}
r_1&: b_{i,1}\mapsto b'_{i,1} = d_{i+1}+c_{i,1}\\
r_1&: b_{i,2}\mapsto b'_{i,2} = c_{i,2}-d_{i-1}\,,
\end{align*}
and, hence, taking the sum, we get
$$
r_1: b_i\mapsto b'_i= c_i+d_{i+1}-d_{i-1}\,.
$$
We proceed further with the relations
\begin{align*}
r_1&: c_{i,1}\mapsto c'_{i,1} = -s_{i-1}+b_{i,1}\\
r_1&: c_{i,2}\mapsto c'_{i,2} = b_{i,2}+s_i\,,
\end{align*}
and, hence, taking the sum, we get
$$
r_1: c_i\mapsto c'_i= b_i+s_i-s_{i-1}
=a_{i-1}-a_{i+1}+b_i+c_i-c_{i-1}+d_{i+1}-d_{i-1}\,.
$$

To  establish  the  relations  for  the images of the cycles $d_i$ we
introduce  relative cycle $\vec j_t$ following from left to right the
top  horizontal edge of the square number $j$ in the \textit{initial}
enumeration  on  the left picture of Figure~\ref{fig:twoline:rotate}.
As  usual, ``horizontal'' is considered with respect to the landscape
orientation in Figure~\ref{fig:twoline:rotate}. In these notations we
get
\begin{align*}
r_1&: d_{1,1}\mapsto d'_{1,1} = b_{2,1}-\vec{4}_t\\
r_1&: d_{i,2}\mapsto d'_{i,2} = -\vec{17}_t+b_{2,2}+s_2\,.
\end{align*}
By adding  up  the  latter  equations and by taking into account that
$\vec{4}_t=-\vec{17}_t$ we obtain
$$
r_1: d_1\mapsto d'_1= b_2+s_2\,.
$$
Analogous considerations show that
$$
r_1: d_i\mapsto d'_i= b_{i+1}+s_{i+1}
=a_i+b_{i+1}+c_{i+1}-d_i\,.
$$

By applying   the   above   relations   to  the  images  of  the  cycles
$a_+,b_+,c_+,d_+$,  we  get  the  following  expression for the matrix
$R_1$:
$$
   %
R_1=\left(
\begin{array}{cccc}
\zeta     & 0              & \zeta-\zeta^2 & 1       \\
0         & 0              &       1       & \zeta^2 \\
1-\zeta^2 & 1              & 1-\zeta       & \zeta^2 \\
1-\zeta   & \zeta^2-\zeta  & \zeta^2-\zeta &   -1
\end{array}
\right)\,.
$$


\subsection{Choice of concrete paths and calculation of the monodromy}
\label{ss:calculation}

Consider the maps
\begin{align*}
v_1: = r^{-1}_3\cdot h^{-1}_2\cdot r_1 &: H_1(\hat S_1;\Z)\to H_1(\hat S_3;\Z)\\
v_2: = r^{-1}_2\cdot h^{-1}_3\cdot r_2 &: H_1(\hat S_2;\Z)\to H_1(\hat S_2;\Z)\\
v_3: = r^{-1}_1\cdot h^{-1}_1\cdot r_3 &: H_1(\hat S_3;\Z)\to H_1(\hat S_1;\Z)
\end{align*}
induced by the vertical shear
$
v=\begin{pmatrix}1&0\\1&1\end{pmatrix}.
$
In  the  chosen  bases  of homology, the restrictions of these linear
maps to the subspaces ${\cE_*}_{\hat S_i}(\zeta)$ have matrices
\begin{align*}
A^{vert}_1&=R^{-1}_3\cdot (A^{hor}_2)^{-1}\cdot R_1 =\Id\cdot\Id\cdot R_1                    =R_1\\
A^{vert}_2&=R^{-1}_2\cdot (A^{hor}_3)^{-1}\cdot R_2 =\Id\cdot (A^{hor}_3)^{-1}\cdot\Id       =(A^{hor}_3)^{-1}\\
A^{vert}_3&=R^{-1}_1\cdot (A^{hor}_1)^{-1}\cdot R_3 =R_1^{-1}\cdot (A^{hor}_1)^{-1}\cdot \Id =R_1^{-1}\cdot (A^{hor}_1)^{-1}
\end{align*}
correspondingly. By multiplying, we obtain

\begin{align*}
A^{vert}_1&=\left(
\begin{array}{cccc}
 \zeta & 0 & \zeta-\zeta^2 & 1 \\
 0 & 0 & 1 & \zeta^2 \\
 1-\zeta^2 & 1 & 1-\zeta & \zeta^2 \\
 1-\zeta & \zeta^2-\zeta & \zeta^2-\zeta & -1
\end{array}
\right)
\\
\\
A^{vert}_2&=\left(
\begin{array}{cccc}
 0 & \zeta^2 & 0 & \zeta \\
 0 & 0 & \zeta^2 & \zeta \\
 1 & 0 & 0 & 1 \\
 0 & 0 & 0 & -\zeta
\end{array}
\right)
\\
\\
A^{vert}_3&=\left(
\begin{array}{cccc}
 0 & 0 & 1 & \zeta^2 \\
 \zeta & 0 & 0 & \zeta \\
 0 & \zeta & 0 & \zeta \\
 0 & 0 & 0 & -\zeta^2
\end{array}
\right)
\end{align*}

In  the  natural  bases of homology, the restrictions of these linear
maps to the subspaces $\Lambda^2 {\cE_*}_{\hat S_i}(\zeta)$ have matrices
\begin{align*}
W^{hor}_1&=\left(
\begin{array}{cccccc}
 \zeta^2-1 & 0 & 0 & -\zeta & 0 & 0 \\
 \zeta-\zeta^2 & -\zeta & 0 & \zeta^2-1 & 0 & 0 \\
 0 & \zeta-\zeta^2 & 1-\zeta & 1-\zeta^2 & \zeta^2-\zeta & \zeta^2 \\
 -\zeta & 0 & 0 & 0 & 0 & 0 \\
 \zeta^2-1 & 0 & 0 & 1-\zeta & \zeta^2 & 0 \\
 \zeta-\zeta^2 & 1-\zeta & \zeta^2 & 2 \zeta^2-\zeta-1 & \zeta-1 & 0
\end{array}
\right)
\\
\\
W^{hor}_3&=\left(
\begin{array}{cccccc}
 0 & -\zeta & -1 & 0 & 0 & \zeta \\
 0 & 0 & 0 & -\zeta & -1 & \zeta \\
 0 & 0 & 0 & 0 & 0 & -\zeta^2 \\
 \zeta^2 & 0 & \zeta^2 & 0 & -\zeta^2 & 0 \\
 0 & 0 & -1 & 0 & 0 & 0 \\
 0 & 0 & 0 & 0 & -1 & 0
\end{array}
\right)
\end{align*}

\begin{align*}
W^{vert}_1&=\left(
\begin{array}{cccccc}
 0 & \zeta & 1 & 0 & 0 & -\zeta \\
 \zeta & 1-\zeta & \zeta^2 & \zeta^2-\zeta & -1 & 0 \\
 1-\zeta^2 & \zeta^2-\zeta & -1 & \zeta^2+\zeta-2 & \zeta-\zeta^2 & 0 \\
 0 & \zeta^2-1 & \zeta-\zeta^2 & -1 & -\zeta^2 & 1 \\
 0 & \zeta-1 & 1-\zeta^2 & \zeta-\zeta^2 & 1-\zeta & -\zeta \\
 \zeta^2-\zeta & 0 & 0 & 1-\zeta^2 & -\zeta & 0
\end{array}
\right)
\\
\\
W^{vert}_2&=\left(
\begin{array}{cccccc}
 0 & 0 & 0 & \zeta & 1 & -1 \\
 -\zeta^2 & 0 & -\zeta & 0 & \zeta^2 & 0 \\
 0 & 0 & 0 & 0 & -1 & 0 \\
 0 & -\zeta^2 & -\zeta & 0 & 0 & \zeta^2 \\
 0 & 0 & 0 & 0 & 0 & -1 \\
 0 & 0 & -\zeta & 0 & 0 & 0
\end{array}
\right)
\\
\\
W^{vert}_3&=\left(
\begin{array}{cccccc}
 0 & -\zeta & -1 & 0 & 0 & \zeta \\
 0 & 0 & 0 & -\zeta & -1 & \zeta \\
 0 & 0 & 0 & 0 & 0 & -\zeta^2 \\
 \zeta^2 & 0 & \zeta^2 & 0 & -\zeta^2 & 0 \\
 0 & 0 & -1 & 0 & 0 & 0 \\
 0 & 0 & 0 & 0 & -1 & 0
\end{array}
\right)
\end{align*}
and $W^{hor}_1=\Id$.

Consider  now  the following loop $\rho_1$ on $\hat\cT$: start with a
horizontal move from $\hat S_1$ and follow the trajectory:
\begin{equation}
\label{eq:rho1}
1\xrightarrow{hor}
2\xrightarrow{vert}
2\xrightarrow{vert}
2\xrightarrow{vert}
2\xrightarrow{hor}
1\xrightarrow{vert}
3\xrightarrow{vert}
1
\end{equation}
The  corresponding  monodromy  matrices  $X$ in $\cE_*(\zeta)$ and $U$ in
$\Lambda^2 \cE_*(\zeta)$ are computed as follows:

$$
X:=A^{vert}_3\cdot A^{vert}_1\cdot A^{hor}_2\cdot\left(A^{vert}_2\right)^3\cdot A^{hor}_1
$$
$$
U:=W^{vert}_3\cdot W^{vert}_1\cdot W^{hor}_2\cdot\left(W^{vert}_2\right)^3\cdot W^{hor}_1
$$

Consider  now  the  second  loop  $\rho_2$ on $\hat\cT$: start with a
vertical move from $\hat S_1$ and follow the trajectory:
\begin{equation}
\label{eq:rho2}
1\xrightarrow{vert}
3\xrightarrow{hor}
3\xrightarrow{hor}
3\xrightarrow{hor}
3\xrightarrow{vert}
1
\end{equation}
The  corresponding  monodromy  matrices  $Y$ in $\cE_*(\zeta)$ and $V$ in
$\Lambda^2 \cE_*(\zeta)$ are computed as follows:
$$
Y:=A^{vert}_3\cdot\left(A^{hor}_3\right)^3\cdot A^{vert}_1
$$
$$
V:=W^{vert}_3\cdot\left(W^{hor}_3\right)^3\cdot W^{vert}_1
$$

As a result we get the following numerical matrices:
$$
X=\left(
\begin{array}{cccc}
 \zeta^2-\zeta & \zeta^2+\zeta-1 & \zeta^2-1 & \zeta \\
 3 \zeta^2-3 & -\zeta^2+3 \zeta-2 & 3 \zeta-2 & -\zeta^2+\zeta+1 \\
 6 \zeta^2-5 & 4 \zeta-4 & \zeta^2+5 \zeta-6 & -3 \zeta^2+3 \zeta +1 \\
 -\zeta^2+6 \zeta-5 & -5 \zeta^2+4 \zeta+1 & -6 \zeta^2+5 \zeta+1 & -2 \zeta^2-2 \zeta+3
\end{array}
\right)
$$
$$
Y=\left(
\begin{array}{cccc}
 \zeta^2-\zeta & \zeta^2 & \zeta^2-1 & \zeta \\
 \zeta^2-\zeta+1 & 2 \zeta^2-\zeta-1 & \zeta^2-1 & \zeta \\
 \zeta^2-2 \zeta+1 & 2 \zeta^2-\zeta-1 & 2 \zeta^2-\zeta & \zeta^2+\zeta-1 \\
 \zeta^2-1 & \zeta-1 & \zeta-1 & -\zeta^2
\end{array}
\right)
$$
$$
U=\left(
\begin{array}{cccccc}
 \zeta^2+\zeta-2 & \zeta^2-\zeta & \zeta-1 & 2 \zeta^2-\zeta
  & \zeta & -\zeta^2 \\
 3 \zeta^2+3 \zeta
 -7 & \zeta-1 & \zeta-2 \zeta^2 & 2 \zeta^2-6 \zeta+4 & 3 \zeta^2-\zeta-1 & 1-\zeta^2 \\
 -5 \zeta^2+6 \zeta-1 & 0 & 1-\zeta^2 & 6 \zeta^2-\zeta-5 & -\zeta^2+2 \zeta-2 & 1-\zeta \\
 -5 \zeta^2+9 \zeta-4 & -9 \zeta^2+3 \zeta+5 & \zeta^2-7 \zeta+5 & 7 \zeta^2-9 \zeta+2 & 6 \zeta^2-6 & 3 \zeta^2-\zeta-1 \\
 -7 \zeta^2-\zeta+8 & \zeta^2-6 \zeta+5 & 5 \zeta^2-2 \zeta-3 & 4 \zeta^2+4 \zeta-8 & -3 \zeta^2+5 \zeta-2 & \zeta-2 \\
 3 \zeta^2-6 \zeta+3 & 5 \zeta^2+\zeta-6 & -4 \zeta^2+4 \zeta-1 & -7 \zeta^2+8 \zeta-1 & -4 \zeta^2-2 \zeta+6 & -\zeta^2-\zeta+2
\end{array}
\right)
$$
$$
V=\left(
\begin{array}{cccccc}
 -\zeta^2+2 \zeta-2 & 1-\zeta^2 & -\zeta & 2 \zeta^2-2 \zeta & \zeta^2+\zeta-1 & 0 \\
 -\zeta^2+2 \zeta-1 & \zeta^2-\zeta & \zeta-1 & 3 \zeta^2-\zeta-1 & 2 \zeta-1 & -\zeta^2 \\
 1-\zeta^2 & 0 & 0 & \zeta-1 & -\zeta^2 & 0 \\
 -\zeta^2-\zeta+2 & 3 \zeta^2-2 \zeta & \zeta^2+2 \zeta-2 & 2 \zeta^2+2 \zeta-4 & 3 \zeta-3 \zeta^2 & -\zeta^2 \\
 \zeta^2-\zeta & \zeta-1 & -\zeta^2 & -2 \zeta^2+\zeta+1 & 1-\zeta
   & 0 \\
 0 & 1-\zeta^2 & \zeta^2-\zeta & 1-\zeta & \zeta^2-1 & -1
\end{array}
\right)
$$

Computing the determinants we get:
$$
\det(XY-YX)=-285
$$
and
$$
\det(UV-VU)=-5292\,.
$$

\section{Computation of the Zariski closure of a monodromy representation}\label{a:Zariski:closure}

Let $\zeta=\exp(2\pi i/3)$ and $\eta=\exp(2\pi i/6)$. Consider the following matrices
$$A=A_3^{hor}=\left(\begin{array}{cccc}0&0&1&\zeta^2 \\ \zeta&0&0&\zeta \\ 0&\zeta&0&\zeta \\ 0&0&0&-\zeta^2\end{array}\right)$$
and
$$B=A_1^{vert}\cdot A_3^{vert}=\left(\begin{array}{cccc}0&\zeta^2-1&\zeta&0 \\ 0&\zeta&0&0 \\
\zeta&\zeta-\zeta^2&1-\zeta^2&0 \\ 1-\zeta^2&1-\zeta^2&1-\zeta&1\end{array}\right)$$

The product
$$C=B\cdot A=\left(\begin{array}{cccc}1-\zeta&\zeta^2&0&-2\zeta \\ \zeta^2&0&0&\zeta^2
\\ \zeta^2-1&\zeta-1&\zeta&-2 \\ \zeta-1&\zeta-\zeta^2&1-\zeta^2&2\zeta\end{array}\right)$$
has characteristic polynomial
$$T^4+(\zeta^2-\zeta)T^3-2\zeta^2T^2+(\zeta^2-1)T+\zeta$$
$$= (T-1)\cdot(T^3-2\zeta T^2+2T-\zeta)$$
Denoting by $\alpha$, $\beta$ and $\mu$ the roots of $T^3-2\zeta T^2+2T-\zeta=0$ with $|\alpha|=|\beta|^{-1}>1=|\mu|$, we have that $C$ has eigenvalues $\alpha$, $1$, $\mu$ and $\beta$ with eigenvectors
\begin{itemize}
\item $v_{\alpha}=\left(-\frac{(-\zeta+2\zeta\alpha)}{(-1+\alpha)\cdot(\zeta+\alpha)}, \frac{\eta(1+\zeta-\alpha)}{(-1+\alpha)\cdot(\zeta+\alpha)}, \frac{\eta(1+\zeta+2\zeta\alpha)}{\eta+\alpha^2}, 1\right)$,
\item $v_1=(\zeta,1,0,0)$,
\item $v_{\mu}=\left(-\frac{(-\zeta+2\zeta\mu)}{(-1+\mu)\cdot(\zeta+\mu)}, \frac{\eta(1+\zeta-\mu)}{(-1+\mu)\cdot(\zeta+\mu)}, \frac{\eta(1+\zeta+2\zeta\mu)}{\eta+\mu^2}, 1\right)$ and
\item $v_{\beta}=\left(-\frac{(-\zeta+2\zeta\beta)}{(-1+\beta)\cdot(\zeta+\beta)}, \frac{\eta(1+\zeta-\beta)}{(-1+\beta)\cdot(\zeta+\beta)}, \frac{\eta(1+\zeta+2\zeta\beta)}{\eta+\beta^2}, 1\right)$
\end{itemize}

Since $\det A=-\zeta$ and $\det B=-1$, we have that $\det C=\zeta$. In particular, $\det C^3=1$ and hence $\alpha^3\beta^3\mu^3=1$. This allows us to write
$\alpha^3=R\cdot e^{iT_{\alpha}}$, $\beta^3=R^{-1}\cdot e^{iT_{\beta}}$ and $\mu^3=e^{iT_{\mu}}$ with $T_{\mu}=-T_{\alpha}-T_{\beta}$.

\begin{theorem}\label{t.ZclAB}The subgroup $\langle A, B\rangle\,\cap\, SU(3,1)$ is Zariski dense in $SU(3,1)$.
\end{theorem}

We start by noticing that $C^3\in SU(3,1)$ is contained in the $1$-parameter subgroup of $SU(3,1)$ with infinitesimal generator $X\in \mathfrak{su}(3,1)$ where
\begin{itemize}
\item $Xv_{\alpha}=(r+iT_{\alpha})\cdot v_{\alpha}$,
\item $Xv_1=0$,
\item $Xv_{\mu}=iT_{\mu} v_{\mu}$ and
\item $X v_{\beta}=(-r+iT_{\beta})\cdot v_{\beta}$
\end{itemize}
and $r=\log R=3\log|\alpha|$.

For computational reasons, let's write $X$ in terms of the canonical basis $\{e_1,\dots,e_4\}$ of $\mathbb{C}^4$ as follows. Put
$$P=\left(
\begin{array}{cccc}
-\frac{(-\zeta+2\zeta\alpha)}{(-1+\alpha)\cdot(\zeta+\alpha)} & \zeta & -\frac{(-\zeta+2\zeta\mu)}{(-1+\mu)\cdot(\zeta+\mu)} & -\frac{(-\zeta+2\zeta\beta)}{(-1+\beta)\cdot(\zeta+\beta)} \\
\frac{\eta(1+\zeta-\alpha)}{(-1+\alpha)\cdot(\zeta+\alpha)} &1 & \frac{\eta(1+\zeta-\mu)}{(-1+\mu)\cdot(\zeta+\mu)} & \frac{\eta(1+\zeta-\beta)}{(-1+\beta)\cdot(\zeta+\beta)} \\
\frac{\eta(1+\zeta+2\zeta\alpha)}{\eta+\alpha^2} & 0 & \frac{\eta(1+\zeta+2\zeta\mu)}{\eta+\mu^2} & \frac{\eta(1+\zeta+2\zeta\beta)}{\eta+\beta^2} \\
1& 0 & 1 & 1
\end{array}
\right)$$
so that $Pe_1=v_\alpha$, $Pe_2=v_1$, $Pe_3=v_\mu$ and $Pe_4=v_\beta$. Then,
$$X=P\cdot D_X\cdot P^{-1}$$
where
$$D_X=\left(\begin{array}{cccc}
r+iT_\alpha&0&0&0\\
0&0&0&0 \\
0&0& iT_\mu&0 \\
0&0&0&-r+iT_\beta
\end{array}\right)$$

Next,  we observe that $A$ has order $18$ (and $B$ has order $6$), so
that we can construct
$$X(n)=A^{n-1}\cdot X\cdot A^{-(n-1)}$$
for $n=1,\dots, 17$.

Applying  the  method  of least squares (cf. Remark \ref{r.Mathematica} below) we verify that
$$X(1), \dots, X(9)$$
are $9$ linearly independent vectors.

Now  we  use the matrix $B$ to conjugate the
resulting   independent   vectors   $X(1),\dots,  X(9)$. We
construct vectors
$$
Y(n)=B\cdot X(n)\cdot B\,,
$$
and    applying   the   method   of   least   squares (cf. Remark \ref{r.Mathematica} below)  verify   that
$$Y(1), Y(3),\dots,Y(7)$$
give  six more vectors such that $X(1),\dots, X(9), Y(1), Y(3), \dots, Y(7)$ span a vector space of dimension $15$.

Note that the vectors $X(1),\dots,X(9), Y(1), Y(3),\dots,Y(7)$ belong to the Lie algebra $\mathfrak{g}_0$ of the Zariski closure $G=\textrm{Zcl}(\langle A, B\rangle\cap SU(3,1))$ of $\langle A, B\rangle\cap SU(3,1)$. Indeed, this is a consequence of the following general lemma:
\begin{lemma} Let $C$ be a hyperbolic or unipotent element of $SU(p,q)$. Then, the logarithm $X\in \mathfrak{su}(p,q)$ of $C$ belongs to the Lie algebra $\mathfrak{g}$ of any Zariski closed subgroup $G$ of $SU(p,q)$ containing $C$.
\end{lemma}

\begin{proof} Since $C$ is hyperbolic or unipotent, its iterates $C^n$, $n\in \mathbb{Z}$, form an infinite discrete subset of $SU(p,q)$. Therefore, any Zariski closed subgroup $G$ of $SU(p,q)$ containing $C$ has dimension $1$ at least: in fact, any $0$-dimensional Zariski closed subgroup is finite. On the other hand, denoting by $X$ the logarithm of $C$, we have that  $\{\exp(tX):t\in\mathbb{R}\}$ is the smallest $1$-parameter subgroup containing all iterates $C^n$, $n\in\mathbb{Z}$, of $C$. So, it follows that $\{\exp(tX):t\in\mathbb{R}\}\subset G$, and, thus, $X\in\mathfrak{g}$ where $\mathfrak{g}$ is the Lie algebra of $G$.
\end{proof}

In particular, coming back to the proof of Theorem \ref{t.ZclAB},
we have  proved that $\mathfrak{g}_0\subset \mathfrak{su}(3,1)$ is a vector space of dimension $15$ at least. Since $\mathfrak{su}(3,1)$ is a $15$-dimensional real Lie algebra, we conclude that $\mathfrak{g}_0=\mathfrak{su}(3,1)$, and hence $G=SU(3,1)$.

This completes the proof of Theorem \ref{t.ZclAB}.
   %

\begin{remark}\label{r.Mathematica} In the webpages of the last two authors (C.M. and A.Z.), the reader will find a Mathematica routine called ``FMZ3-Zariski-numerics$\_$det1.nb'' where the numerical verification of the linear independence of the vectors $X(n)$, $n=1,\dots,9$, $Y(m)$, $m=1,3,\dots,7$.
\end{remark}

\bigskip

\subsection*{Acknowledgments}
The authors are grateful to A.~Eskin, M.~Kontsevich, M.~M\"oller, and
J.-C.~Yoccoz for extremely stimulating discussions, and to Y. Guivarch
for the question behind Proposition~\ref{prop:isometryU31} above.

We  highly  appreciated explanations of M.~Kontsevich and A.~Wright
concerning the Deligne Semisimplicity Theorem.

We   would   like   to   thank   M.~M\"oller  for  suggesting to us that the statement of Lemma~\ref{lm:Ezeta:strong:irreducibility} should hold,  A.~Eskin for his ideas
leading to the proof of this Lemma, and G.~Pearlstein, I.~Rivin and
M.~Sapir  for  interesting discussions around the computation of the
Zariski closure of concrete examples.

The  authors  are  thankful  to  Coll\`ege de France, HIM, IHES, IUF,
IMPA,  MPIM, and the Universities of Chicago, Maryland, Rennes 1, and
Paris 7 for hospitality during the preparation of this paper.

C.M. was partially supported by the
Balzan Research Project of J. Palis, and C.M. and A.Z. were partially supported by the French ANR grant ``GeoDyM'' (ANR-11-BS01-0004).

\clearpage

\special{
psfile=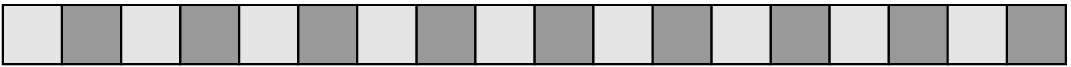
hscale=85
vscale=85
angle=90
hoffset=40
voffset=-453 
}
\begin{picture}(0,0)(-9,22) 
\begin{picture}(0,0)(0,0)
\put(0,0){11}
\put(0,-24.5){10}
\put(3,-49){9}
\put(3,-73.5){8}
\put(3,-98){7}
\put(3,-122.5){6}
\end{picture}
\begin{picture}(0,0)(2.5,145)
\put(0,0){17}
\put(0,-24.5){16}
\put(0,-49){15}
\put(0,-73.5){14}
\put(0,-98){13}
\put(0,-122.5){12}
\end{picture}
\begin{picture}(0,0)(5,290)
\put(3,0){5}
\put(3,-24.5){4}
\put(3,-49){3}
\put(3,-73.5){2}
\put(3,-98){1}
\put(3,-122.5){0}
\end{picture}
\begin{picture}(0,0)(2.5,0)
\put(-16,12){\tiny\textit A}
\put(14.5,12){\tiny\textit C}
\put(14.5,-12.75){\tiny\textit F}
\put(14.5,-37){\tiny\textit C}
\put(-16,-61){\tiny\textit B}
\put(14.5,-61){\tiny\textit E}
\put(14.5,-86){\tiny\textit C}
\put(14.5,-109.5){\tiny\textit D}
\end{picture}
\begin{picture}(0,0)(5,145)
\put(-16,12){\tiny\textit A}
\put(14.5,12){\tiny\textit C}
\put(14.5,-12.75){\tiny\textit F}
\put(14.5,-37){\tiny\textit C}
\put(-16,-61){\tiny\textit B}
\put(14.5,-61){\tiny\textit E}
\put(14.5,-86){\tiny\textit C}
\put(14.5,-109.5){\tiny\textit D}
\end{picture}
\begin{picture}(0,0)(7.5,290)
\put(-16,12){\tiny\textit A}
\put(14.5,12){\tiny\textit C}
\put(14.5,-12.75){\tiny\textit F}
\put(14.5,-37){\tiny\textit C}
\put(-16,-61){\tiny\textit B}
\put(14.5,-61){\tiny\textit E}
\put(14.5,-86){\tiny\textit C}
\put(14.5,-109.5){\tiny\textit D}
\end{picture}
\begin{picture}(0,0)(10,436)
\put(-16,12){\tiny\textit A}
\put(14.5,12){\tiny\textit C}
\end{picture}
\put(-12,15){\tiny $0$}
\put(-30,4.5){\rotatebox{180}{\tiny $12$}} 
\put(4,4.5){\rotatebox{180}{\tiny $4$}}
\put(-30,-19.5){\rotatebox{180}{\tiny $13$}}
\put(4,-19.5){\rotatebox{180}{\tiny $17$}}
\put(-30,-44){\rotatebox{180}{\tiny $14$}}
\put(4,-44){\rotatebox{180}{\tiny $2$}}
\put(-26,-68){\rotatebox{180}{\tiny $3$}}
\put(4,-68){\rotatebox{180}{\tiny $15$}}
\put(-26,-92.5){\rotatebox{180}{\tiny $4$}}
\put(4,-92.5){\rotatebox{180}{\tiny $0$}}
\put(-26,-117){\rotatebox{180}{\tiny $5$}}
\put(4,-117){\rotatebox{180}{\tiny $13$}}
\put(-26,-139.5){\rotatebox{180}{\tiny $0$}}
\put(4,-139.5){\rotatebox{180}{\tiny $10$}}
\put(-26,-164){\rotatebox{180}{\tiny $1$}}
\put(4,-164){\rotatebox{180}{\tiny $5$}}
\put(-26,-189){\rotatebox{180}{\tiny $2$}}
\put(4,-189){\rotatebox{180}{\tiny $8$}}
\put(-26,-213){\rotatebox{180}{\tiny $9$}}
\put(4,-213){\rotatebox{180}{\tiny $3$}}
\put(-30,-238){\rotatebox{180}{\tiny $10$}}
\put(4,-238){\rotatebox{180}{\tiny $6$}}
\put(-30,-262){\rotatebox{180}{\tiny $11$}}
\put(4,-262){\rotatebox{180}{\tiny $1$}}
\put(-26,-285.5){\rotatebox{180}{\tiny $6$}}
\put(4,-285.5){\rotatebox{180}{\tiny $16$}}
\put(-26,-310){\rotatebox{180}{\tiny $7$}}
\put(4,-310){\rotatebox{180}{\tiny $11$}}
\put(-26,-334.5){\rotatebox{180}{\tiny $8$}}
\put(4,-334.5){\rotatebox{180}{\tiny $14$}}
\put(-30,-359){\rotatebox{180}{\tiny $15$}}
\put(4,-359){\rotatebox{180}{\tiny $9$}}
\put(-30,-383.5){\rotatebox{180}{\tiny $16$}}
\put(4,-383.5){\rotatebox{180}{\tiny $12$}}
\put(-30,-408){\rotatebox{180}{\tiny $17$}}
\put(4,-408){\rotatebox{180}{\tiny $7$}}
\put(-13,-428){\tiny $11$}
\put(-13.5,-446){$\hat S_3$}
\end{picture}


\special{
psfile=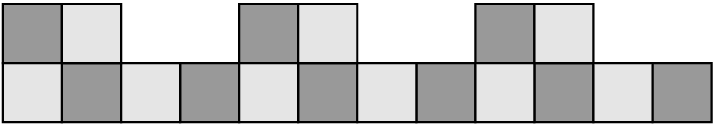
hscale=85
vscale=85
angle=90
hoffset=195    
voffset=-298.5 
}
\begin{picture}(0,0)(-165,-135.5) 
\begin{picture}(0,0)(0,97.5)
\put(-24.5,-98){17}
\put(-24.5,-122.5){16}
\put(0,-49){15}
\put(0,-73.5){14}
\put(0,-98){13}
\put(0,-122.5){12}
\end{picture}
\begin{picture}(0,0)(2.5,193)
\put(-24.5,-98){11}
\put(-24.5,-122.5){10}
\put(3,-49){9}
\put(3,-73.5){8}
\put(3,-98){7}
\put(3,-122.5){6}
\end{picture}
\begin{picture}(0,0)(5,290)
\put(-21.5,-98){5}
\put(-21.5,-122.5){4}
\put(3,-49){3}
\put(3,-73.5){2}
\put(3,-98){1}
\put(3,-122.5){0}
\end{picture}
\begin{picture}(0,0)(2.5,145)
\put(-16,12){\tiny\textit C}
\put(14.5,12){\tiny\textit F}
\put(-16,-12.75){\tiny\textit A}
\put(-40,-37){\tiny\textit E}
\put(-16,-34){\tiny\textit C}
\put(14.5,-37){\tiny\textit D}
\put(-40,-61){\tiny\textit B}
\put(-40,-85){\tiny\textit E}
\put(-16,-88){\tiny\textit C}
\put(14.5,-85){\tiny\textit F}
\end{picture}
\begin{picture}(0,0)(5,241.5)
\put(-16,-12.75){\tiny\textit A}
\put(-40,-37){\tiny\textit E}
\put(-16,-34){\tiny\textit C}
\put(14.5,-37){\tiny\textit D}
\put(-40,-61){\tiny\textit B}
\put(-40,-85){\tiny\textit E}
\put(-16,-88.5){\tiny\textit C}
\put(14.5,-85){\tiny\textit F}
\end{picture}
\begin{picture}(0,0)(7.5,338)
\put(-16,-12.75){\tiny\textit A}
\put(-40,-37){\tiny\textit E}
\put(-16,-34){\tiny\textit C}
\put(14.5,-37){\tiny\textit D}
\put(-40,-61.5){\tiny\textit B}
\put(-40,-86){\tiny\textit E}
\put(-14,-89){\tiny\textit C}
\put(14.5,-86){\tiny\textit F}
\end{picture}
\begin{picture}(0,0)(-0.5,145)
\put(-12,15){\tiny $0$}
\put(-26,4.5){\rotatebox{180}{\tiny $2$}} 
\put(4,4.5){\rotatebox{180}{\tiny $6$}}
\put(-26,-19.5){\rotatebox{180}{\tiny $9$}}
\put(4,-19.5){\rotatebox{180}{\tiny $7$}}
\put(-39,-34){\tiny $10$}
\put(-52,-45){\rotatebox{180}{\tiny $4$}}
\put(4,-45){\rotatebox{180}{\tiny $2$}}
\put(-54,-70){\rotatebox{180}{\tiny $11$}}
\put(4,-70){\rotatebox{180}{\tiny $3$}}
\put(-37,-88){\tiny $5$}
\put(-30,-93){\rotatebox{180}{\tiny $14$}}
\put(4,-93){\rotatebox{180}{\tiny $0$}}
\put(-26,-117){\rotatebox{180}{\tiny $3$}}
\put(4,-117){\rotatebox{180}{\tiny $1$}}
\put(-37,-130){\tiny $4$}
\put(-54,-140.5){\rotatebox{180}{\tiny $16$}}
\put(4,-140.5){\rotatebox{180}{\tiny $14$}}
\put(-50,-165){\rotatebox{180}{\tiny $5$}}
\put(4,-165){\rotatebox{180}{\tiny $15$}}
\put(-39,-185){\tiny $17$}
\put(-26,-189.5){\rotatebox{180}{\tiny $8$}}
\put(4,-189.5){\rotatebox{180}{\tiny $12$}}
\put(-30,-213){\rotatebox{180}{\tiny $15$}}
\put(4,-213){\rotatebox{180}{\tiny $13$}}
\put(-39,-227){\tiny $16$}
\put(-54,-238){\rotatebox{180}{\tiny $10$}}
\put(4,-238){\rotatebox{180}{\tiny $8$}}
\put(-54,-262){\rotatebox{180}{\tiny $17$}}
\put(4,-262){\rotatebox{180}{\tiny $9$}}
\put(-39,-282){\tiny $11$}
\put(-14.5,-282){\tiny $15$}
\put(-26,-300){$\hat S_2$}
\end{picture}
\end{picture}


\special{ 
psfile=twolines18_0.eps
hscale=85
vscale=85
angle=90
hoffset=350 
voffset=-286.5 
}
   %
\begin{picture}(0,0)(-320,-147.5) 
\begin{picture}(0,0)(0,97.5)
\put(-24.5,-98){17}
\put(-24.5,-122.5){16}
\put(0,-49){15}
\put(0,-73.5){14}
\put(0,-98){13}
\put(0,-122.5){12}
\end{picture}
\begin{picture}(0,0)(2.5,193)
\put(-24.5,-98){11}
\put(-24.5,-122.5){10}
\put(3,-49){9}
\put(3,-73.5){8}
\put(3,-98){7}
\put(3,-122.5){6}
\end{picture}
\begin{picture}(0,0)(5,290)
\put(-21.5,-98){5}
\put(-21.5,-122.5){4}
\put(3,-49){3}
\put(3,-73.5){2}
\put(3,-98){1}
\put(3,-122.5){0}
\end{picture}
\begin{picture}(0,0)(2.5,144.5)
\put(-16,12){\tiny\textit C}
\put(14.5,-12.75){\tiny\textit F}
\put(-16,-12.75){\tiny\textit A}
\put(-40,-37){\tiny\textit B}
\put(-16,-34){\tiny\textit C}
\put(14.5,-61){\tiny\textit D}
\put(-40,-61){\tiny\textit E}
\put(-40,-85){\tiny\textit B}
\put(-16,-88){\tiny\textit C}
\end{picture}
\begin{picture}(0,0)(5,241.5)
\put(14.5,-12.75){\tiny\textit F}
\put(-16,-12.75){\tiny\textit A}
\put(-40,-37){\tiny\textit B}
\put(-16,-34){\tiny\textit C}
\put(14.5,-61){\tiny\textit D}
\put(-40,-61){\tiny\textit E}
\put(-40,-85){\tiny\textit B}
\put(-16,-88){\tiny\textit C}
\end{picture}
\begin{picture}(0,0)(7.5,338)
\put(14.5,-12.75){\tiny\textit F}
\put(-16,-12.75){\tiny\textit A}
\put(-40,-37){\tiny\textit B}
\put(-16,-34){\tiny\textit C}
\put(14.5,-61){\tiny\textit D}
\put(-40,-61){\tiny\textit E}
\put(-40,-89){\tiny\textit B}
\put(-14,-89){\tiny\textit C}
\put(14,-89){\tiny\textit F}
\end{picture}
\begin{picture}(0,0)(-0.5,145)
\put(-12,15){\tiny $0$}
\put(-26,4.5){\rotatebox{180}{\tiny $2$}} 
\put(4,4.5){\rotatebox{180}{\tiny $8$}}
\put(-26,-19.5){\rotatebox{180}{\tiny $9$}}
\put(4,-19.5){\rotatebox{180}{\tiny $3$}}
\put(-39,-34){\tiny $10$}
\put(-52,-45){\rotatebox{180}{\tiny $4$}}
\put(4,-45){\rotatebox{180}{\tiny $6$}}
\put(-54,-70){\rotatebox{180}{\tiny $11$}}
\put(4,-70){\rotatebox{180}{\tiny $1$}}
\put(-37,-88){\tiny $5$}
\put(-30,-93){\rotatebox{180}{\tiny $14$}}
\put(4,-93){\rotatebox{180}{\tiny $2$}}
\put(-26,-117){\rotatebox{180}{\tiny $3$}}
\put(4,-117){\rotatebox{180}{\tiny $15$}}
\put(-37,-130){\tiny $4$}
\put(-54,-140.5){\rotatebox{180}{\tiny $16$}}
\put(4,-140.5){\rotatebox{180}{\tiny $0$}}
\put(-50,-165){\rotatebox{180}{\tiny $5$}}
\put(4,-165){\rotatebox{180}{\tiny $13$}}
\put(-39,-185){\tiny $17$}
\put(-26,-189.5){\rotatebox{180}{\tiny $8$}}
\put(4,-189.5){\rotatebox{180}{\tiny $14$}}
\put(-30,-213){\rotatebox{180}{\tiny $15$}}
\put(4,-213){\rotatebox{180}{\tiny $9$}}
\put(-39,-227){\tiny $16$}
\put(-54,-238){\rotatebox{180}{\tiny $10$}}
\put(4,-238){\rotatebox{180}{\tiny $12$}}
\put(-54,-262){\rotatebox{180}{\tiny $17$}}
\put(4,-262){\rotatebox{180}{\tiny $7$}}
\put(-39,-282){\tiny $11$}
\put(-14.5,-282){\tiny $15$}
\put(-26,-300){$\hat S_1$}
\end{picture}
\end{picture}

\includegraphics{Tdisc3.eps}
\begin{picture}(0,0)(-248,411) 
\put(-176,-66){$h$}
\put(-118,-61){$r$}
\put(-67,-50){$h$}
\put(-67,-83){$h$}
\put(-9,-66){$r$}
\put(-140,-72){\scriptsize $\hat S_3$}
\put(-97,-72){\scriptsize $\hat S_2$}
\put(-36,-66){\scriptsize  $\hat S_1$}
\end{picture}

\vspace{500bp}

\begin{figure}[htb]
\caption{
\label{fig:PSL2Z:orbit}
\label{fig:three:surfaces}
$\PSLZ$-orbit of $\hat S=\hat S_3$.
}
\end{figure}

%

\clearpage

\special{
psfile=oneline18.eps
hscale=85
vscale=85
angle=90
hoffset=40
voffset=-538 
}
\begin{picture}(0,0)(-9,107) 
\begin{picture}(0,0)(0,0)
\put(0,0){11}
\put(0,-24.5){10}
\put(3,-49){9}
\put(3,-73.5){8}
\put(3,-98){7}
\put(3,-122.5){6}
\end{picture}
\begin{picture}(0,0)(2.5,145)
\put(0,0){17}
\put(0,-24.5){16}
\put(0,-49){15}
\put(0,-73.5){14}
\put(0,-98){13}
\put(0,-122.5){12}
\end{picture}
\begin{picture}(0,0)(5,290)
\put(3,0){5}
\put(3,-24.5){4}
\put(3,-49){3}
\put(3,-73.5){2}
\put(3,-98){1}
\put(3,-122.5){0}
\end{picture}
\begin{picture}(0,0)(2.5,0)
\put(-16,12){\tiny\textit A}
\put(14.5,12){\tiny\textit C}
\put(14.5,-12.75){\tiny\textit F}
\put(14.5,-37){\tiny\textit C}
\put(-16,-61){\tiny\textit B}
\put(14.5,-61){\tiny\textit E}
\put(14.5,-86){\tiny\textit C}
\put(14.5,-109.5){\tiny\textit D}
\end{picture}
\begin{picture}(0,0)(5,145)
\put(-16,12){\tiny\textit A}
\put(14.5,12){\tiny\textit C}
\put(14.5,-12.75){\tiny\textit F}
\put(14.5,-37){\tiny\textit C}
\put(-16,-61){\tiny\textit B}
\put(14.5,-61){\tiny\textit E}
\put(14.5,-86){\tiny\textit C}
\put(14.5,-109.5){\tiny\textit D}
\end{picture}
\begin{picture}(0,0)(7.5,290)
\put(-16,12){\tiny\textit A}
\put(14.5,12){\tiny\textit C}
\put(14.5,-12.75){\tiny\textit F}
\put(14.5,-37){\tiny\textit C}
\put(-16,-61){\tiny\textit B}
\put(14.5,-61){\tiny\textit E}
\put(14.5,-86){\tiny\textit C}
\put(14.5,-109.5){\tiny\textit D}
\end{picture}
\begin{picture}(0,0)(10,436)
\put(-16,12){\tiny\textit A}
\put(14.5,12){\tiny\textit C}
\end{picture}
\put(-12,15){\tiny $0$}
\put(-30,4.5){\rotatebox{180}{\tiny $12$}} 
\put(4,4.5){\rotatebox{180}{\tiny $4$}}
\put(-30,-19.5){\rotatebox{180}{\tiny $13$}}
\put(4,-19.5){\rotatebox{180}{\tiny $17$}}
\put(-30,-44){\rotatebox{180}{\tiny $14$}}
\put(4,-44){\rotatebox{180}{\tiny $2$}}
\put(-26,-68){\rotatebox{180}{\tiny $3$}}
\put(4,-68){\rotatebox{180}{\tiny $15$}}
\put(-26,-92.5){\rotatebox{180}{\tiny $4$}}
\put(4,-92.5){\rotatebox{180}{\tiny $0$}}
\put(-26,-117){\rotatebox{180}{\tiny $5$}}
\put(4,-117){\rotatebox{180}{\tiny $13$}}
\put(-26,-139.5){\rotatebox{180}{\tiny $0$}}
\put(4,-139.5){\rotatebox{180}{\tiny $10$}}
\put(-26,-164){\rotatebox{180}{\tiny $1$}}
\put(4,-164){\rotatebox{180}{\tiny $5$}}
\put(-26,-189){\rotatebox{180}{\tiny $2$}}
\put(4,-189){\rotatebox{180}{\tiny $8$}}
\put(-26,-213){\rotatebox{180}{\tiny $9$}}
\put(4,-213){\rotatebox{180}{\tiny $3$}}
\put(-30,-238){\rotatebox{180}{\tiny $10$}}
\put(4,-238){\rotatebox{180}{\tiny $6$}}
\put(-30,-262){\rotatebox{180}{\tiny $11$}}
\put(4,-262){\rotatebox{180}{\tiny $1$}}
\put(-26,-285.5){\rotatebox{180}{\tiny $6$}}
\put(4,-285.5){\rotatebox{180}{\tiny $16$}}
\put(-26,-310){\rotatebox{180}{\tiny $7$}}
\put(4,-310){\rotatebox{180}{\tiny $11$}}
\put(-26,-334.5){\rotatebox{180}{\tiny $8$}}
\put(4,-334.5){\rotatebox{180}{\tiny $14$}}
\put(-30,-359){\rotatebox{180}{\tiny $15$}}
\put(4,-359){\rotatebox{180}{\tiny $9$}}
\put(-30,-383.5){\rotatebox{180}{\tiny $16$}}
\put(4,-383.5){\rotatebox{180}{\tiny $12$}}
\put(-30,-408){\rotatebox{180}{\tiny $17$}}
\put(4,-408){\rotatebox{180}{\tiny $7$}}
\put(-13,-428){\tiny $11$}
\end{picture}

\special{
psfile=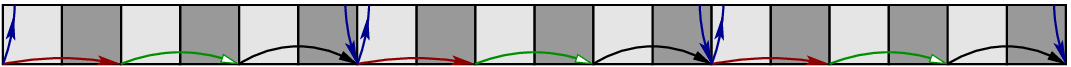
hscale=85
vscale=85
angle=90
hoffset=115
voffset=-528 
}
\begin{picture}(0,0)(-77.5,81) 
\begin{picture}(0,0)(0,0)
\put(2,4){\textcolor{blue}{$d_{2,2}$}}
\put(27,-25){$c_3$}
\put(27,-75){\textcolor{mygreen}{$b_3$}}
\put(27,-122){\textcolor{red}{$a_3$}}
\end{picture}
\begin{picture}(0,0)(2.5,146)
\put(2,9){\textcolor{blue}{$d_{3,1}$}}
\put(2,-11){\textcolor{blue}{$d_{1,2}$}}
\put(27,-25){$c_2$}
\put(27,-75){\textcolor{mygreen}{$b_2$}}
\put(27,-122){\textcolor{red}{$a_2$}}
\end{picture}
\begin{picture}(0,0)(5,291)
\put(2,9){\textcolor{blue}{$d_{2,1}$}}
\put(2,-11){\textcolor{blue}{$d_{3,2}$}}
\put(27,-25){$c_1$}
\put(27,-75){\textcolor{mygreen}{$b_1$}}
\put(27,-122){\textcolor{red}{$a_1$}}
\end{picture}
\begin{picture}(0,0)(7.5,437)
\put(2,9){\textcolor{blue}{$d_{1,1}$}}
\end{picture}
\end{picture}
\special{ 
psfile=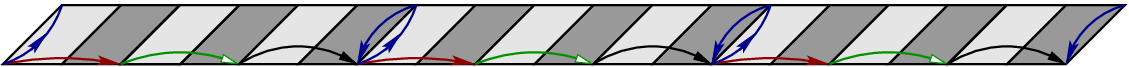
hscale=85
vscale=85
angle=90
hoffset=175.5
voffset=-516 
}
\begin{picture}(0,0)(-147,81) 
\begin{picture}(0,0)(0,0)
\put(2,0){\textcolor{blue}{$d'_{2,2}$}}
\put(27,-25){$c'_3$}
\put(27,-75){\textcolor{mygreen}{$b'_3$}}
\put(27,-122){\textcolor{red}{$a'_3$}}
\end{picture}
\begin{picture}(0,0)(2.5,146)
\put(2,26){\textcolor{blue}{$d'_{3,1}$}}
\put(2,2){\textcolor{blue}{$d'_{1,2}$}}
\put(27,-25){$c'_2$}
\put(27,-75){\textcolor{mygreen}{$b'_2$}}
\put(27,-122){\textcolor{red}{$a'_2$}}
\end{picture}
\begin{picture}(0,0)(5,291)
\put(2,26){\textcolor{blue}{$d'_{2,1}$}}
\put(2,2){\textcolor{blue}{$d'_{3,2}$}}
\put(27,-25){$c'_1$}
\put(27,-75){\textcolor{mygreen}{$b'_1$}}
\put(27,-122){\textcolor{red}{$a'_1$}}
\end{picture}
\begin{picture}(0,0)(7.5,437)
\put(2,27){\textcolor{blue}{$d'_{1,1}$}}
\end{picture}
\end{picture}
\special{ 
psfile=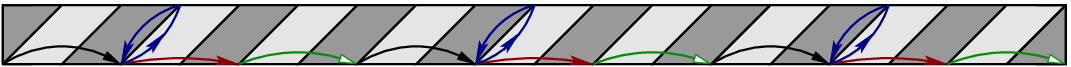
hscale=85
vscale=85
angle=90
hoffset=245.5
voffset=-420
}
\begin{picture}(0,0)(-217,82)
\begin{picture}(0,0)(0,0)
\put(27,70){\textcolor{mygreen}{$b'_1$}}
\put(27,23){\textcolor{red}{$a'_1$}}
\put(2,26){\textcolor{blue}{$d'_{1,1}$}}
\put(2,0){\textcolor{blue}{$d'_{2,2}$}}
\put(27,-25){$c'_3$}
\put(27,-75){\textcolor{mygreen}{$b'_3$}}
\put(27,-122){\textcolor{red}{$a'_3$}}
\end{picture}
\begin{picture}(0,0)(2.5,146)
\put(2,26){\textcolor{blue}{$d'_{3,1}$}}
\put(2,2){\textcolor{blue}{$d'_{1,2}$}}
\put(27,-25){$c'_2$}
\put(27,-75){\textcolor{mygreen}{$b'_2$}}
\put(27,-122){\textcolor{red}{$a'_2$}}
\end{picture}
\begin{picture}(0,0)(5,291)
\put(2,26){\textcolor{blue}{$d'_{2,1}$}}
\put(2,2){\textcolor{blue}{$d'_{3,2}$}}
\put(27,-25){$c'_1$}
\end{picture}
\end{picture}
\special{ 
psfile=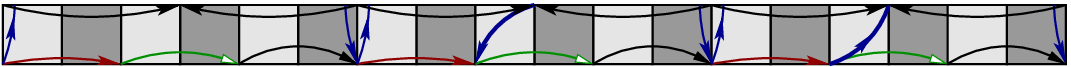
hscale=85
vscale=85
angle=90
hoffset=316
voffset=-420.5
}
\begin{picture}(0,0)(-287.5,-14.5)
\begin{picture}(0,0)(0,0)
\put(4,-10){\textcolor{blue}{$d_{2,2}$}}
\put(27,-25){$c_3$}
\put(-9,-36.5){$e_3$}
\put(27,-75){\textcolor{mygreen}{$b_3$}}
\put(5,-88){\textcolor{blue}{$d'_{1,1}$}}
\put(-9,-109){$e_1$}
\put(27,-122){\textcolor{red}{$a_3$}}
\end{picture}
\begin{picture}(0,0)(2.5,146)
\put(5,8){\textcolor{blue}{$d_{3,1}$}}
\put(4,-10){\textcolor{blue}{$d_{1,2}$}}
\put(27,-25){$c_2$}
\put(-9,-36.5){$e_2$}
\put(27,-75){\textcolor{mygreen}{$b_2$}}
\put(5,-79){\textcolor{blue}{$d'_{1,2}$}}
\put(-9,-109){$e_3$}
\put(27,-122){\textcolor{red}{$a_2$}}
\end{picture}
\begin{picture}(0,0)(5,291)
\put(4,9){\textcolor{blue}{$d_{2,1}$}}
\put(4,-10){\textcolor{blue}{$d_{3,2}$}}
\put(27,-25){$c_1$}
\put(-9,-36.5){$e_1$}
\put(27,-75){\textcolor{mygreen}{$b_1$}}
\put(-9,-109){$e_2$}
\put(27,-122){\textcolor{red}{$a_1$}}
\end{picture}
\begin{picture}(0,0)(7.5,437)
\put(4,9){\textcolor{blue}{$d_{1,1}$}}
\end{picture}
\end{picture}

\vspace*{501pt} 

\begin{figure}[bht]
 %
 %
\caption{
\label{fig:oneline:shear}
One-cylinder surface $\hat S_3$ (left two pictures) is sheared by $h$
(middle  picture);  then  cut  and  reglued  (the picture next to the
right)  to fit finally the initial surface $\hat S_3$ (the picture on
the right). }
\end{figure}



\clearpage

\special{
psfile=oneline18.eps
hscale=85
vscale=85
angle=90
hoffset=40
voffset=-448 
}
\begin{picture}(0,0)(-9,17)
\begin{picture}(0,0)(0,0)
\put(0,0){11}
\put(0,-24.5){10}
\put(3,-49){9}
\put(3,-73.5){8}
\put(3,-98){7}
\put(3,-122.5){6}
\end{picture}
\begin{picture}(0,0)(2.5,145)
\put(0,0){17}
\put(0,-24.5){16}
\put(0,-49){15}
\put(0,-73.5){14}
\put(0,-98){13}
\put(0,-122.5){12}
\end{picture}
\begin{picture}(0,0)(5,290)
\put(3,0){5}
\put(3,-24.5){4}
\put(3,-49){3}
\put(3,-73.5){2}
\put(3,-98){1}
\put(3,-122.5){0}
\end{picture}
\begin{picture}(0,0)(2.5,0)
\put(-16,12){\tiny\textit A}
\put(14.5,12){\tiny\textit C}
\put(14.5,-12.75){\tiny\textit F}
\put(14.5,-37){\tiny\textit C}
\put(-16,-61){\tiny\textit B}
\put(14.5,-61){\tiny\textit E}
\put(14.5,-86){\tiny\textit C}
\put(14.5,-109.5){\tiny\textit D}
\end{picture}
\begin{picture}(0,0)(5,145)
\put(-16,12){\tiny\textit A}
\put(14.5,12){\tiny\textit C}
\put(14.5,-12.75){\tiny\textit F}
\put(14.5,-37){\tiny\textit C}
\put(-16,-61){\tiny\textit B}
\put(14.5,-61){\tiny\textit E}
\put(14.5,-86){\tiny\textit C}
\put(14.5,-109.5){\tiny\textit D}
\end{picture}
\begin{picture}(0,0)(7.5,290)
\put(-16,12){\tiny\textit A}
\put(14.5,12){\tiny\textit C}
\put(14.5,-12.75){\tiny\textit F}
\put(14.5,-37){\tiny\textit C}
\put(-16,-61){\tiny\textit B}
\put(14.5,-61){\tiny\textit E}
\put(14.5,-86){\tiny\textit C}
\put(14.5,-109.5){\tiny\textit D}
\end{picture}
\begin{picture}(0,0)(10,436)
\put(-16,12){\tiny\textit A}
\put(14.5,12){\tiny\textit C}
\end{picture}
\put(-12,15){\tiny $0$}
\put(-30,4.5){\rotatebox{180}{\tiny $12$}} 
\put(4,4.5){\rotatebox{180}{\tiny $4$}}
\put(-30,-19.5){\rotatebox{180}{\tiny $13$}}
\put(4,-19.5){\rotatebox{180}{\tiny $17$}}
\put(-30,-44){\rotatebox{180}{\tiny $14$}}
\put(4,-44){\rotatebox{180}{\tiny $2$}}
\put(-26,-68){\rotatebox{180}{\tiny $3$}}
\put(4,-68){\rotatebox{180}{\tiny $15$}}
\put(-26,-92.5){\rotatebox{180}{\tiny $4$}}
\put(4,-92.5){\rotatebox{180}{\tiny $0$}}
\put(-26,-117){\rotatebox{180}{\tiny $5$}}
\put(4,-117){\rotatebox{180}{\tiny $13$}}
\put(-26,-139.5){\rotatebox{180}{\tiny $0$}}
\put(4,-139.5){\rotatebox{180}{\tiny $10$}}
\put(-26,-164){\rotatebox{180}{\tiny $1$}}
\put(4,-164){\rotatebox{180}{\tiny $5$}}
\put(-26,-189){\rotatebox{180}{\tiny $2$}}
\put(4,-189){\rotatebox{180}{\tiny $8$}}
\put(-26,-213){\rotatebox{180}{\tiny $9$}}
\put(4,-213){\rotatebox{180}{\tiny $3$}}
\put(-30,-238){\rotatebox{180}{\tiny $10$}}
\put(4,-238){\rotatebox{180}{\tiny $6$}}
\put(-30,-262){\rotatebox{180}{\tiny $11$}}
\put(4,-262){\rotatebox{180}{\tiny $1$}}
\put(-26,-285.5){\rotatebox{180}{\tiny $6$}}
\put(4,-285.5){\rotatebox{180}{\tiny $16$}}
\put(-26,-310){\rotatebox{180}{\tiny $7$}}
\put(4,-310){\rotatebox{180}{\tiny $11$}}
\put(-26,-334.5){\rotatebox{180}{\tiny $8$}}
\put(4,-334.5){\rotatebox{180}{\tiny $14$}}
\put(-30,-359){\rotatebox{180}{\tiny $15$}}
\put(4,-359){\rotatebox{180}{\tiny $9$}}
\put(-30,-383.5){\rotatebox{180}{\tiny $16$}}
\put(4,-383.5){\rotatebox{180}{\tiny $12$}}
\put(-30,-408){\rotatebox{180}{\tiny $17$}}
\put(4,-408){\rotatebox{180}{\tiny $7$}}
\put(-13,-428){\tiny $11$}
\put(5,25){$\hat S_3$} 
\end{picture}

\special{
psfile=oneline18_arcs.eps
hscale=85
vscale=85
angle=90
hoffset=78
voffset=-438
}
\begin{picture}(0,0)(-40.5,-9) 
\begin{picture}(0,0)(0,0)
\put(2,4){\textcolor{blue}{$d_{2,2}$}}
\put(27,-25){$c_3$}
\put(27,-75){\textcolor{mygreen}{$b_3$}}
\put(27,-122){\textcolor{red}{$a_3$}}
\end{picture}
\begin{picture}(0,0)(2.5,146)
\put(2,9){\textcolor{blue}{$d_{3,1}$}}
\put(2,-11){\textcolor{blue}{$d_{1,2}$}}
\put(27,-25){$c_2$}
\put(27,-75){\textcolor{mygreen}{$b_2$}}
\put(27,-122){\textcolor{red}{$a_2$}}
\end{picture}
\begin{picture}(0,0)(5,291)
\put(2,9){\textcolor{blue}{$d_{2,1}$}}
\put(2,-11){\textcolor{blue}{$d_{3,2}$}}
\put(27,-25){$c_1$}
\put(27,-75){\textcolor{mygreen}{$b_1$}}
\put(27,-122){\textcolor{red}{$a_1$}}
\end{picture}
\begin{picture}(0,0)(7.5,437)
\put(2,9){\textcolor{blue}{$d_{1,1}$}}
\end{picture}
\end{picture}


\special{
psfile=twolines18_0.eps
hscale=85
vscale=85
angle=180
hoffset=305
voffset=-452
}
\begin{picture}(0,0)(-33,459)
\put(-24.5,7){\rotatebox{180}{11}}
\put(0,0.25){12}
\put(27,7){\rotatebox{180}{1}}
\put(49,0.25){16}
\put(75,7){\rotatebox{180}{5}}
\put(99.5,0.25){6}
\put(123,7){\rotatebox{180}{13}}
\put(146.5,0.25){10}
\put(171,7){\rotatebox{180}{17}}
\put(197,0.25){0}
\put(220,7){\rotatebox{180}{7}}
\put(244.5,0.25){4}
\put(27,-16.6){\rotatebox{180}{2}}
\put(49,-24.25){15}
\put(123,-16.5){\rotatebox{180}{14}}
\put(148.5,-24.25){9}
\put(220,-16.5){\rotatebox{180}{8}}
\put(244.5,-24.25){3}
\put(-21.5,22){\rotatebox{180}{\tiny 10}}
\put(2.5,18){\tiny 13}
\put(28.5,22){\rotatebox{180}{\tiny 0}}
\put(51,18){\tiny 17}
\put(76.5,22){\rotatebox{180}{\tiny 4}}
\put(100.5,18){\tiny 7}
\put(123.5,22){\rotatebox{180}{\tiny 12}}
\put(143,18){\tiny 11}
\put(171.5,22){\rotatebox{180}{\tiny 16}}
\put(198,18){\tiny 1}
\put(221,22){\rotatebox{180}{\tiny 6}}
\put(246,18){\tiny 5}
\put(-35,1){\tiny 4}
\put(261,5){\rotatebox{180}{\tiny 11}}
   %
   %
\put(-21,-10){\rotatebox{180}{\tiny 0}}
\put(3.5,-14){\tiny 5}
\put(75,-10){\rotatebox{180}{\tiny 12}}
\put(99,-14){\tiny 17}
\put(174,-10){\rotatebox{180}{\tiny 6}}
\put(196.5,-14){\tiny 11}
\put(14,-22){\tiny 9}
\put(67.5,-18){\rotatebox{180}{\tiny 8}}
\put(110,-22){\tiny 3}
\put(164,-18){\rotatebox{180}{\tiny 2}}
\put(204,-22){\tiny 15}
\put(261,-18){\rotatebox{180}{\tiny 14}}
\put(28.5,-34){\rotatebox{180}{\tiny 3}}
\put(51,-38){\tiny 14}
\put(124.5,-34){\rotatebox{180}{\tiny 15}}
\put(149,-38){\tiny 8}
\put(221,-34){\rotatebox{180}{\tiny 9}}
\put(245,-38){\tiny 2}
\end{picture}

\special{
psfile=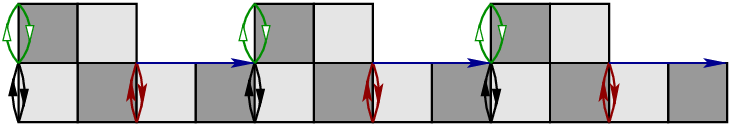
hscale=85
vscale=85
angle=90
angle=180
hoffset=311  %
voffset=-501 
}
\begin{picture}(0,0)(-30,506)
\begin{picture}(0,0)(0,0)
\put(0,0){\textcolor{red}{$a_{2,1}$}}
      \put(-8,-19){\textcolor{blue}{$d_2$}}
\put(25,0){\textcolor{red}{$a_{1,2}$}}
\put(50,0){\textcolor{black}{$c_{2,1}$}}
\put(74,0){\textcolor{black}{$c_{1,2}$}}
\put(48,-24.5){\textcolor{mygreen}{$b_{2,2}$}}
\put(75,-24.5){\textcolor{mygreen}{$b_{3,1}$}}
\end{picture}
\begin{picture}(0,0)(-94.5,0)
\put(0,0){\textcolor{red}{$a_{3,1}$}}
      \put(-8,-19){\textcolor{blue}{$d_3$}}
\put(25,0){\textcolor{red}{$a_{2,2}$}}
\put(50,0){\textcolor{black}{$c_{3,1}$}}
\put(74,0){\textcolor{black}{$c_{2,2}$}}
\put(48,-24.5){\textcolor{mygreen}{$b_{3,2}$}}
\put(75,-24.5){\textcolor{mygreen}{$b_{1,1}$}}
\end{picture}
\begin{picture}(0,0)(-189,0)
\put(0,0){\textcolor{red}{$a_{1,1}$}}
      \put(-8,-19){\textcolor{blue}{$d_1$}}
\put(25,0){\textcolor{red}{$a_{3,2}$}}
\put(50,0){\textcolor{black}{$c_{1,1}$}}
\put(74,0){\textcolor{black}{$c_{3,2}$}}
\put(48,-24.5){\textcolor{mygreen}{$b_{1,2}$}}
\put(75,-24.5){\textcolor{mygreen}{$b_{2,1}$}}
\end{picture}
\end{picture}


\special{
psfile=twolines18_1.eps
hscale=85
vscale=85
angle=90
hoffset=294.5
voffset=-263.5
}
\begin{picture}(0,0)(-259.5,-182) 
\begin{picture}(0,0)(2.5,244)
    \put(-10,81){\textcolor{blue}{$d_2$}}
\put(5,66.5){\textcolor{red}{$a_{2,1}$}}
\put(5,50){\textcolor{red}{$a_{1,2}$}}
\put(5,18.5){\textcolor{black}{$c_{2,1}$}}
\put(5,-1){\textcolor{black}{$c_{1,2}$}}
\put(-19.5,18.5){\textcolor{mygreen}{$b_{2,2}$}}
\put(-19.55,-1){\textcolor{mygreen}{$b_{3,1}$}}
\end{picture}
\begin{picture}(0,0)(5,341)
    \put(-10,81){\textcolor{blue}{$d_3$}}
\put(5,66.5){\textcolor{red}{$a_{3,1}$}}
\put(5,50){\textcolor{red}{$a_{2,2}$}}
\put(5,18.5){\textcolor{black}{$c_{3,1}$}}
\put(5,-1){\textcolor{black}{$c_{2,2}$}}
\put(-19.5,18.5){\textcolor{mygreen}{$b_{3,2}$}}
\put(-19.55,-1){\textcolor{mygreen}{$b_{1,1}$}}
\end{picture}
\begin{picture}(0,0)(7.5,437)
    \put(-10,81){\textcolor{blue}{$d_1$}}
\put(5,66.5){\textcolor{red}{$a_{1,1}$}}
\put(5,50){\textcolor{red}{$a_{3,2}$}}
\put(5,18.5){\textcolor{black}{$c_{1,1}$}}
\put(5,-1){\textcolor{black}{$c_{3,2}$}}
\put(-19.5,18.5){\textcolor{mygreen}{$b_{1,2}$}}
\put(-19.55,-1){\textcolor{mygreen}{$b_{2,1}$}}
\end{picture}
\end{picture}

\special{
psfile=twolines18_0.eps
hscale=85
vscale=85
angle=90
hoffset=356    
voffset=-245 
}
\begin{picture}(0,0)(-326,-189) 
\begin{picture}(0,0)(0,97.5)
\put(-24.5,-98){17}
\put(-24.5,-122.5){16}
\put(0,-49){15}
\put(0,-73.5){14}
\put(0,-98){13}
\put(0,-122.5){12}
\end{picture}
\begin{picture}(0,0)(2.5,193)
\put(-24.5,-98){11}
\put(-24.5,-122.5){10}
\put(3,-49){9}
\put(3,-73.5){8}
\put(3,-98){7}
\put(3,-122.5){6}
\end{picture}
\begin{picture}(0,0)(5,290)
\put(-21.5,-98){5}
\put(-21.5,-122.5){4}
\put(3,-49){3}
\put(3,-73.5){2}
\put(3,-98){1}
\put(3,-122.5){0}
\end{picture}
\begin{picture}(0,0)(2.5,145)
\put(-16,12){\tiny\textit C}
\put(14.5,12){\tiny\textit F}
\put(-16,-12.75){\tiny\textit A}
\put(-40,-37){\tiny\textit E}
\put(-16,-34){\tiny\textit C}
\put(14.5,-37){\tiny\textit D}
\put(-40,-61){\tiny\textit B}
\put(-40,-85){\tiny\textit E}
\put(-16,-88){\tiny\textit C}
\put(14.5,-85){\tiny\textit F}
\end{picture}
\begin{picture}(0,0)(5,241.5)
\put(-16,-12.75){\tiny\textit A}
\put(-40,-37){\tiny\textit E}
\put(-16,-34){\tiny\textit C}
\put(14.5,-37){\tiny\textit D}
\put(-40,-61){\tiny\textit B}
\put(-40,-85){\tiny\textit E}
\put(-16,-88.5){\tiny\textit C}
\put(14.5,-85){\tiny\textit F}
\end{picture}
\begin{picture}(0,0)(7.5,338)
\put(-16,-12.75){\tiny\textit A}
\put(-40,-37){\tiny\textit E}
\put(-16,-34){\tiny\textit C}
\put(14.5,-37){\tiny\textit D}
\put(-40,-61.5){\tiny\textit B}
\put(-40,-86){\tiny\textit E}
\put(-14,-89){\tiny\textit C}
\put(14.5,-86){\tiny\textit F}
\end{picture}
\begin{picture}(0,0)(-0.5,145)
\put(-12,15){\tiny $0$}
\put(-26,4.5){\rotatebox{180}{\tiny $2$}} 
\put(4,4.5){\rotatebox{180}{\tiny $6$}}
\put(-26,-19.5){\rotatebox{180}{\tiny $9$}}
\put(4,-19.5){\rotatebox{180}{\tiny $7$}}
\put(-39,-34){\tiny $10$}
\put(-52,-45){\rotatebox{180}{\tiny $4$}}
\put(4,-45){\rotatebox{180}{\tiny $2$}}
\put(-54,-70){\rotatebox{180}{\tiny $11$}}
\put(4,-70){\rotatebox{180}{\tiny $3$}}
\put(-37,-88){\tiny $5$}
\put(-30,-93){\rotatebox{180}{\tiny $14$}}
\put(4,-93){\rotatebox{180}{\tiny $0$}}
\put(-26,-117){\rotatebox{180}{\tiny $3$}}
\put(4,-117){\rotatebox{180}{\tiny $1$}}
\put(-37,-130){\tiny $4$}
\put(-54,-140.5){\rotatebox{180}{\tiny $16$}}
\put(4,-140.5){\rotatebox{180}{\tiny $14$}}
\put(-50,-165){\rotatebox{180}{\tiny $5$}}
\put(4,-165){\rotatebox{180}{\tiny $15$}}
\put(-39,-185){\tiny $17$}
\put(-26,-189.5){\rotatebox{180}{\tiny $8$}}
\put(4,-189.5){\rotatebox{180}{\tiny $12$}}
\put(-30,-213){\rotatebox{180}{\tiny $15$}}
\put(4,-213){\rotatebox{180}{\tiny $13$}}
\put(-39,-227){\tiny $16$}
\put(-54,-238){\rotatebox{180}{\tiny $10$}}
\put(4,-238){\rotatebox{180}{\tiny $8$}}
\put(-54,-262){\rotatebox{180}{\tiny $17$}}
\put(4,-262){\rotatebox{180}{\tiny $9$}}
\put(-39,-282){\tiny $11$}
\put(-14.5,-282){\tiny $15$}
\put(-47,24){$\hat S_2$}
\end{picture}
\end{picture}

\vspace*{261pt}

\begin{figure}[hbt]
%
%
\caption{
\label{fig:oneline:rotate}
Rotation $r_3:\hat S_3\to \hat S_2$.
}
\end{figure}

\vspace*{-10pt}\hspace*{89.75truept}
\begin{minipage}{151truept} 
We  pass from the original horizontal cylinder decomposition of $\hat
S_3$  (left  two  pictures)  to  the  vertical cylinder decomposition
(bottom two pictures) and then rotate $\hat S_3$ by $\pi/2$ clockwise
(right two pictures). The squares are renumbered after the rotation.
\end{minipage}



\clearpage

\special{
psfile=twolines18_0.eps
hscale=85
vscale=85
angle=90
hoffset=58
voffset=-338.5 
}
\begin{picture}(0,0)(-27,-93) 
\begin{picture}(0,0)(0,97.5)
\put(-24.5,-98){17}
\put(-24.5,-122.5){16}
\put(0,-49){15}
\put(0,-73.5){14}
\put(0,-98){13}
\put(0,-122.5){12}
\end{picture}
\begin{picture}(0,0)(2.5,193)
\put(-24.5,-98){11}
\put(-24.5,-122.5){10}
\put(3,-49){9}
\put(3,-73.5){8}
\put(3,-98){7}
\put(3,-122.5){6}
\end{picture}
\begin{picture}(0,0)(5,290)
\put(-21.5,-98){5}
\put(-21.5,-122.5){4}
\put(3,-49){3}
\put(3,-73.5){2}
\put(3,-98){1}
\put(3,-122.5){0}
\end{picture}
\begin{picture}(0,0)(2.5,145)
\put(-16,12){\tiny\textit C}
\put(14.5,12){\tiny\textit F}
\put(-16,-12.75){\tiny\textit A}
\put(-40,-37){\tiny\textit E}
\put(-16,-34){\tiny\textit C}
\put(14.5,-37){\tiny\textit D}
\put(-40,-61){\tiny\textit B}
\put(-40,-85){\tiny\textit E}
\put(-16,-88){\tiny\textit C}
\put(14.5,-85){\tiny\textit F}
\end{picture}
\begin{picture}(0,0)(5,241.5)
\put(-16,-12.75){\tiny\textit A}
\put(-40,-37){\tiny\textit E}
\put(-16,-34){\tiny\textit C}
\put(14.5,-37){\tiny\textit D}
\put(-40,-61){\tiny\textit B}
\put(-40,-85){\tiny\textit E}
\put(-16,-88.5){\tiny\textit C}
\put(14.5,-85){\tiny\textit F}
\end{picture}
\begin{picture}(0,0)(7.5,338)
\put(-16,-12.75){\tiny\textit A}
\put(-40,-37){\tiny\textit E}
\put(-16,-34){\tiny\textit C}
\put(14.5,-37){\tiny\textit D}
\put(-40,-61.5){\tiny\textit B}
\put(-40,-86){\tiny\textit E}
\put(-14,-89){\tiny\textit C}
\put(14.5,-86){\tiny\textit F}
\end{picture}
\begin{picture}(0,0)(-0.5,145)
\put(-12,15){\tiny $0$}
\put(-26,4.5){\rotatebox{180}{\tiny $2$}} 
\put(4,4.5){\rotatebox{180}{\tiny $6$}}
\put(-26,-19.5){\rotatebox{180}{\tiny $9$}}
\put(4,-19.5){\rotatebox{180}{\tiny $7$}}
\put(-39,-34){\tiny $10$}
\put(-52,-45){\rotatebox{180}{\tiny $4$}}
\put(4,-45){\rotatebox{180}{\tiny $2$}}
\put(-54,-70){\rotatebox{180}{\tiny $11$}}
\put(4,-70){\rotatebox{180}{\tiny $3$}}
\put(-37,-88){\tiny $5$}
\put(-30,-93){\rotatebox{180}{\tiny $14$}}
\put(4,-93){\rotatebox{180}{\tiny $0$}}
\put(-26,-117){\rotatebox{180}{\tiny $3$}}
\put(4,-117){\rotatebox{180}{\tiny $1$}}
\put(-37,-130){\tiny $4$}
\put(-54,-140.5){\rotatebox{180}{\tiny $16$}}
\put(4,-140.5){\rotatebox{180}{\tiny $14$}}
\put(-50,-165){\rotatebox{180}{\tiny $5$}}
\put(4,-165){\rotatebox{180}{\tiny $15$}}
\put(-39,-185){\tiny $17$}
\put(-26,-189.5){\rotatebox{180}{\tiny $8$}}
\put(4,-189.5){\rotatebox{180}{\tiny $12$}}
\put(-30,-213){\rotatebox{180}{\tiny $15$}}
\put(4,-213){\rotatebox{180}{\tiny $13$}}
\put(-39,-227){\tiny $16$}
\put(-54,-238){\rotatebox{180}{\tiny $10$}}
\put(4,-238){\rotatebox{180}{\tiny $8$}}
\put(-54,-262){\rotatebox{180}{\tiny $17$}}
\put(4,-262){\rotatebox{180}{\tiny $9$}}
\put(-39,-282){\tiny $11$}
\put(-14.5,-282){\tiny $15$}
\put(25,25){$\hat S_2$} 
\end{picture}
\end{picture}

\special{
psfile=twolines18_1.eps
hscale=85
vscale=85
angle=90
hoffset=131
voffset=-335 
}
\begin{picture}(0,0)(-96,-110.5) 
\begin{picture}(0,0)(2.5,244)
\put(-10,81){\textcolor{blue}{$d_2$}}
\put(5,66.5){\textcolor{red}{$a_{2,1}$}}
\put(5,50){\textcolor{red}{$a_{1,2}$}}
\put(5,18.5){\textcolor{black}{$c_{2,1}$}}
\put(5,-1){\textcolor{black}{$c_{1,2}$}}
\put(-19.5,18.5){\textcolor{mygreen}{$b_{2,2}$}}
\put(-19.55,-1){\textcolor{mygreen}{$b_{3,1}$}}
\end{picture}
\begin{picture}(0,0)(5,341)
\put(-10,81){\textcolor{blue}{$d_3$}}
\put(5,66.5){\textcolor{red}{$a_{3,1}$}}
\put(5,50){\textcolor{red}{$a_{2,2}$}}
\put(5,18.5){\textcolor{black}{$c_{3,1}$}}
\put(5,-1){\textcolor{black}{$c_{2,2}$}}
\put(-19.5,18.5){\textcolor{mygreen}{$b_{3,2}$}}
\put(-19.55,-1){\textcolor{mygreen}{$b_{1,1}$}}
\end{picture}
\begin{picture}(0,0)(7.5,437)
\put(-10,81){\textcolor{blue}{$d_1$}}
\put(5,66.5){\textcolor{red}{$a_{1,1}$}}
\put(5,50){\textcolor{red}{$a_{3,2}$}}
\put(5,18.5){\textcolor{black}{$c_{1,1}$}}
\put(5,-1){\textcolor{black}{$c_{3,2}$}}
\put(-19.5,18.5){\textcolor{mygreen}{$b_{1,2}$}}
\put(-19.55,-1){\textcolor{mygreen}{$b_{2,1}$}}
\end{picture}
\end{picture}
\special{ 
psfile=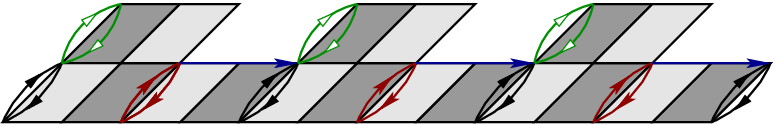
hscale=85
vscale=85
angle=90
hoffset=189
voffset=-316.5 
}
\begin{picture}(0,0)(-163.5,-120) 
\begin{picture}(0,0)(2.5,244)
\put(10.5,115.5){\textcolor{black}{$c'_{1,1}$}}
\put(1,100.5){\textcolor{black}{$c'_{3,2}$}}
\put(-10,95){\textcolor{blue}{$d'_2$}}
\put(10,68){\textcolor{red}{$a'_{2,1}$}}
\put(1,52){\textcolor{red}{$a'_{1,2}$}}
\put(10.5,19.5){\textcolor{black}{$c'_{2,1}$}}
\put(1.25,4){\textcolor{black}{$c'_{1,2}$}}
\put(-14.5,46){\textcolor{mygreen}{$b'_{2,2}$}}
\put(-23,25){\textcolor{mygreen}{$b'_{3,1}$}}
\end{picture}
\begin{picture}(0,0)(5,341)
\put(-10,95){\textcolor{blue}{$d'_3$}}
\put(10,68){\textcolor{red}{$a'_{3,1}$}}
\put(1,52){\textcolor{red}{$a'_{2,2}$}}
\put(10.5,19.5){\textcolor{black}{$c'_{3,1}$}}
\put(1.25,4){\textcolor{black}{$c'_{2,2}$}}
\put(-14.5,46){\textcolor{mygreen}{$b'_{3,2}$}}
\put(-23,25){\textcolor{mygreen}{$b'_{1,1}$}}
\end{picture}
\begin{picture}(0,0)(7.5,437)
\put(-10,95){\textcolor{blue}{$d'_1$}}
\put(10,68){\textcolor{red}{$a'_{1,1}$}}
\put(1,52){\textcolor{red}{$a'_{3,2}$}}
\put(-14.5,46){\textcolor{mygreen}{$b'_{1,2}$}}
\put(-23,25){\textcolor{mygreen}{$b'_{2,1}$}}
  \put(10.5,19.5){\textcolor{black}{$c'_{1,1}$}}
  \put(1,2){\textcolor{black}{$c'_{3,2}$}}
\end{picture}
\end{picture}
\special{ 
psfile=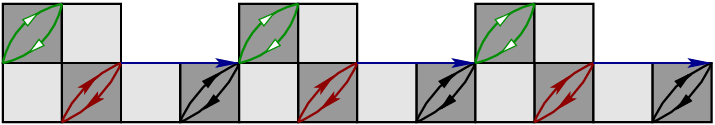
hscale=85
vscale=85
angle=90
hoffset=257.5
voffset=-293 
}
\begin{picture}(0,0)(-232,-120) 
\begin{picture}(0,0)(2.5,244)
\put(10.5,115.5){\textcolor{black}{$c_{1,1}$}}
\put(1,100.5){\textcolor{black}{$c_{3,2}$}}
\put(-10,95){\textcolor{blue}{$d_2$}}
\put(10,68){\textcolor{red}{$a_{2,1}$}}
\put(1,52){\textcolor{red}{$a_{1,2}$}}
\put(10.5,19.5){\textcolor{black}{$c_{2,1}$}}
\put(1.25,4){\textcolor{black}{$c_{1,2}$}}
\put(-14.5,46){\textcolor{mygreen}{$b_{2,2}$}}
\put(-23,25){\textcolor{mygreen}{$b_{3,1}$}}
\end{picture}
\begin{picture}(0,0)(5,341)
\put(10,68){\textcolor{red}{$a_{3,1}$}}
\put(1,52){\textcolor{red}{$a_{2,2}$}}
\put(-10,95){\textcolor{blue}{$d_3$}}
\put(10.5,19.5){\textcolor{black}{$c_{3,1}$}}
\put(1.25,4){\textcolor{black}{$c_{2,2}$}}
\put(-14.5,46){\textcolor{mygreen}{$b_{3,2}$}}
\put(-23,25){\textcolor{mygreen}{$b_{1,1}$}}
\end{picture}
\begin{picture}(0,0)(7.5,437)
\put(10,68){\textcolor{red}{$a_{1,1}$}}
\put(1,52){\textcolor{red}{$a_{3,2}$}}
\put(-10,95){\textcolor{blue}{$d_1$}}
\put(-14.5,46){\textcolor{mygreen}{$b_{1,2}$}}
\put(-23,25){\textcolor{mygreen}{$b_{2,1}$}}
\end{picture}
\end{picture}

\special{ 
psfile=twolines18_0.eps
hscale=85
vscale=85
angle=90
hoffset=350 
voffset=-293
}
\begin{picture}(0,0)(-320,-140.75) 
\begin{picture}(0,0)(0,97.5)
\put(-24.5,-98){17}
\put(-24.5,-122.5){16}
\put(0,-49){15}
\put(0,-73.5){14}
\put(0,-98){13}
\put(0,-122.5){12}
\end{picture}
\begin{picture}(0,0)(2.5,193)
\put(-24.5,-98){11}
\put(-24.5,-122.5){10}
\put(3,-49){9}
\put(3,-73.5){8}
\put(3,-98){7}
\put(3,-122.5){6}
\end{picture}
\begin{picture}(0,0)(5,290)
\put(-21.5,-98){5}
\put(-21.5,-122.5){4}
\put(3,-49){3}
\put(3,-73.5){2}
\put(3,-98){1}
\put(3,-122.5){0}
\end{picture}
\begin{picture}(0,0)(2.5,144.5)
\put(-16,12){\tiny\textit C}
\put(14.5,-12.75){\tiny\textit F}
\put(-16,-12.75){\tiny\textit A}
\put(-40,-37){\tiny\textit B}
\put(-16,-34){\tiny\textit C}
\put(14.5,-61){\tiny\textit D}
\put(-40,-61){\tiny\textit E}
\put(-40,-85){\tiny\textit B}
\put(-16,-88){\tiny\textit C}
\end{picture}
\begin{picture}(0,0)(5,241.5)
\put(14.5,-12.75){\tiny\textit F}
\put(-16,-12.75){\tiny\textit A}
\put(-40,-37){\tiny\textit B}
\put(-16,-34){\tiny\textit C}
\put(14.5,-61){\tiny\textit D}
\put(-40,-61){\tiny\textit E}
\put(-40,-85){\tiny\textit B}
\put(-16,-88){\tiny\textit C}
\end{picture}
\begin{picture}(0,0)(7.5,338)
\put(14.5,-12.75){\tiny\textit F}
\put(-16,-12.75){\tiny\textit A}
\put(-40,-37){\tiny\textit B}
\put(-16,-34){\tiny\textit C}
\put(14.5,-61){\tiny\textit D}
\put(-40,-61){\tiny\textit E}
\put(-40,-89){\tiny\textit B}
\put(-14,-89){\tiny\textit C}
\end{picture}
\begin{picture}(0,0)(-0.5,145)
\put(-12,15){\tiny $0$}
\put(-26,4.5){\rotatebox{180}{\tiny $2$}} 
\put(4,4.5){\rotatebox{180}{\tiny $8$}}
\put(-26,-19.5){\rotatebox{180}{\tiny $9$}}
\put(4,-19.5){\rotatebox{180}{\tiny $3$}}
\put(-39,-34){\tiny $10$}
\put(-52,-45){\rotatebox{180}{\tiny $4$}}
\put(4,-45){\rotatebox{180}{\tiny $6$}}
\put(-54,-70){\rotatebox{180}{\tiny $11$}}
\put(4,-70){\rotatebox{180}{\tiny $1$}}
\put(-37,-88){\tiny $5$}
\put(-30,-93){\rotatebox{180}{\tiny $14$}}
\put(4,-93){\rotatebox{180}{\tiny $2$}}
\put(-26,-117){\rotatebox{180}{\tiny $3$}}
\put(4,-117){\rotatebox{180}{\tiny $15$}}
\put(-37,-130){\tiny $4$}
\put(-54,-140.5){\rotatebox{180}{\tiny $16$}}
\put(4,-140.5){\rotatebox{180}{\tiny $0$}}
\put(-50,-165){\rotatebox{180}{\tiny $5$}}
\put(4,-165){\rotatebox{180}{\tiny $13$}}
\put(-39,-185){\tiny $17$}
\put(-26,-189.5){\rotatebox{180}{\tiny $8$}}
\put(4,-189.5){\rotatebox{180}{\tiny $14$}}
\put(-30,-213){\rotatebox{180}{\tiny $15$}}
\put(4,-213){\rotatebox{180}{\tiny $9$}}
\put(-39,-227){\tiny $16$}
\put(-54,-238){\rotatebox{180}{\tiny $10$}}
\put(4,-238){\rotatebox{180}{\tiny $12$}}
\put(-54,-262){\rotatebox{180}{\tiny $17$}}
\put(4,-262){\rotatebox{180}{\tiny $7$}}
\put(-39,-282){\tiny $11$}
\put(-14.5,-282){\tiny $15$}
\end{picture}
\put(-53,-120){$\hat S_1$}
\end{picture}

\vspace*{315pt}

\begin{figure}[bht]
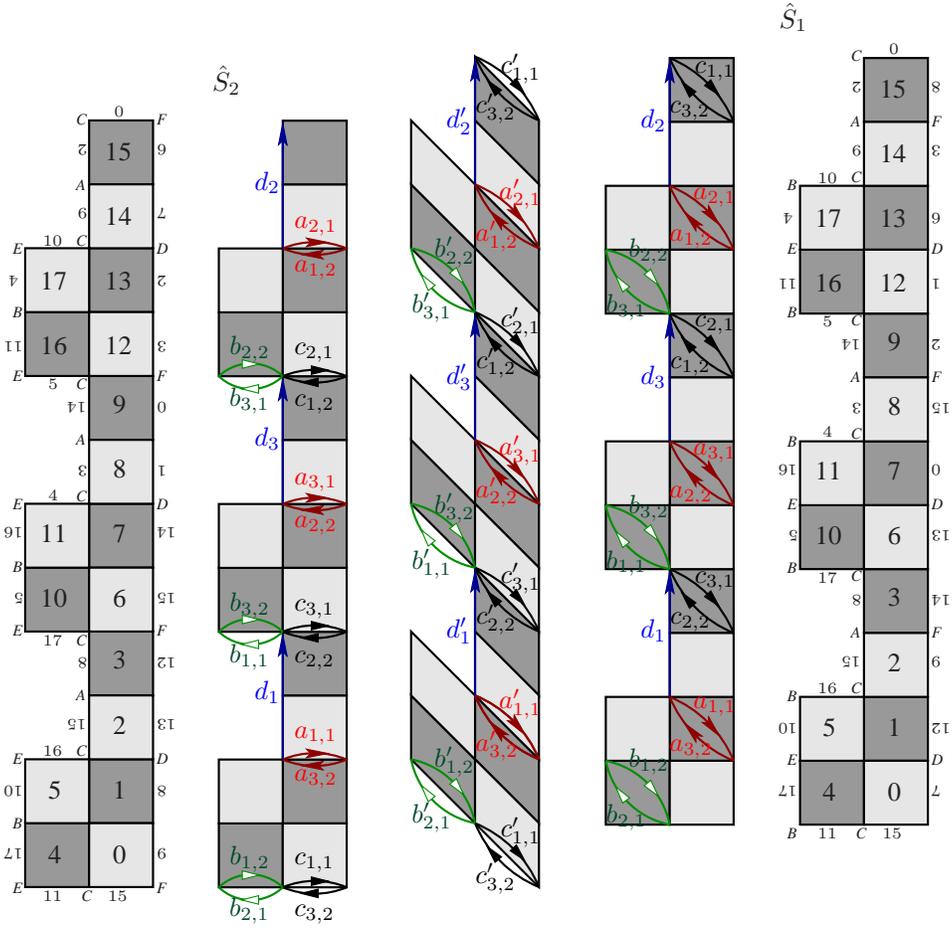

 %
 %
\caption{
\label{fig:twoline:one:shear}
Two-cylinder surface $\hat S_2$ (left two pictures) is sheared by $h$
(middle  picture);  then  cut  and  reglued  (the picture next to the
right) to produce the two-cylinder surface $\hat S_1$ (the picture on
the right).
}
\end{figure}


\clearpage

\special{
psfile=twolines18_0.eps
hscale=85
vscale=85
angle=90
hoffset=58
voffset=-363
}
\begin{picture}(0,0)(-27,-68.5) 
\begin{picture}(0,0)(0,97.5)
\put(-24.5,-98){17}
\put(-24.5,-122.5){16}
\put(0,-49){15}
\put(0,-73.5){14}
\put(0,-98){13}
\put(0,-122.5){12}
\end{picture}
\begin{picture}(0,0)(2.5,193)
\put(-24.5,-98){11}
\put(-24.5,-122.5){10}
\put(3,-49){9}
\put(3,-73.5){8}
\put(3,-98){7}
\put(3,-122.5){6}
\end{picture}
\begin{picture}(0,0)(5,290)
\put(-21.5,-98){5}
\put(-21.5,-122.5){4}
\put(3,-49){3}
\put(3,-73.5){2}
\put(3,-98){1}
\put(3,-122.5){0}
\end{picture}
\begin{picture}(0,0)(2.5,145)
\put(-16,12){\tiny\textit C}
\put(14.5,12){\tiny\textit F}
\put(-16,-12.75){\tiny\textit A}
\put(-40,-37){\tiny\textit E}
\put(-16,-34){\tiny\textit C}
\put(14.5,-37){\tiny\textit D}
\put(-40,-61){\tiny\textit B}
\put(-40,-85){\tiny\textit E}
\put(-16,-88){\tiny\textit C}
\put(14.5,-85){\tiny\textit F}
\end{picture}
\begin{picture}(0,0)(5,241.5)
\put(-16,-12.75){\tiny\textit A}
\put(-40,-37){\tiny\textit E}
\put(-16,-34){\tiny\textit C}
\put(14.5,-37){\tiny\textit D}
\put(-40,-61){\tiny\textit B}
\put(-40,-85){\tiny\textit E}
\put(-16,-88.5){\tiny\textit C}
\put(14.5,-85){\tiny\textit F}
\end{picture}
\begin{picture}(0,0)(7.5,338)
\put(-16,-12.75){\tiny\textit A}
\put(-40,-37){\tiny\textit E}
\put(-16,-34){\tiny\textit C}
\put(14.5,-37){\tiny\textit D}
\put(-40,-61.5){\tiny\textit B}
\put(-40,-86){\tiny\textit E}
\put(-14,-89){\tiny\textit C}
\put(14.5,-86){\tiny\textit F}
\end{picture}
\begin{picture}(0,0)(-0.5,145)
\put(-12,15){\tiny $0$}
\put(-26,4.5){\rotatebox{180}{\tiny $2$}} 
\put(4,4.5){\rotatebox{180}{\tiny $6$}}
\put(-26,-19.5){\rotatebox{180}{\tiny $9$}}
\put(4,-19.5){\rotatebox{180}{\tiny $7$}}
\put(-39,-34){\tiny $10$}
\put(-52,-45){\rotatebox{180}{\tiny $4$}}
\put(4,-45){\rotatebox{180}{\tiny $2$}}
\put(-54,-70){\rotatebox{180}{\tiny $11$}}
\put(4,-70){\rotatebox{180}{\tiny $3$}}
\put(-37,-88){\tiny $5$}
\put(-30,-93){\rotatebox{180}{\tiny $14$}}
\put(4,-93){\rotatebox{180}{\tiny $0$}}
\put(-26,-117){\rotatebox{180}{\tiny $3$}}
\put(4,-117){\rotatebox{180}{\tiny $1$}}
\put(-37,-130){\tiny $4$}
\put(-54,-140.5){\rotatebox{180}{\tiny $16$}}
\put(4,-140.5){\rotatebox{180}{\tiny $14$}}
\put(-50,-165){\rotatebox{180}{\tiny $5$}}
\put(4,-165){\rotatebox{180}{\tiny $15$}}
\put(-39,-185){\tiny $17$}
\put(-26,-189.5){\rotatebox{180}{\tiny $8$}}
\put(4,-189.5){\rotatebox{180}{\tiny $12$}}
\put(-30,-213){\rotatebox{180}{\tiny $15$}}
\put(4,-213){\rotatebox{180}{\tiny $13$}}
\put(-39,-227){\tiny $16$}
\put(-54,-238){\rotatebox{180}{\tiny $10$}}
\put(4,-238){\rotatebox{180}{\tiny $8$}}
\put(-54,-262){\rotatebox{180}{\tiny $17$}}
\put(4,-262){\rotatebox{180}{\tiny $9$}}
\put(-39,-282){\tiny $11$}
\put(-14.5,-282){\tiny $15$}
\put(25,25){$\hat S_2$}
\end{picture}
\end{picture}

\special{
psfile=twolines18_1.eps
hscale=85
vscale=85
angle=90
hoffset=131
voffset=-359.5 
}
\begin{picture}(0,0)(-96,-86) 
\begin{picture}(0,0)(2.5,244)
   \put(-10,81){\textcolor{blue}{$d_2$}}
\put(5,66.5){\textcolor{red}{$a_{2,1}$}}
\put(5,50){\textcolor{red}{$a_{1,2}$}}
\put(5,18.5){\textcolor{black}{$c_{2,1}$}}
\put(5,-1){\textcolor{black}{$c_{1,2}$}}
\put(-19.5,18.5){\textcolor{mygreen}{$b_{2,2}$}}
\put(-19.55,-1){\textcolor{mygreen}{$b_{3,1}$}}
\end{picture}
\begin{picture}(0,0)(5,341)
   \put(-10,81){\textcolor{blue}{$d_3$}}
\put(5,66.5){\textcolor{red}{$a_{3,1}$}}
\put(5,50){\textcolor{red}{$a_{2,2}$}}
\put(5,18.5){\textcolor{black}{$c_{3,1}$}}
\put(5,-1){\textcolor{black}{$c_{2,2}$}}
\put(-19.5,18.5){\textcolor{mygreen}{$b_{3,2}$}}
\put(-19.55,-1){\textcolor{mygreen}{$b_{1,1}$}}
\end{picture}
\begin{picture}(0,0)(7.5,437)
   \put(-10,81){\textcolor{blue}{$d_1$}}
\put(5,66.5){\textcolor{red}{$a_{1,1}$}}
\put(5,50){\textcolor{red}{$a_{3,2}$}}
\put(5,18.5){\textcolor{black}{$c_{1,1}$}}
\put(5,-1){\textcolor{black}{$c_{3,2}$}}
\put(-19.5,18.5){\textcolor{mygreen}{$b_{1,2}$}}
\put(-19.55,-1){\textcolor{mygreen}{$b_{2,1}$}}
\end{picture}
\end{picture}
\special{ 
psfile=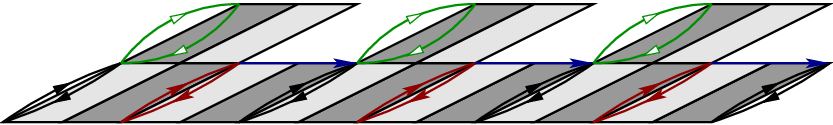
hscale=85
vscale=85
angle=90
hoffset=189
voffset=-341.5 
}
\begin{picture}(0,0)(-163.5,-120) 
\begin{picture}(0,0)(2.5,244)
   \put(-9,97){\textcolor{blue}{$d'_2$}}
\put(7,65){\textcolor{red}{$a'_{2,1}$}}
\put(3,30){\textcolor{red}{$a'_{1,2}$}}
\put(9,11){\textcolor{black}{$c'_{2,1}$}}
\put(2,-15){\textcolor{black}{$c'_{1,2}$}}
\put(-14.5,65){\textcolor{mygreen}{$b'_{2,2}$}}
\put(-23,30){\textcolor{mygreen}{$b'_{3,1}$}}
\end{picture}
\begin{picture}(0,0)(5,341)
   \put(-9,97){\textcolor{blue}{$d'_3$}}
\put(7,65){\textcolor{red}{$a'_{3,1}$}}
\put(3,30){\textcolor{red}{$a'_{2,2}$}}
\put(9,11){\textcolor{black}{$c'_{3,1}$}}
\put(2,-15){\textcolor{black}{$c'_{2,2}$}}
\put(-14.5,65){\textcolor{mygreen}{$b'_{3,2}$}}
\put(-22,30){\textcolor{mygreen}{$b'_{1,1}$}}
\end{picture}
\begin{picture}(0,0)(7.5,437)
   \put(-9,97){\textcolor{blue}{$d'_1$}}
\put(7,64){\textcolor{red}{$a'_{1,1}$}}
\put(3,30){\textcolor{red}{$a'_{3,2}$}}
\put(9,11){\textcolor{black}{$c'_{1,1}$}}
\put(2,-15){\textcolor{black}{$c'_{3,2}$}}
\put(-14.5,64){\textcolor{mygreen}{$b'_{1,2}$}}
\put(-22,30){\textcolor{mygreen}{$b'_{2,1}$}}
\end{picture}
\end{picture}
\special{ 
psfile=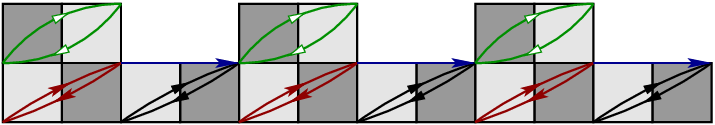
hscale=85
vscale=85
angle=90
hoffset=257.5
voffset=-294 
}
\begin{picture}(0,0)(-232,-120) 
\begin{picture}(0,0)(2.5,244)
   \put(-9,96){\textcolor{blue}{$d'_2$}}
\put(9,110){\textcolor{black}{$c'_{1,1}$}}
\put(2,84){\textcolor{black}{$c'_{3,2}$}}
\put(7,65){\textcolor{red}{$a_{2,1}$}}
\put(3,30){\textcolor{red}{$a_{1,2}$}}
\put(9,11){\textcolor{black}{$c_{2,1}$}}
\put(2,-15){\textcolor{black}{$c_{1,2}$}}
\put(-14.5,65){\textcolor{mygreen}{$b_{2,2}$}}
\put(-23,30){\textcolor{mygreen}{$b_{3,1}$}}
\end{picture}
\begin{picture}(0,0)(5,341)
   \put(-9,96){\textcolor{blue}{$d'_3$}}
\put(7,65){\textcolor{red}{$a'_{3,1}$}}
\put(3,30){\textcolor{red}{$a'_{2,2}$}}
\put(9,11){\textcolor{black}{$c'_{3,1}$}}
\put(2,-15){\textcolor{black}{$c'_{2,2}$}}
\put(-14.5,65){\textcolor{mygreen}{$b'_{3,2}$}}
\put(-22,30){\textcolor{mygreen}{$b'_{1,1}$}}
\end{picture}
\begin{picture}(0,0)(7.5,437)
   \put(-9,96){\textcolor{blue}{$d'_1$}}
\put(7,64){\textcolor{red}{$a'_{1,1}$}}
\put(3,30){\textcolor{red}{$a'_{3,2}$}}
\put(-14.5,64){\textcolor{mygreen}{$b'_{1,2}$}}
\put(-22,30){\textcolor{mygreen}{$b'_{2,1}$}}
\end{picture}
\end{picture}
\special{ 
psfile=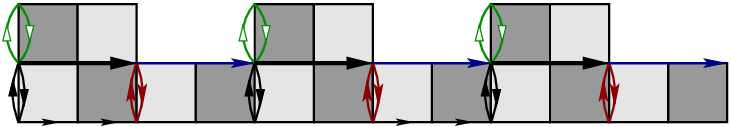
hscale=85
vscale=85
angle=90
hoffset=326 
voffset=-300.5 
}
\begin{picture}(0,0)(-302,-132.5) 
\begin{picture}(0,0)(2.5,244)
   \put(-10.5,82){\textcolor{blue}{$d_2$}}
   \put(-10,38){\textcolor{black}{$s_2$}}
\put(5,66.5){\textcolor{red}{$a_{2,1}$}}
\put(5,50){\textcolor{red}{$a_{1,2}$}}
\put(5,18.5){\textcolor{black}{$c_{2,1}$}}
\put(5,-1){\textcolor{black}{$c_{1,2}$}}
\put(-19.5,18.5){\textcolor{mygreen}{$b_{2,2}$}}
\put(-19.55,-1){\textcolor{mygreen}{$b_{3,1}$}}
\end{picture}
\begin{picture}(0,0)(5,341)
   \put(25, 94){\textcolor{black}{$9_b$}}
   \put(-10.5,82){\textcolor{blue}{$d_3$}}
   \put(25, 70){\textcolor{black}{$8_b$}}
   \put(-10,38){\textcolor{black}{$s_3$}}
\put(5,66.5){\textcolor{red}{$a_{3,1}$}}
\put(5,50){\textcolor{red}{$a_{2,2}$}}
\put(5,18.5){\textcolor{black}{$c_{3,1}$}}
\put(5,-1){\textcolor{black}{$c_{2,2}$}}
\put(-19.5,18.5){\textcolor{mygreen}{$b_{3,2}$}}
\put(-19.55,-1){\textcolor{mygreen}{$b_{1,1}$}}
\end{picture}
\begin{picture}(0,0)(7.5,437)
   \put(-10.5,82){\textcolor{blue}{$d_1$}}
   \put(-10,38){\textcolor{black}{$s_1$}}
   \put(25, 44){\textcolor{black}{$1_b$}}
\put(5,66.5){\textcolor{red}{$a_{1,1}$}}
\put(5,50){\textcolor{red}{$a_{3,2}$}}
   \put(25,20){\textcolor{black}{$0_b$}}
\put(5,18.5){\textcolor{black}{$c_{1,1}$}}
\put(5,-1){\textcolor{black}{$c_{3,2}$}}
\put(-19.5,18.5){\textcolor{mygreen}{$b_{1,2}$}}
\put(-19.55,-1){\textcolor{mygreen}{$b_{2,1}$}}
\end{picture}
\put(-38,-126){$\hat S_2$}
\end{picture}

\vspace*{350pt}

\begin{figure}[bht]
 %
 %
\caption{
\label{fig:twoline:shear}
Two-cylinder  surface  $\hat  S_2$  (left two pictures) is sheared by
$h^2$ (middle picture); then cut and reglued (the picture next to the
right) to fit finally the original surface $\hat S_2$ (the picture on
the right).
}
\end{figure}


\clearpage

\special{ 
psfile=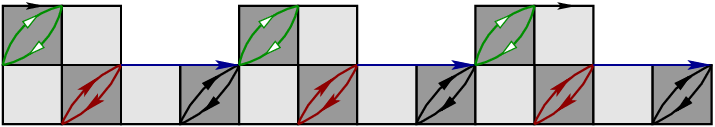
hscale=85
vscale=85
angle=90
hoffset=57.5
voffset=-302 
}
\begin{picture}(0,0)(-22.5,-121) 
\begin{picture}(0,0)(2.5,244)
\put(10.5,115.5){\textcolor{black}{$c_{1,1}$}}
\put(1,100.5){\textcolor{black}{$c_{3,2}$}}
\put(-10,95){\textcolor{blue}{$d_2$}}
\put(10,68){\textcolor{red}{$a_{2,1}$}}
   \put(-38,58){$17_t$}
   \put(-9.5,64){$s_2$}
\put(1,52){\textcolor{red}{$a_{1,2}$}}
\put(10.5,19.5){\textcolor{black}{$c_{2,1}$}}
\put(1.25,4){\textcolor{black}{$c_{1,2}$}}
\put(-14.5,46){\textcolor{mygreen}{$b_{2,2}$}}
\put(-23,25){\textcolor{mygreen}{$b_{3,1}$}}
\end{picture}
\begin{picture}(0,0)(5,341)
\put(10,68){\textcolor{red}{$a_{3,1}$}}
\put(1,52){\textcolor{red}{$a_{2,2}$}}
   \put(-9.5,64){$s_3$}
\put(-10,95){\textcolor{blue}{$d_3$}}
\put(10.5,19.5){\textcolor{black}{$c_{3,1}$}}
\put(1.25,4){\textcolor{black}{$c_{2,2}$}}
\put(-14.5,46){\textcolor{mygreen}{$b_{3,2}$}}
\put(-23,25){\textcolor{mygreen}{$b_{1,1}$}}
\end{picture}
\begin{picture}(0,0)(7.5,437)
\put(10,68){\textcolor{red}{$a_{1,1}$}}
\put(1,52){\textcolor{red}{$a_{3,2}$}}
   \put(-33.5,34){$4_t$}
   \put(-9.5,64){$s_1$}
\put(-10,95){\textcolor{blue}{$d_1$}}
\put(-14.5,46){\textcolor{mygreen}{$b_{1,2}$}}
\put(-23,25){\textcolor{mygreen}{$b_{2,1}$}}
\end{picture}
\put(34,-111){$\hat S_1$} 
\end{picture}

\special{ 
psfile=twolines18_0.eps
hscale=85
vscale=85
angle=90
hoffset=130.5 
voffset=-292 
}
\begin{picture}(0,0)(-99.5,-142) 
\begin{picture}(0,0)(0,97.5)
\put(-24.5,-98){17}
\put(-24.5,-122.5){16}
\put(0,-49){15}
\put(0,-73.5){14}
\put(0,-98){13}
\put(0,-122.5){12}
\end{picture}
\begin{picture}(0,0)(2.5,193)
\put(-24.5,-98){11}
\put(-24.5,-122.5){10}
\put(3,-49){9}
\put(3,-73.5){8}
\put(3,-98){7}
\put(3,-122.5){6}
\end{picture}
\begin{picture}(0,0)(5,290)
\put(-21.5,-98){5}
\put(-21.5,-122.5){4}
\put(3,-49){3}
\put(3,-73.5){2}
\put(3,-98){1}
\put(3,-122.5){0}
\end{picture}
\begin{picture}(0,0)(2.5,144.5)
\put(-16,12){\tiny\textit C}
\put(14.5,-12.75){\tiny\textit F}
\put(-16,-12.75){\tiny\textit A}
\put(-40,-37){\tiny\textit B}
\put(-16,-34){\tiny\textit C}
\put(14.5,-61){\tiny\textit D}
\put(-40,-61){\tiny\textit E}
\put(-40,-85){\tiny\textit B}
\put(-16,-88){\tiny\textit C}
\end{picture}
\begin{picture}(0,0)(5,241.5)
\put(14.5,-12.75){\tiny\textit F}
\put(-16,-12.75){\tiny\textit A}
\put(-40,-37){\tiny\textit B}
\put(-16,-34){\tiny\textit C}
\put(14.5,-61){\tiny\textit D}
\put(-40,-61){\tiny\textit E}
\put(-40,-85){\tiny\textit B}
\put(-16,-88){\tiny\textit C}
\end{picture}
\begin{picture}(0,0)(7.5,338)
\put(14.5,-12.75){\tiny\textit F}
\put(-16,-12.75){\tiny\textit A}
\put(-40,-37){\tiny\textit B}
\put(-16,-34){\tiny\textit C}
\put(14.5,-61){\tiny\textit D}
\put(-40,-61){\tiny\textit E}
\put(-40,-89){\tiny\textit B}
\put(-14,-89){\tiny\textit C}
\end{picture}
\begin{picture}(0,0)(-0.5,145)
\put(-12,15){\tiny $0$}
\put(-26,4.5){\rotatebox{180}{\tiny $2$}} 
\put(4,4.5){\rotatebox{180}{\tiny $8$}}
\put(-26,-19.5){\rotatebox{180}{\tiny $9$}}
\put(4,-19.5){\rotatebox{180}{\tiny $3$}}
\put(-39,-34){\tiny $10$}
\put(-52,-45){\rotatebox{180}{\tiny $4$}}
\put(4,-45){\rotatebox{180}{\tiny $6$}}
\put(-54,-70){\rotatebox{180}{\tiny $11$}}
\put(4,-70){\rotatebox{180}{\tiny $1$}}
\put(-37,-88){\tiny $5$}
\put(-30,-93){\rotatebox{180}{\tiny $14$}}
\put(4,-93){\rotatebox{180}{\tiny $2$}}
\put(-26,-117){\rotatebox{180}{\tiny $3$}}
\put(4,-117){\rotatebox{180}{\tiny $15$}}
\put(-37,-130){\tiny $4$}
\put(-54,-140.5){\rotatebox{180}{\tiny $16$}}
\put(4,-140.5){\rotatebox{180}{\tiny $0$}}
\put(-50,-165){\rotatebox{180}{\tiny $5$}}
\put(4,-165){\rotatebox{180}{\tiny $13$}}
\put(-39,-185){\tiny $17$}
\put(-26,-189.5){\rotatebox{180}{\tiny $8$}}
\put(4,-189.5){\rotatebox{180}{\tiny $14$}}
\put(-30,-213){\rotatebox{180}{\tiny $15$}}
\put(4,-213){\rotatebox{180}{\tiny $9$}}
\put(-39,-227){\tiny $16$}
\put(-54,-238){\rotatebox{180}{\tiny $10$}}
\put(4,-238){\rotatebox{180}{\tiny $12$}}
\put(-54,-262){\rotatebox{180}{\tiny $17$}}
\put(4,-262){\rotatebox{180}{\tiny $7$}}
\put(-39,-282){\tiny $11$}
\put(-14.5,-282){\tiny $15$}
\end{picture}
\end{picture}

\special{
psfile=twolines18_0.eps
hscale=85
vscale=85
angle=180
hoffset=324 
voffset=-332 
}
\begin{picture}(0,0)(-52,339) 
\put(-21.5,0.25){\rotatebox{90}{15}}
\put(2.5,0.25){\rotatebox{90}{14}}
\put(27,0.25){\rotatebox{90}{13}}
\put(49.5,0.25){\rotatebox{90}{12}}
\put(75,2.25){\rotatebox{90}{9}}
\put(99.5,2.25){\rotatebox{90}{8}}
\put(123,2.25){\rotatebox{90}{7}}
\put(146.5,2.25){\rotatebox{90}{6}}
\put(171,2.25){\rotatebox{90}{3}}
\put(197,2.25){\rotatebox{90}{2}}
\put(220,2.25){\rotatebox{90}{1}}
\put(244.5,2.25){\rotatebox{90}{0}}
\put(27,-24.25){\rotatebox{90}{17}}
\put(49.5,-24.25){\rotatebox{90}{16}}
\put(123,-24.25){\rotatebox{90}{11}}
\put(146.5,-24.25){\rotatebox{90}{10}}
\put(220,-21.25){\rotatebox{90}{5}}
\put(244.5,-21.25){\rotatebox{90}{4}}
\put(-20,21){\rotatebox{-90}{\tiny 8}}
\put(3.5,21){\rotatebox{-90}{\tiny 3}}
\put(28.5,21){\rotatebox{-90}{\tiny 6}}
\put(51,21){\rotatebox{-90}{\tiny 1}}
\put(76.5,21){\rotatebox{-90}{\tiny 2}}
\put(100.5,24){\rotatebox{-90}{\tiny 15}}
\put(124.5,21){\rotatebox{-90}{\tiny 0}}
\put(149,24){\rotatebox{-90}{\tiny 13}}
\put(172,24){\rotatebox{-90}{\tiny 14}}
\put(198,21){\rotatebox{-90}{\tiny 9}}
\put(221,24){\rotatebox{-90}{\tiny 12}}
\put(246,21){\rotatebox{-90}{\tiny 7}}
\put(-35.5,3){\rotatebox{90}{\tiny 0}}
\put(261,1){\rotatebox{90}{\tiny 15}}
\put(-21.5,-9){\rotatebox{-90}{\tiny 2}}
\put(3.5,-9){\rotatebox{-90}{\tiny 9}}
\put(76.5,-8.5){\rotatebox{-90}{\tiny 14}}
\put(100.5,-9){\rotatebox{-90}{\tiny 3}}
\put(172,-9){\rotatebox{-90}{\tiny 8}}
\put(198,-8.5){\rotatebox{-90}{\tiny 15}}
\put(13,-24){\rotatebox{90}{\tiny 10}}
\put(67.5,-22){\rotatebox{90}{\tiny 5}}
\put(110,-22){\rotatebox{90}{\tiny 4}}
\put(164,-24){\rotatebox{90}{\tiny 17}}
\put(206,-24){\rotatebox{90}{\tiny 16}}
\put(261,-24){\rotatebox{90}{\tiny 11}}
\put(28.5,-33){\rotatebox{-90}{\tiny 4}}
\put(51,-32.5){\rotatebox{-90}{\tiny 11}}
\put(124.5,-32.5){\rotatebox{-90}{\tiny 16}}
\put(149,-33){\rotatebox{-90}{\tiny 5}}
\put(221,-32.5){\rotatebox{-90}{\tiny 10}}
\put(246,-32.5){\rotatebox{-90}{\tiny 17}}
\begin{picture}(0,0)(-53,-3)
\put(-61,14.5){\tiny\textit F}
\put(-12.75,14.5){\tiny\textit D}
\put(-88,-16){\tiny\textit C}
\put(-61.5,-16){\tiny\textit A}
\put(-40,-16){\tiny\textit C}
\put(14,-16){\tiny\textit C}
\put(-39,-40){\tiny\textit B}
\put(-12.75,-40){\tiny\textit E}
\put(14,-40){\tiny\textit B}
\end{picture}
\begin{picture}(0,0)(-147.5,-3)
\put(-61,14.5){\tiny\textit F}
\put(-12.75,14.5){\tiny\textit D}
\put(-61.5,-16){\tiny\textit A}
\put(-40,-16){\tiny\textit C}
\put(14,-16){\tiny\textit C}
\put(-39,-40){\tiny\textit B}
\put(-12.75,-40){\tiny\textit E}
\put(13.5,-40){\tiny\textit B}
\end{picture}
\begin{picture}(0,0)(-242,-3)
\put(-61,14.5){\tiny\textit F}
\put(-12.75,14.5){\tiny\textit D}
\put(-61.5,-16){\tiny\textit A}
\put(-40,-16){\tiny\textit C}
\put(14,-13){\tiny\textit C}
\put(-39,-40){\tiny\textit B}
\put(-12.75,-40){\tiny\textit E}
\put(14,-40){\tiny\textit B}
\end{picture}
\end{picture}

\special{
psfile=twolines18_second_surface.eps
hscale=85
vscale=85
angle=90
angle=180
hoffset=324
voffset=-395 
}
\begin{picture}(0,0)(-43,400) 
\begin{picture}(0,0)(0,0)
\put(-26,9){\textcolor{black}{$c'_{1,1}$}}
\put(-7,-6){\textcolor{black}{$c'_{3,2}$}}
\put(-1,-18.5){\textcolor{blue}{$d'_2$}}
\put(21,9){\textcolor{red}{$a'_{2,1}$}}
\put(40.5,-6){\textcolor{red}{$a'_{1,2}$}}
\put(70.75,9){\textcolor{black}{$c'_{2,1}$}}
\put(89.5,-6){\textcolor{black}{$c'_{1,2}$}}
\put(42.5,-18.5){\textcolor{mygreen}{$b'_{2,2}$}}
\put(67,-30.5){\textcolor{mygreen}{$b'_{3,1}$}}
\end{picture}
\begin{picture}(0,0)(-94.5,0)
\put(-1,-18.5){\textcolor{blue}{$d'_3$}}
\put(21,9){\textcolor{red}{$a'_{3,1}$}}
\put(40.25,-6){\textcolor{red}{$a'_{2,2}$}}
\put(70.75,9){\textcolor{black}{$c'_{3,1}$}}
\put(89.5,-6){\textcolor{black}{$c'_{2,2}$}}
\put(42.5,-18.5){\textcolor{mygreen}{$b'_{3,2}$}}
\put(67,-30.5){\textcolor{mygreen}{$b'_{1,1}$}}
\end{picture}
\begin{picture}(0,0)(-189,0)
\put(-1,-18.5){\textcolor{blue}{$d'_1$}}
\put(20.75,9){\textcolor{red}{$a'_{1,1}$}}
\put(40,-6){\textcolor{red}{$a'_{3,2}$}}
\put(42,-18.5){\textcolor{mygreen}{$b'_{1,2}$}}
\put(66,-30.5){\textcolor{mygreen}{$b'_{2,1}$}}
\end{picture}
\end{picture}

\special{ 
psfile=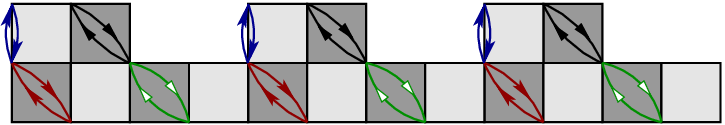
hscale=85
vscale=85
angle=90
hoffset=276.5 
voffset=-259.5 
}
\begin{picture}(0,0)(-242,-165.7) 
\begin{picture}(0,0)(2.5,244)
\put(1,99.5){\textcolor{mygreen}{$b'_{1,1}$}}
\put(10,74){\textcolor{mygreen}{$b'_{3,2}$}}
\put(-4.5,47){\textcolor{red}{$a'_{2,1}$}}
\put(10,30){\textcolor{red}{$a'_{1,2}$}}
\put(-23.5,74){\textcolor{black}{$c'_{2,2}$}}
\put(-9,60){\textcolor{black}{$c'_{3,1}$}}
\put(-20,35){\textcolor{blue}{$d'_{3,2}$}}
\put(-20,18.5){\textcolor{blue}{$d'_{2,1}$}}
\end{picture}
\begin{picture}(0,0)(5,341)
\put(1,99.5){\textcolor{mygreen}{$b'_{1,1}$}}
\put(10,74){\textcolor{mygreen}{$b'_{3,2}$}}
\put(-4.5,47){\textcolor{red}{$a'_{3,1}$}}
\put(10,30){\textcolor{red}{$a'_{2,2}$}}
\put(-23.5,74){\textcolor{black}{$c'_{3,2}$}}
\put(-9,60){\textcolor{black}{$c'_{1,1}$}}
\put(-20,35){\textcolor{blue}{$d'_{1,2}$}}
\put(-20,18.5){\textcolor{blue}{$d'_{3,1}$}}
\end{picture}
\begin{picture}(0,0)(7.5,437)
\put(1,99.5){\textcolor{mygreen}{$b'_{1,1}$}}
\put(10,74){\textcolor{mygreen}{$b'_{3,2}$}}
\put(-4.5,47){\textcolor{red}{$a'_{1,1}$}}
\put(10,30){\textcolor{red}{$a'_{3,2}$}}
\put(-23.5,74){\textcolor{black}{$c'_{1,2}$}}
\put(-9,60){\textcolor{black}{$c'_{2,1}$}}
\put(-20,35){\textcolor{blue}{$d'_{2,2}$}}
\put(-20,18.5){\textcolor{blue}{$d'_{1,1}$}}
\end{picture}
\end{picture}

\special{ 
psfile=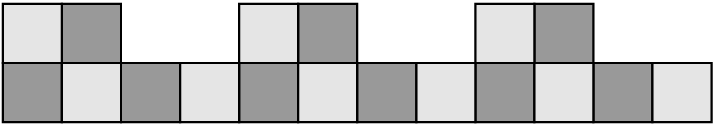
hscale=85
vscale=85
angle=90
hoffset=350  
voffset=-244 
}
\begin{picture}(0,0)(-320,-189.75) 
\begin{picture}(0,0)(-2.5,98) 
\put(-24.5,-92){\rotatebox{-90}{3}}
\put(-24.5,-123){\rotatebox{90}{14}}
\put(0,-47){\rotatebox{90}{5}}
\put(0,-64){\rotatebox{-90}{10}}
\put(0,-92){\rotatebox{-90}{6}}
\put(0,-122.5){\rotatebox{90}{13}}
\end{picture}
\begin{picture}(0,0)(0,196)
\put(-24.5,-87){\rotatebox{-90}{15}}
\put(-24.5,-119){\rotatebox{90}{8}}
\put(0,-49){\rotatebox{90}{17}}
\put(0,-66){\rotatebox{-90}{4}}
\put(0,-89.5){\rotatebox{-90}{0}}
\put(0,-119){\rotatebox{90}{7}}
\end{picture}
\begin{picture}(0,0)(2.5,290.5)
\put(-24.5,-92){\rotatebox{-90}{9}}
\put(-24.5,-121){\rotatebox{90}{2}}
\put(0,-51){\rotatebox{90}{11}}
\put(0,-65){\rotatebox{-90}{16}}
\put(0,-90){\rotatebox{-90}{12}}
\put(0,-121){\rotatebox{90}{1}}
\end{picture}
\begin{picture}(0,0)(2.5,144.5)
\put(-16,12){\tiny\textit C}
\put(-16,-12.75){\tiny\textit B}
\put(14.5,-12.75){\tiny\textit E}
\put(-40,-37){\tiny\textit A}
\put(-16,-34){\tiny\textit C}
\put(-40,-61){\tiny\textit F}
\put(14.5,-61){\tiny\textit D}
\put(-40,-85){\tiny\textit A}
\put(-16,-88.5){\tiny\textit C}
\end{picture}
\begin{picture}(0,0)(5,241.5)
\put(-16,-12.75){\tiny\textit B}
\put(14.5,-12.75){\tiny\textit E}
\put(-40,-37){\tiny\textit A}
\put(-16,-34){\tiny\textit C}
\put(-40,-61){\tiny\textit F}
\put(14.5,-61){\tiny\textit D}
\put(-40,-85){\tiny\textit A}
\put(-16,-88){\tiny\textit C}
\end{picture}
\begin{picture}(0,0)(7.5,338)
\put(-16,-12.75){\tiny\textit B}
\put(14.5,-12.75){\tiny\textit E}
\put(-40,-37){\tiny\textit A}
\put(-16,-34){\tiny\textit C}
\put(-40,-61){\tiny\textit F}
\put(14.5,-61){\tiny\textit D}
\put(-40,-89){\tiny\textit A}
\put(-14,-89){\tiny\textit C}
\end{picture}
\begin{picture}(0,0)(-0.25,145)
\put(-12,15){\rotatebox{90}{\tiny $1$}}
\put(-26,-1.5){\rotatebox{90}{\tiny $16$}} 
\put(4,0.5){\rotatebox{90}{\tiny $4$}}
\put(-26,-19.5){\rotatebox{-90}{\tiny $17$}}
\put(4,-19.5){\rotatebox{-90}{\tiny $11$}}
\put(-37,-34){\rotatebox{90}{\tiny $8$}}
\put(-51,-45){\rotatebox{-90}{\tiny $2$}}
\put(4,-45){\rotatebox{-90}{\tiny $7$}}
\put(-51,-75){\rotatebox{90}{\tiny $15$}}
\put(4,-75){\rotatebox{90}{\tiny $12$}}
\put(-37,-84){\rotatebox{-90}{\tiny $9$}}
\put(-26,-99){\rotatebox{90}{\tiny $10$}}
\put(4,-99){\rotatebox{90}{\tiny $16$}}
\put(-26,-115){\rotatebox{-90}{\tiny $11$}}
\put(4,-117){\rotatebox{-90}{\tiny $5$}}
\put(-37,-130){\rotatebox{90}{\tiny $2$}}
\put(-51,-140){\rotatebox{-90}{\tiny $14$}}
\put(4,-141.5){\rotatebox{-90}{\tiny $1$}}
\put(-51,-170){\rotatebox{90}{\tiny $9$}}
\put(4,-170){\rotatebox{90}{\tiny $6$}}
\put(-37,-181){\rotatebox{-90}{\tiny $3$}}
\put(-26,-194.5){\rotatebox{90}{\tiny $4$}}
\put(4,-196.5){\rotatebox{90}{\tiny $10$}}
\put(-26,-215){\rotatebox{-90}{\tiny $5$}}
\put(4,-213){\rotatebox{-90}{\tiny $17$}}
\put(-37,-227){\rotatebox{90}{\tiny $14$}}
\put(-51,-238){\rotatebox{-90}{\tiny $8$}}
\put(4,-238){\rotatebox{-90}{\tiny $13$}}
\put(-51,-267){\rotatebox{90}{\tiny $3$}}
\put(4,-267){\rotatebox{90}{\tiny $0$}}
\put(-37,-277){\rotatebox{-90}{\tiny $15$}}
\put(-11,-282){\rotatebox{90}{\tiny $5$}}
\put(-50,24){$\hat S_1$}
\end{picture}
\end{picture}

\vspace*{405pt}

\begin{figure}[bht]
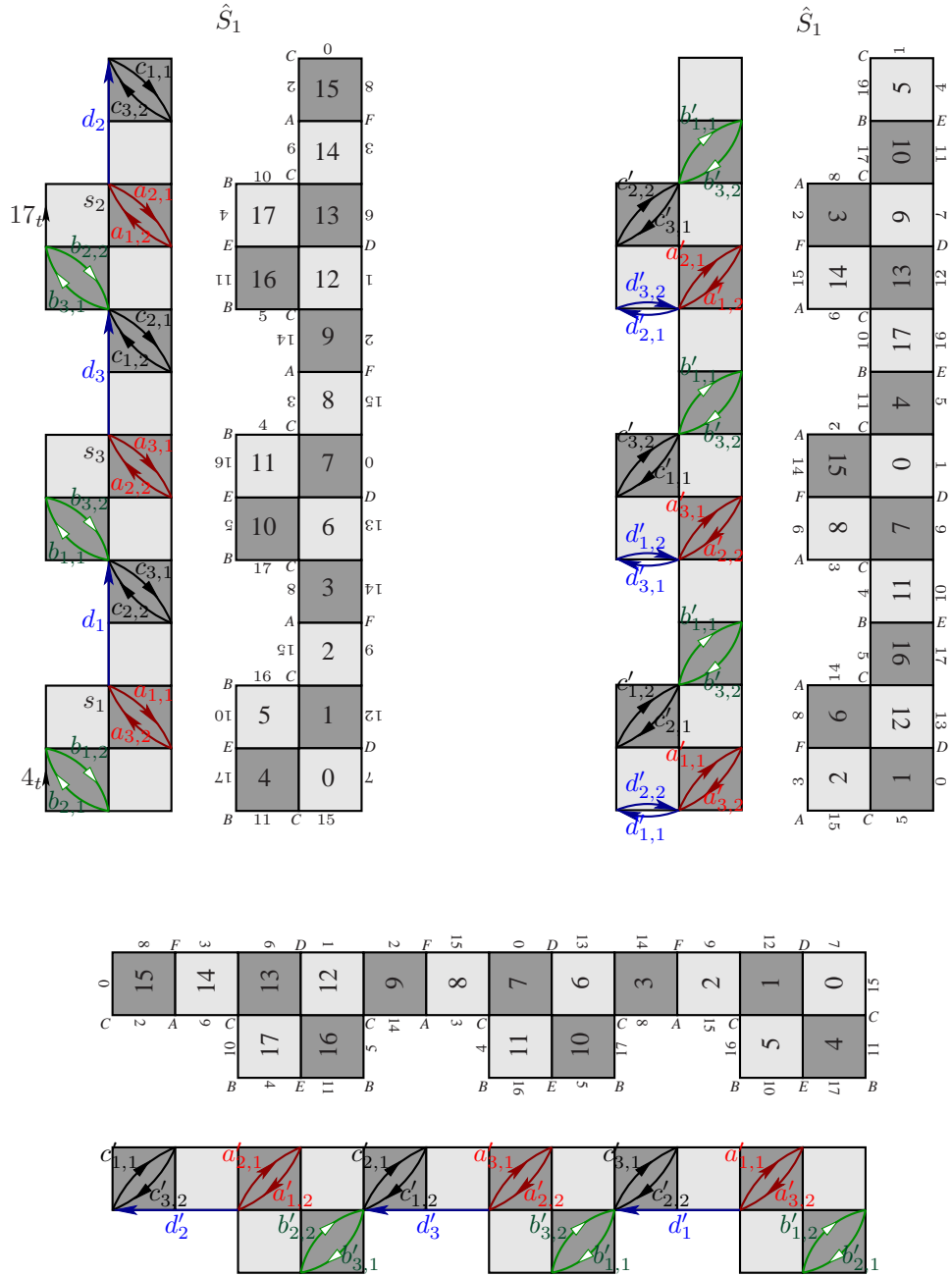

 %
 %
\caption{
\label{fig:twoline:rotate}
Rotation  $r_1:\hat  S_1\to \hat S_1$. The two-cylinder surface
$\hat S_1$ (left two pictures) is rotated by $\pi/2$ counterclockwise
(bottom  two  pictures)  then  cut  and  reglued  to  fit the initial
pattern. }
\end{figure}


\end{document}